%% file: Triple_Trans.tex
\documentclass{amsart}
\usepackage{amssymb,latexsym,amsmath,amscd,graphicx,graphics,epic,eepic,bm,color,array,mathrsfs,fullpage}
\usepackage{enumerate}
\usepackage[all,knot,poly]{xy}

\newcommand{\zed}{\mathbb{Z}}

\newcommand{\Q}{\mathbb{Q}}
\newcommand{\fH}{\mathcal{H}}
\newcommand{\fC}{\mathcal{C}}
\newcommand{\fP}{\mathcal{P}}

\newcommand{\Hom}{\mathrm{Hom}}

\newcommand{\Sym}{\mathrm{Sym}}

\newcommand{\ve}{\varepsilon}

\newcommand{\id}{\mathrm{id}}

\newcommand{\gdim}{\mathrm{gdim}}
\newcommand{\mf}{\mathrm{mf}}

\newcommand{\hmf}{\mathrm{hmf}}

\newcommand{\slmf}{\mathfrak{sl}}

\theoremstyle{plain}
\newtheorem{theorem}{Theorem}[section]
\newtheorem{lemma}[theorem]{Lemma}
\newtheorem{proposition}[theorem]{Proposition}
\newtheorem{corollary}[theorem]{Corollary}

\newtheorem{question}[theorem]{Question}

\theoremstyle{definition}
\newtheorem{definition}[theorem]{Definition}

\theoremstyle{remark}
\newtheorem{remark}[theorem]{Remark}

\numberwithin{equation}{section}

\begin{document}

\title{A Family of Transverse Link Homologies}

\author{Hao Wu}

\thanks{The author was partially supported by NSF grant DMS-1205879 and a Collaboration Grant for Mathematicians from the Simons Foundation.}

\address{Department of Mathematics, The George Washington University, Monroe Hall, Room 240, 2115 G Street, NW, Washington DC 20052, USA. Telephone: 1-202-994-0653, Fax: 1-202-994-6760}

\email{haowu@gwu.edu}



\begin{abstract}
We define a homology $\fH_N$ for closed braids by applying Khovanov and Rozansky's matrix factorization construction with potential $ax^{N+1}$. Up to a grading shift, $\fH_0$ is the HOMFLYPT homology defined in \cite{KR2}. We demonstrate that, for $N\geq 1$, $\fH_N$ is a $\zed_2\oplus\zed^{\oplus 3}$-graded $\Q[a]$-module that is invariant under transverse Markov moves, but not under negative stabilization/de-stabilization. Thus, for $N\geq 1$, this homology is an invariant for transverse links in the standard contact $S^3$, but not for smooth links. We also discuss the decategorification of $\fH_N$ and the relation between $\fH_N$ and the $\slmf(N)$ Khovanov-Rozansky homology defined in \cite{KR1}.
\end{abstract}

\maketitle

\section{Introduction}\label{sec-intro}

\subsection{Transverse links in the standard contact $S^3$}

A contact structure $\xi$ on an oriented $3$-manifold $M$ is an oriented tangent plane distribution such that there is a $1$-form $\alpha$ on $M$ satisfying $\xi=\ker\alpha$, $d\alpha|_{\xi}>0$ and $\alpha\wedge d\alpha>0$. Such a $1$-form is called a contact form for $\xi$. The standard contact structure $\xi_{st}$ on $S^3$ is given by the contact form $\alpha_{st} = dz-ydx+xdy=dz+r^2d\theta$. 

We say that an oriented smooth link $L$ in $S^3$ is transverse if $\alpha_{st}|_L>0$. Two transverse links are said to be transverse isotopic if there is an isotopy from one to the other through transverse links. In \cite{Ben}, Bennequin proved that every transverse link is transverse isotopic to a counterclockwise transverse closed braid around the $z$-axis. Clearly, any smooth counterclockwise closed braid around the $z$-axis can be smoothly isotoped into a transverse closed braid around the $z$-axis without changing the braid word. In the rest of this paper, all closed braids are counterclockwise and around the $z$-axis.

Recall that two closed braids represent the same smooth link if and only if one of them can be changed into the other by a finite sequence of Markov moves, which are:
\begin{itemize}
    \item Braid group relations generated by
    \begin{itemize}
	  \item $\sigma_i\sigma_i^{-1}=\sigma_i^{-1}\sigma_i=\emptyset$,
	  \item $\sigma_i\sigma_j=\sigma_j\sigma_i$, when $|i-j|>1$,
	  \item $\sigma_i\sigma_{i+1}\sigma_i=\sigma_{i+1}\sigma_i\sigma_{i+1}$.
    \end{itemize}
    \item Conjugations: $\mu\leftrightsquigarrow\eta^{-1}\mu\eta$,
    where $\mu,~\eta\in \mathbf{B}_m$.
    \item Stabilizations and destabilizations:
    \begin{itemize}
	  \item positive: $\mu~(\in \mathbf{B}_m)\leftrightsquigarrow \mu\sigma_m~(\in \mathbf{B}_{m+1})$,
	  \item negative: $\mu~(\in \mathbf{B}_m)\leftrightsquigarrow \mu\sigma_m^{-1}~(\in \mathbf{B}_{m+1})$.
    \end{itemize}
\end{itemize}
In the above, $\mathbf{B}_m$ is the braid group on $m$ strands.

The following theorem by Orevkov, Shevchishin \cite{OSh} and Wrinkle \cite{Wr} describes when two transverse closed braids are transverse isotopic. 

\begin{theorem}\cite{OSh,Wr}\label{transversal-markov}
Two transverse closed braids are transverse isotopic if and only if the braid word of one of them can be changed into that of the other by a finite sequence of braid group relations, conjugations and positive stabilizations and destabilizations.
\end{theorem}

From now on, braid group relations, conjugations and positive stabilizations and destabilizations will be called transverse Markov moves. Theorem \ref{transversal-markov} tells us that there is a one-to-one correspondence 
\[
\{\text{Transverse isotopy classes of transverse links}\} \longleftrightarrow \{\text{Closed braids modulo transverse Markov moves}\}.
\] 

Note that $\xi_{st}$ admits a nowhere vanishing basis $\{\partial_x + y\partial_z, \partial_y-x\partial_z\}$. For each transverse link $L$, this basis induces a trivialization of the normal bundle of $L$ in $S^3$, that is, a framing of $L$. We call this framing the contact framing of $L$. With its contact framing, any transverse link is also a framed link. It is easy to see that, if two transverse links are transverse isotopic, then they are isotopic as framed links. It is possible for two transverse links to be non-isotopic as transverse links, but still isotopic as framed links. If a smooth link type contains two transverse links that are isotopic as framed links but not as transverse links, then we call this smooth link type ``transverse non-simple". An invariant for transverse links is called classical if it depends only on the framed link type of the transverse link. Otherwise, it is called non-classical or effective (in the sense that it is effective in detecting transverse non-simplicity.) 

See, for example, \cite{Birman-Menasco,Etnyre-contact-notes-1,Etnyre-Honda-cabling,Ng-contact-homology} for more about transverse links and their invariants.

\subsection{The Khovanov-Rozansky homology} In \cite{KR1}, Khovanov and Rozansky introduced an approach to construct link homologies using matrix factorizations, which consists of the following steps:
\begin{enumerate}
	\item Choose a base ring $R$ and a potential polynomial $p(x) \in R[x]$.
	\item Define matrix factorizations associated to MOY graphs using this potential $p(x)$.
	\item Define chain complexes of matrix factorizations associated to link diagrams using the crossing information.
\end{enumerate}

This approach has been carried out for the following potential polynomials:
\begin{itemize}
	\item $x^{N+1}\in \Q[x]$, which gives the $\slmf(N)$ Khovanov-Rozansky homology in \cite{KR1};
	\item $ax \in \Q[a,x]$, which gives the HOMFLYPT homology in \cite{KR2};
	\item $x^{N+1} + \sum_{l=1}^N \lambda_l x^l \in \Q[x]$, which gives the deformed $\slmf(N)$ Khovanov-Rozansky homology in \cite{Gornik,Wu7};
	\item $x^{N+1} + \sum_{l=1}^N a_l x^l \in \Q[a_1,\dots,a_N,x]$, which gives the equivariant $\slmf(N)$ Khovanov-Rozansky homology in \cite{Krasner}.
\end{itemize}

Among these link homologies, the HOMFLYPT homology appears somewhat different. All the other homologies are invariant under all Reidemeister moves. Therefore, these homologies of a given smooth link can be computed from any diagram of this link. But the HOMFLYPT homology is only invariant under braid-like Reidemeister moves. So the HOMFLYPT homology of a smooth link can only be computed from its braid diagrams.

In \cite{KR2}, Khovanov and Rozansky proposed to study the homology defined by the potential polynomial $\sum_{l=1}^{N+1} a_l x^l \in \Q[a_1,\dots,a_N,a_{N+1},x]$, which generalize the HOMFLYPT homology.

\subsection{A family of transverse link homologies} By Theorem \ref{transversal-markov}, one can construct an invariant for transverse links by constructing an invariant for closed braids that is invariant under transverse Markov moves. This is what we will do in the current paper.

More precisely, we will generalize the construction in \cite{KR2} to a matrix factorization construction with the potential polynomial $p(x)=ax^{N+1} \in \Q[a,x]$. We will:
\begin{enumerate}
	\item work with the potential polynomial $ax^{N+1} \in \Q[a,x]$,
	\item define matrix factorizations associated to MOY graphs in Definition \ref{def-MOY-mf},
	\item define chain complexes of matrix factorizations associated to link diagrams in Definition \ref{def-chain-tangle}.
\end{enumerate}

For each $N \geq0$, this construction gives a $\zed_2\oplus\zed^{\oplus 3}$-graded homology $\fH_N$. Of course, when $N=0$, the homology we get is just the HOMFLYPT homology with a grading shift. It turns out that, when $N\geq 1$, $\fH_N$ is only invariant under positive Reidemeister move I and braid-like Reidemeister moves II and III. So it is not a smooth link invariant. But, when we restrict to closed braids, this homology is invariant under transverse Markov moves. Thus, by Theorem \ref{transversal-markov}, it is a transverse link invariant. The following is our main result.

\begin{theorem}\label{thm-trans-link-homology}
Suppose $N\geq 1$. Let $B$ be a closed braid and $(\fC_N(B),d_{mf},d_\chi)$ the chain complex of matrix factorizations associated to $B$ defined in Definition \ref{def-chain-tangle}. Then the homotopy type of $\fC_N(B)$ does not change under transverse Markov moves. Moreover, the homotopy equivalences induced by transverse Markov moves preserve the $\zed_2\oplus \zed^{\oplus 3}$-grading of $\fC_N(B)$, where the $\zed_2$-grading is the $\zed_2$-grading of the underlying matrix factorization and the three $\zed$-gradings are the homological, $a$- and $x$-gradings of $\fC_N(B)$.

Consequently, for the homology $\fH_N(B)=H(H(\fC_N(B),d_{mf}),d_\chi)$ of $\fC_N(B)$ defined in Definition \ref{def-homology-braid}, every transverse Markov move on $B$ induces an isomorphism of $\fH_N(B)$ preserving the $\zed_2\oplus \zed^{\oplus 3}$-grading of $\fH_N(B)$ inherited from $\fC_N(B)$.
\end{theorem}

Theorem \ref{thm-trans-link-homology} follows directly from Propositions \ref{prop-RI+inv}, \ref{prop-RIIinv} and \ref{prop-RIIIinv}. The proofs of the invariances under braid-like Reidemeister moves II and III in Propositions \ref{prop-RIIinv} and \ref{prop-RIIIinv} are fairly similar to that in \cite{KR1,KR2}. But the proof of the invariance under positive Reidemeister move I in Proposition \ref{prop-RI+inv} is quite different from that in \cite{KR1,KR2}. This is mainly because we need to handle matrix factorizations that are not homotopically finite. 

\begin{question}
Is $\fH_N(B)$ a classical or a non-classical invariant for transverse links?
\end{question}

\subsection{Negative stabilization}\label{subsec-intro-neg-stab} Next we describe how $\fH_N$ changes under negative stabilizations and demonstrate by a simple example that $\fH_N$ is not invariant under negative stabilizations. Theorem \ref{thm-neg-stabilization} and Corollary \ref{cor-unknot-neg-stab} in this subsection will be proved in Subsection \ref{subsec-neg-stab} below. We will use the following notations in our statements.

\begin{definition}\label{def-star-notation}
For a $\zed_2 \oplus \zed^{\oplus 3}$-graded space $H$ with a $\zed_2$-grading, a homological grading, an $a$-grading and an $x$-grading. We denote by $H^{\ve,i,j,k}$ the subspace of $H$ of homogeneous elements of $\zed_2$-degree $\ve$, homological degree $i$, $a$-degree $j$ and $x$-degree $k$. 

Replacing one of these indices by a ``$\star$" means direct summing over all possible values of this index. For example:
\begin{eqnarray*}
H^{\ve,i,\star,k} & = & \bigoplus_{j\in\zed} H^{\ve,i,j,k}, \\
H^{\ve,i,\star,\star} & = & \bigoplus_{(j,k)\in\zed^{\oplus 2}} H^{\ve,i,j,k}.
\end{eqnarray*}

Moreover, we denote by $H\{q,r\}$ the $\zed_2 \oplus \zed^{\oplus 3}$-graded space obtained from $H$ by shifting the $a$-grading by $q$ and the $x$-grading by $r$. That is, $(H\{q,r\})^{\ve,i,j,k} = H^{\ve,i,j-q,k-r}$. We also use the notation $H^{\ve,i,\star,\star}\{q,r\}:=(H\{q,r\})^{\ve,i,\star,\star}$.
\end{definition}

To state Theorem \ref{thm-neg-stabilization}, we need to introduce a homomorphism $\pi_0$. Note that, for any $\zed_2\oplus \zed^{\oplus 2}$-graded matrix factorization $M$ of $0$ over $\Q[a]$, $M/aM$ is a $\zed_2\oplus \zed^{\oplus 2}$-graded matrix factorization of $0$ over $\Q$. Denote by $\pi_0: M \rightarrow M/aM$ the standard quotient map, which induces a homomorphism $\pi_0: H(M) \rightarrow H(M/aM)$ of homology of matrix factorizations. For a chain complex $(C,d)$ of $\zed_2\oplus \zed^{\oplus 2}$-graded matrix factorizations of $0$ over $\Q[a]$, this further induces a homomorphism $\pi_0: H(H(C,d_{mf}),d) \rightarrow H(H(C/aC,d_{mf}),d)$. Generally speaking, these induced homomorphisms are no longer quotient maps.

\begin{theorem}\label{thm-neg-stabilization}
Let $L$ be a transverse closed braid, and $L_-$ a transverse closed braid obtained from $L$ by a single negative stabilization. Then the chain complex $(H(\fC_N(L_-),d_{mf}),d_\chi)$ is isomorphic to the total chain complex of
\[
0\rightarrow \underbrace{(H(\fC_N(L),d_{mf}),d_\chi)\{-2,0\}}_{0} \xrightarrow{\pi_0} \underbrace{(H(\fC_N(L)/a\fC_N(L),d_{mf}),d_\chi)\{-2,0\}}_{1} \rightarrow 0,
\]
where the underbraces indicate shifts of the homological grading. This isomorphism preserves the $\zed_2 \oplus \zed^{\oplus 3}$-grading. In particular, there is a long exact sequence
{\footnotesize
\[
\cdots \rightarrow \fH_N^{\ve,i-1,\star,\star}(L)\{-2,0\} \xrightarrow{\pi_0} \mathscr{H}_N^{\ve,i-1,\star,\star}(L)\{-2,0\} \rightarrow \fH_N^{\ve,i,\star,\star}(L_-) \rightarrow \fH_N^{\ve,i,\star,\star}(L)\{-2,0\} \xrightarrow{\pi_0} \mathscr{H}_N^{\ve,i,\star,\star}(L)\{-2,0\} \rightarrow \cdots
\]}

\noindent preserving the $a$- and $x$-gradings, where $\mathscr{H}_N(L) := H(H(\fC_N(L)/a\fC_N(L),d_{mf}),d_\chi)$.
\end{theorem}

The following corollary shows that, for $N \geq 1$, $\fH_N$ is not invariant under negative stabilizations.

\begin{corollary}\label{cor-unknot-neg-stab}
Let $U$ be the transverse unknot represented by the $1$-strand braid, and $U_-$ the transverse unknot obtained from $U$ by a single negative stabilization. Then, for $N\geq 1$, as $\zed^{\oplus2}$-graded $\Q[a]$-modules,
\begin{eqnarray*}
\fH_N^{\ve,i,\star,\star}(U) & \cong & {\begin{cases}
(\bigoplus_{l=0}^{N-1} \Q[a]\{-1,-N+1+2l\}) \oplus (\bigoplus_{m=0}^{\infty} \Q[a]/(a) \{-1,N+1+2m\}) & \text{if } \ve=1 \text{ and  } i=0,\\
0 & \text{otherwise,} 
\end{cases}} \\
\fH_N^{\ve,i,\star,\star}(U_-) & \cong & {\begin{cases}
\bigoplus_{l=0}^{N-1} \Q[a]\{-1,-N+1+2l\} & \text{if } \ve=1 \text{ and  } i=0,\\
\bigoplus_{m=0}^{\infty} \Q[a]/(a) \{-2,2m\} & \text{if } \ve=0 \text{ and  } i=1,\\
0 & \text{otherwise.} 
\end{cases}}
\end{eqnarray*}
\end{corollary}

\subsection{Decategorification}

\begin{definition}\label{def-decat}
We define the decategorification $\fP_N$ of $\fH_N$ by
\[
\fP_N(B) := \sum_{(\ve,i,j,k)\in \zed_2 \oplus \zed^{\oplus3}}  (-1)^i \tau^\ve \alpha^j \xi^k \dim_\Q\fH_N^{\ve,i,j,k}(B) \in \zed[[\alpha,\xi]][\alpha^{-1},\xi^{-1},\tau]/(\tau^2-1)
\]
for any closed braid $B$.\footnote{From the definition of $\fC_N(B)$, one can see that its homological grading is bounded and its $a$- and $x$-gradings are bounded below.}
\end{definition}

Next, we describe $\fP_N$ by a skein definition that is very similar to the classical HOMFLYPT skein relation.

\begin{theorem}\label{thm-decat}
\begin{enumerate}[1.]
	\item $\fP_N$ is invariant under transverse Markov moves.
	\item $\alpha^{-1}\xi^{-N}\fP_N(\input{crossing+inline}) - \alpha\xi^{N}\fP_N(\input{crossing-inline})= \tau(\xi^{-1}-\xi)\fP_N(\input{arcs-inline})$. \vspace{5pt}
	\item $\fP_N(U^{\sqcup m}) = (\tau \alpha^{-1} [N])^m ( \frac{1}{1-\alpha^2} + \frac{\left(\frac{\tau \alpha \xi^{-1} + \xi^{-N}}{\xi^{-N}-\xi^N}\right)^m-1}{\tau\alpha\xi^{-N-1} +1})$, where $U^{\sqcup m}$ is the $m$-strand closed braid with no crossings and $[N]:= \frac{\xi^{-N}-\xi^N}{\xi^{-1} -\xi}$.
	\item Parts 1--3 above uniquely determine the value of $\fP_N$ on every closed braid.
\end{enumerate}
\end{theorem}

Theorem \ref{thm-decat} will be proved in Subsection \ref{subsec-skein} below. In a nutshell, Parts 1--3 of this theorem follow from the definition of $\fH_N$ and Theorems \ref{thm-trans-link-homology}, \ref{thm-neg-stabilization}, while Part 4 follows from the ``invariant computation tree" constructed by Franks and Williams in \cite{FW}. 

Although Theorem \ref{thm-decat} looks like the classical HOMFLYPT skein relation, there are two differences between $\fP_N$ and the HOMFLYPT polynomial:
\begin{enumerate}
	\item With an appropriate normalization, the HOMFLYPT polynomial is invariant under all Reidemeister moves. But, by Corollary \ref{cor-unknot-neg-stab}, $\fP_N$ is not invariant under negative stabilizations, which also implies that $\fP_N$ is not invariant under the non-braid-like Reidemeister move II.
	\item With an appropriate normalization, the HOMFLYPT polynomial is multiplicative under disjoint union of link diagrams. But, by Part 3 of Theorem \ref{thm-decat}, $\fP_N$ is not.
\end{enumerate}

It is not clear if one can resolve these differences by a re-normalization of the HOMFLYPT polynomial.

\begin{question}\label{que-decat-classical}
Is $\fP_N$ a classical or a non-classical invariant for transverse links?
\end{question}

$\fP_N$ does not seem to detect transverse non-simplicity from flype moves. We demonstrate this with some examples from \cite{Birman-Menasco,Ng-contact-homology}.

\begin{corollary}\label{cor-flype}
\begin{enumerate}[1.]
	\item Let $B_1$ be the closed $3$-braid $\sigma_1^{2p+1}\sigma_2^{2r}\sigma_1^{2q}\sigma_2^{-1}$ and $B_2$ be the closed $3$-braid $\sigma_1^{2p+1}\sigma_2^{-1}\sigma_1^{2q}\sigma_2^{2r}$. Then $\fP_N(B_1) = \fP_N(B_2)$.
	\item Let $B_3$ be the closed $4$-braid $\sigma_1\sigma_2^{-1}\sigma_1\sigma_2^{-1}\sigma_3^3\sigma_2\sigma_3^{-1}$ and $B_4$ the closed $4$-braid $\sigma_1\sigma_2^{-1}\sigma_1\sigma_2^{-1}\sigma_3^{-1}\sigma_2\sigma_3^{3}$. Then $\fP_N(B_3) = \fP_N(B_4)$.
\end{enumerate}
\end{corollary}

Corollary \ref{cor-flype} will be proved in Subsection \ref{subsec-flype} below. $B_1$ and $B_2$ (resp. $B_3$ and $B_4$) are related a flype move defined in \cite{Birman-Menasco}. So they are isotopic as smooth knots and have the same contact framing. In \cite{Birman-Menasco}, Birman and Menasco proved that, for an infinite family of $(p,q,r)$, $B_1$ and $B_2$ are not isotopic as transverse knots. In \cite{Ng-contact-homology}, Ng proved that $B_3$ and $B_4$ are not isotopic as transverse knots. Corollary \ref{cor-flype} shows that $\fP_N$ does not distinguish between of the transverse knots $B_1$ and $B_2$ (resp. $B_3$ and $B_4$.)

\subsection{Relation to the $\slmf(N)$ Khovanov-Rozansky homology} Denote by $H_N$ the $\slmf(N)$ Khovanov-Rozansky homology defined in \cite{KR1}. $H_N$ is a $\zed_2\oplus \zed^{\oplus2}$-graded link homology theory, where the $\zed_2$-grading is the $\zed_2$-grading of the underlying matrix factorization and the two $\zed$-gradings are the homological grading and the $x$-grading. We denote by $H_N^{\ve,i,k}$ the homogeneous component of $H_N$ of $\zed_2$-degree $\ve$, homological grading $i$ and $x$-grading $k$. 

The following theorem, which will be proved in Section \ref{sec-module-structure} below, describes the relation between $\fH_N$ and $H_N$.

\begin{theorem}\label{thm-sl-N-rel} 
Let $B$ be a closed braid, and $(\ve,i,k) \in\zed_2\oplus \zed^{\oplus2}$.
\begin{enumerate}[1.]
	\item $H_N^{\ve,i,k}(B) \cong \fH_N^{\ve,i,\star,k}(B)/(a-1)\fH_N^{\ve,i,\star,k}(B)$.
	\item As a $\zed$-graded $\Q[a]$-module, 
	\[
	\fH_N^{\ve,i,\star,k}(B)\cong (\bigoplus_{p=1}^{m_{\ve,i,k}} \Q[a]\{s_p\}) \bigoplus (\bigoplus_{q=1}^{n_{\ve,i,k}} \Q[a]/(a^{l_q})\{t_q\}),
	\] 
	where
	\begin{itemize}
	  \item $\{s\}$ means shifting the $a$-grading by $s$,
	  \item $m_{\ve,i,k} = \dim_\Q H_N^{\ve,i,k}(B)<\infty$, 
	  \item $n_{\ve,i,k}$ is a finite non-negative integer determined by $B$ and the triple $(\ve,i,k)$,
	  \item $\{s_1,\dots,s_{m_{\ve,i,k}}\} \subset \zed$ is a sequence determined up to permutation by $B$ and the triple $(\ve,i,k)$,
	  \item $\{(l_1,t_1),\dots,(l_{n_{\ve,i,k}},t_{n_{\ve,i,k}})\} \subset \zed^{\oplus2}$ is a sequence determined up to permutation by $B$ and the triple $(\ve,i,k)$.
	\end{itemize}
\end{enumerate}
\end{theorem}

\subsection{Organization of this paper} We review the definition and basic properties of matrix factorizations in Section \ref{sec-mf-review}. Then we define the matrix factorizations associated to MOY graphs and chain complexes associated to link diagrams in Sections \ref{sec-MOY-mf} and \ref{sec-def}. The invariance is established in Sections \ref{sec-R1}--\ref{sec-R3}. Finally, we discuss the decategorification and the relation to the $\slmf(N)$ Khovanov-Rozansky homology in Sections \ref{sec-decat} and \ref{sec-module-structure}.

Although this paper is mostly self-contained, some prior experiences with the Khovanov-Rozansky homology would certainly be helpful.

\section{Matrix Factorizations}\label{sec-mf-review}

In this section, we review the definition and some basic properties of matrix factorizations over the bigraded polynomial ring $\Q[a,X_1,\dots,X_k]$. 

We write $R=\Q[a,X_1,\dots,X_k]$ and fix a non-negative integer $N$ throughout this section.

\subsection{$\zed^{\oplus2}$-graded $R$-modules}

\begin{definition}\label{def-bigrading-ring}
We define a $\zed^{\oplus2}$-grading on $R=\Q[a,X_1,\dots,X_k]$ by letting $\deg a = (2,0)$ and $\deg X_i =(0,2n_i)$ for $i=1,\dots, k$, where each $n_i$ is a positive integer. We call the first component of this $\zed^{\oplus2}$-grading the $a$-grading and denote its degree function by $\deg_a$. We call the second component of this $\zed^{\oplus2}$-grading the $x$-grading and denote its degree function by $\deg_x$. An element of $R$ is said to be homogeneous if it is homogeneous with respect to both the $a$-grading and the $x$-grading.

A $\zed^{\oplus2}$-graded $R$-module $M$ is a $R$-module $M$ equipped with a $\zed^{\oplus2}$-grading such that, for any homogeneous element\footnote{An element of $M$ is said to be homogeneous if it is homogeneous with respect to both $\zed$-gradings.} $m$ of $M$, $\deg (am) = \deg m + (2,0)$ and $\deg (X_i m) = \deg m +(0,2n_i)$ for $i=1,\dots, k$. Again, we call the first component of this $\zed^{\oplus2}$-grading of $M$ the $a$-grading and denote its degree function by $\deg_a$. We call the second component of this $\zed^{\oplus2}$-grading of $M$ the $x$-grading and denote its degree function by $\deg_x$. 

We say that the $\zed^{\oplus2}$-grading on $M$ is bounded below if both the $a$-grading and the $x$-grading are bounded below.

For a $\zed^{\oplus2}$-graded $R$-module $M$, we denote by $M\{j,k\}$ the $\zed^{\oplus2}$-graded $R$-module obtained by shifting the $\zed^{\oplus2}$-grading of $M$ by $(j,k)$. That is, for any homogeneous element $m$ of $M$, $\deg_{M\{j,k\}} m = \deg_M m + (j,k)$.
\end{definition}

In our construction of $\mathcal{H}_N$, we need to use the fact that, if the $\zed^{\oplus2}$-grading of a free $\zed^{\oplus2}$-graded $R$-module $M$ is bounded below, then $M$ admits a homogeneous basis over $R$. To prove this, we start with the following lemma, which is implicitly given in \cite{Passman-book}.

\begin{lemma}\cite[Lemma 4.4]{Wu-color}\label{lemma-homogeneous-basis-exists}
Suppose that $M$ is a $\zed$-graded free $\Q[X_1,\dots,X_k]$-module whose grading is bounded below. Then $M$ admits a homogeneous basis over $\Q[X_1,\dots,X_k]$.
\end{lemma}

\begin{proof}
See \cite[Subsection 4.1]{Wu-color}.
\end{proof}

\begin{lemma}\label{lemma-bi-homogeneous-basis-exists}
Suppose that $M$ is a $\zed^{\oplus2}$-graded free $R$-module and its $a$-grading and $x$-grading are both bounded below. Then $M$ admits a homogeneous basis over $R$.

In particular, if $M$ is a $\zed^{\oplus2}$-graded finitely generated free $R$-module, then $M$ admits a homogeneous basis over $R$.
\end{lemma}

\begin{proof}
Assume that $M$ is a free $\zed^{\oplus2}$-graded $R$-module and both the $a$-grading and the $x$-grading on $M$ are bounded below. Denote by $j_0$ the lowest $a$-degree for any non-zero homogeneous element of $M$. Since $a$ is homogeneous, the $R$-module $M/aM$ inherits the $\zed^{\oplus2}$-grading of $M$. Also, note that the multiplication by $X_i$ does not affect the $a$-grading. So, as a $\zed$-graded $\Q[X_1,\dots,X_k]$-module, $M/aM = \bigoplus_{j=j_0}^\infty \mathcal{M}^j$, where $\mathcal{M}^j$ is the component of $M/aM$ of element homogeneous with respect to the $a$-grading of $a$-degree $j$. On each $\mathcal{M}^j$, the $x$-grading is bounded below. So, by Lemma \ref{lemma-homogeneous-basis-exists}, each $\mathcal{M}^j$ admits a homogeneous basis $\{\hat{v}_{j,p}~|~ p \in I_j\}$ with respect to the $x$-grading, where $I_j$ is an index set. Thus, $\{\hat{v}_{j,p}~|~ j\geq j_0,~p \in I_j\}$ is a homogeneous basis for the $\zed^{\oplus2}$-graded $\Q[X_1,\dots,X_k]$-module $M/aM$.

Denote by $\pi_M$ the standard quotient map $\pi_M:M\rightarrow M/aM$. For each $\hat{v}_{j,p}$, there exists a homogeneous element $v_{j,p}$ of $M$ such that $\pi_M(v_{j,p})= \hat{v}_{j,p}$, $\deg_a v_{j,p} = \deg_a\hat{v}_{j,p} = j$ and $\deg_x v_{j,p} = \deg_x \hat{v}_{j,p}$. We claim that $\{v_{j,p}~|~ j\geq j_0,~p \in I_j\}$ is a homogeneous basis for the $\zed^{\oplus2}$-graded $R$-module $M$. 

First, we prove $\{v_{j,p}~|~ j\geq j_0,~p \in I_j\}$ is $R$-linearly independent. Assume $\sum_{i=1}^l f_{j_i,p_i} v_{j_i,p_i} =0$ for some $\{v_{j_1,p_1},\dots,v_{j_l,p_l}\} \subset \{v_{j,p}~|~ j\geq j_0,~p \in I_j\}$, where $f_{j_i,p_i}$ is a non-zero element in $R$ for each $i=1,\dots,l$. After possibly dividing this sum by a power of $a$, we assume without loss of generality that, for some $i$, $f_{j_i,p_i}$ is not a multiple of $a$. Denote by $\pi$ the standard quotient map $\pi:R\rightarrow R/aR \cong \Q[X_1,\dots,X_k]$. Then, in $M/aM$, we have $\pi_M(\sum_{i=1}^l f_{j_i,p_i} v_{j_i,p_i}) = \sum_{i=1}^l \pi(f_{j_i,p_i}) \hat{v}_{j_i,p_i} =0$, where $\pi(f_{j_i,p_i}) \neq 0$ for some $i$. This is a contradiction since $\{\hat{v}_{j,p}~|~ j\geq j_0,~p \in I_j\}$ is a basis for the $\Q[X_1,\dots,X_k]$-module $M/aM$.

Now we prove by an induction on $\deg_a u$ that any homogeneous element $u$ of $M$ is in the span of $\{v_{j,p}~|~ j\geq j_0,~p \in I_j\}$. Recall that $j_0$ is the lowest $a$-degree for any non-zero homogeneous element of $M$. So, if $\deg_a u < j_0$ then $u=0$, which is in the span of $\{v_{j,p}~|~ j\geq j_0,~p \in I_j\}$. Now assume that, for some $j\geq j_0$, $u$ is in the span of $\{v_{j,p}~|~ j\geq j_0,~p \in I_j\}$ whenever $\deg_a u<j$. Suppose $\deg_a u = j$. Then $\pi_M(u) \in \mathcal{M}^j$ and, therefore, $\pi_M(u) = \sum_{p\in I_j} c_p \hat{v}_{j,p}$, where $c_p \in \Q[X_1,\dots,X_k]$. Thus, $\pi_M(u-\sum_{p\in I_j} c_p v_{j,p}) = 0$ and $u-\sum_{p\in I_j} c_p v_{j,p} \in aM$. Then, there is a element $v$ in $M$ such that 
\begin{itemize}
	\item $av= u-\sum_{p\in I_j} c_p v_{j,p}$,
	\item $v$ is homogeneous with respect to the $a$-grading and $\deg_a v = j-2$.
\end{itemize}
Note that each homogeneous part of $v$ is of $a$-degree $j-2$. By the induction hypothesis, $v$ is in the span of $\{v_{j,p}~|~ j\geq j_0,~p \in I_j\}$. Hence, $u$ is also in the span of $\{v_{j,p}~|~ j\geq j_0,~p \in I_j\}$. 

This completes the induction and proves that $M$ admits a homogeneous basis over $R$ if the $a$-grading and $x$-grading on $M$ are both bounded below.

If $M$ is finitely generated, then $M$ is generated by a finite set of homogeneous element. Then the lowest $a$-degree and $x$-degree of these element are lower bounds for the $a$-grading and the $x$-grading of $M$. So the lemma applies to $M$.  
\end{proof}

\subsection{Matrix factorizations and morphisms of matrix factorizations}

\begin{definition}\label{def-mf}
Let $w$ be a homogeneous element of $R=\Q[a,X_1,\dots,X_k]$ with bidegree $(2,2N+2)$. A $\zed_2\oplus\zed^{\oplus2}$-graded matrix factorization $M$ of $w$ over $R$ is a collection of two $\zed^{\oplus2}$-graded free $R$-modules $M_0$, $M_1$ and two homogeneous $R$-module maps $d_0:M_0\rightarrow M_1$, $d_1:M_1\rightarrow M_0$ of bidegree $(1,N+1)$, called differential maps, such that 
\[
d_1 \circ d_0=w\cdot\id_{M_0}, \hspace{1cm}  d_0 \circ d_1=w\cdot\id_{M_1}.
\]
We usually write $M$ as $M_0 \xrightarrow{d_0} M_1 \xrightarrow{d_1} M_0$.

The $\zed_2$-grading of $M$ takes value $\ve$ on $M_\ve$. The $a$- and $x$-gradings of $M$ are the $a$- and $x$-gradings of the underlying $\zed^{\oplus 2}$-graded $R$-module $M_0 \oplus M_1$.

Following \cite{KR1}, we denote by $M\left\langle 1\right\rangle$ the matrix factorization $M_1 \xrightarrow{d_1} M_0 \xrightarrow{d_0} M_1$ and write $M\left\langle j\right\rangle = M \underbrace{\left\langle 1\right\rangle\cdots\left\langle 1\right\rangle}_{j \text{ times }}$.

For any $\zed_2\oplus\zed^{\oplus2}$-graded matrix factorization $M$ of $w$ over $R$ and $j,k \in \zed$, $M\{j,k\}$ is naturally a $\zed_2\oplus\zed^{\oplus2}$-graded matrix factorization of $w$ over $R$.

For any two $\zed_2\oplus\zed^{\oplus2}$-graded matrix factorizations $M$ and $M'$ of $w$ over $R$, $M\oplus M'$ is naturally a $\zed_2\oplus\zed^{\oplus2}$-graded matrix factorization of $w$ over $R$.

Let $w$ and $w'$ be two homogeneous elements of $R$ with bidegree $(2,2N+2)$. For $\zed_2\oplus\zed^{\oplus2}$-graded matrix factorizations $M$ of $w$ and $M'$ of $w'$ over $R$, the tensor product $M\otimes_R M'$ is the $\zed_2\oplus\zed^{\oplus2}$-graded matrix factorization of $w+w'$ over $R$ such that:
\begin{itemize}
	\item $(M\otimes M')_0 = (M_0\otimes M'_0)\oplus (M_1\otimes M'_1)$, $(M\otimes M')_1 = (M_1\otimes M'_0)\oplus (M_1\otimes M'_0)$,
	\item The differential is given by the signed Leibniz rule. That is, $d(m\otimes m')=(dm)\otimes m' + (-1)^\ve m \otimes (dm')$ for $m\in M_\ve$ and $m'\in M'$.
\end{itemize}
\end{definition}

\begin{definition}\label{def-morph-mf}
Let $w$ be a homogeneous element of $R$ with bidegree $(2,2N+2)$, and $M$, $M'$ any two $\zed_2\oplus\zed^{\oplus2}$-graded matrix factorizations of $w$ over $R$. 
\begin{enumerate}
	\item A morphism of $\zed_2\oplus\zed^{\oplus2}$-graded matrix factorizations from $M$ to $M'$ is a homogeneous $R$-module homomorphism $f:M\rightarrow M'$ preserving the $\zed_2\oplus\zed^{\oplus2}$-grading satisfying $d_{M'}f=fd_{M}$. We denote by $\Hom_\mf(M,M')$ the $\Q$-space of all morphisms of $\zed_2\oplus\zed^{\oplus2}$-graded matrix factorizations from $M$ to $M'$.
	\item Two morphisms $f$ and $g$ of $\zed_2\oplus\zed^{\oplus2}$-graded matrix factorizations from $M$ to $M'$ are called homotopic if there is an $R$-module homomorphism $h:M\rightarrow M'$ shifting the $\zed_2$-grading by $1$ such that $f-g = d_{M'}h+hd_M$. In this case, we write $f \simeq g$. We denote by $\Hom_\hmf(M,M')$ the $\Q$-space of all homotopy classes of morphisms of $\zed_2\oplus\zed^{\oplus2}$-graded matrix factorizations from $M$ to $M'$. That is, $\Hom_\hmf(M,M') = \Hom_\mf(M,M') / \simeq$.
	\item Two morphisms $f$ and $g$ of $\zed_2\oplus\zed^{\oplus2}$-graded matrix factorizations from $M$ to $M'$ are called projectively homotopic if there is a $c\in\Q\setminus\{0\}$ such that $f \simeq cg$. In this case, we write $f \approx g$.
\end{enumerate}
\end{definition}

Let $M$ and $M'$ be as in Definition \ref{def-morph-mf}. Consider the $R$-module $\Hom_R(M,M')$ of $R$-module homomorphisms from $M$ to $M'$. It admits a $\zed_2$-grading that takes value 
\begin{itemize}
	\item $0$ on $\Hom^0_R(M,M')=\Hom_R(M_0,M'_0)\oplus\Hom_R(M_1,M'_1)$,
	\item $1$ on $\Hom^1_R(M,M')=\Hom_R(M_1,M'_0)\oplus\Hom_R(M_0,M'_1)$.
\end{itemize}
Moreover, $\Hom_R(M,M')$ admits a differential map $d$ given by $d(f)=d_{M'} \circ f -(-1)^\ve f \circ d_M$ for $f \in \Hom^\ve_R(M,M')$, which makes $\Hom_R(M,M')$ a chain complex with a $\zed_2$-homological grading.

\begin{lemma}\label{lemma-Hom-space}
Let $w$ be a homogeneous element of $R$ with bidegree $(2,2N+2)$, and $M$, $M'$ any two $\zed_2\oplus\zed^{\oplus2}$-graded matrix factorizations of $w$ over $R$. 
\begin{enumerate}
	\item A homogeneous $R$-module homomorphism $f:M\rightarrow M'$ preserving the $\zed_2\oplus\zed^{\oplus2}$-grading is a morphism of $\zed_2\oplus\zed^{\oplus2}$-graded matrix factorizations if and only if $df=0$.
	\item Two morphisms $f$ and $g$ of $\zed_2\oplus\zed^{\oplus2}$-graded matrix factorizations from $M$ to $M'$ are homotopic if and only if $f-g=dh$ for some $h \in \Hom^1_R(M,M')$.
	\item If $M$ is finitely generated over $R$, then $\Hom_R(M,M')$ is naturally $\zed_2\oplus\zed^{\oplus2}$-graded and
	      \begin{itemize}
	        \item $\Hom_\mf(M,M') = (\ker d)^{0,0,0}$, where $(\ker d)^{\ve,j,k}$ is the $\Q$-subspace of $\ker d$ of homogeneous elements of $\zed_2\oplus\zed^{\oplus2}$-degree $(\ve,j,k)$.
	        \item $\Hom_\hmf(M,M') = H^{0,0,0}(\Hom_R(M,M'),d)$, where $H^{\ve,j,k}(\Hom_R(M,M'),d)$ is the $\Q$-subspace of $H(\Hom_R(M,M'),d)$ of homogeneous elements of $\zed_2\oplus\zed^{\oplus2}$-degree $(\ve,j,k)$. 
        \end{itemize}
\end{enumerate}
\end{lemma}

\begin{proof}
The first two parts of the lemma are simple reformulations of definitions. For Part (3), note that, when $M$ is finitely generated, the $\zed_2\oplus\zed^{\oplus2}$-gradings of $M$ and $M'$ induces a $\zed_2\oplus\zed^{\oplus2}$-grading on $\Hom_R(M,M')$. Since $d$ is homogeneous under this grading, both $\ker d$ and $H(\Hom_R(M,M'),d)$ inherits this $\zed_2\oplus\zed^{\oplus2}$-grading. The rest follows easily.
\end{proof}

\begin{definition}\label{def-mf-iso-homo}
Let $w$ be a homogeneous element of $R$ with bidegree $(2,2N+2)$, and $M$, $M'$ any two $\zed_2\oplus\zed^{\oplus2}$-graded matrix factorizations of $w$ over $R$. 
\begin{enumerate}
	\item An isomorphism of $\zed_2\oplus\zed^{\oplus2}$-graded matrix factorizations from $M$ to $M'$ is a morphism of $\zed_2\oplus\zed^{\oplus2}$-graded matrix factorizations that is also an isomorphism of the underlying $R$-modules. We say that $M$ and $M'$ are isomorphic, or $M\cong M'$, if there is an isomorphism from $M$ to $M'$.
	\item $M$ and $M'$ are called homotopic, or $M\simeq M'$, if there are morphisms $f:M\rightarrow M'$ and $g:M'\rightarrow M$ such that $g\circ f \simeq \id_M$ and $f\circ g \simeq \id_{M'}$. $f$ and $g$ are called homotopy equivalences between $M$ and $M'$.
\end{enumerate}
\end{definition}

\subsection{Koszul matrix factorizations}

In the definition of $\fH_N$, we will use matrix factorizations of a special form, called Koszul matrix factorizations. We now review the definition and basic properties of Koszul matrix factorizations. 

\begin{definition}\label{def-koszul-mf}
If $a_0,a_1\in R$ are homogeneous elements with $\deg a_0 +\deg a_1=(2,2N+2)$, then denote by $(a_0,a_1)_R$ the $\zed_2\oplus\zed^{\oplus2}$-graded matrix factorization $R \xrightarrow{a_0} R\{1-\deg_a a_0,~N+1-\deg_x{a_0}\} \xrightarrow{a_1} R$ of $a_0a_1$ over $R$. More generally, if $a_{1,0},a_{1,1},\dots,a_{l,0},a_{l,1}\in R$ are homogeneous with $\deg a_{j,0} +\deg a_{j,1}=(2,2N+2)$, then denote by 
\[
\left(%
\begin{array}{cc}
  a_{1,0}, & a_{1,1} \\
  a_{2,0}, & a_{2,1} \\
  \dots & \dots \\
  a_{l,0}, & a_{l,1}
\end{array}%
\right)_R
\]
the tenser product $(a_{1,0},a_{1,1})_R \otimes_R (a_{2,0},a_{2,1})_R \otimes_R \cdots \otimes_R (a_{l,0},a_{l,1})_R$. This is a $\zed_2\oplus\zed^{\oplus2}$-graded matrix factorization of $\sum_{j=1}^l a_{j,0} a_{j,1}$ over $R$, and is call the Koszul matrix factorization associated to the above matrix. We drop``$R$" from the notation when it is clear from the context. 

Note that the above Koszul matrix factorization is finitely generated over $R$.
\end{definition}

\begin{lemma}\cite{KR1,KR2}\label{lemma-zed-2-shift} 
Let $a_0,~a_1$ and  $a_{1,0},a_{1,1},\dots,a_{l,0},a_{l,1}$ be as in Definition \ref{def-koszul-mf}. Then 
\begin{eqnarray*}
(a_1,a_0)_R & \cong & (a_0,a_1)_R\left\langle 1\right\rangle\{1-\deg_a a_1,~N+1-\deg_x{a_1}\}, \\
\left(%
\begin{array}{cc}
  a_{1,1}, & a_{1,0} \\
  a_{2,1}, & a_{2,0} \\
  \dots & \dots \\
  a_{l,1}, & a_{l,0}
\end{array}%
\right)_R
& \cong &
\left(%
\begin{array}{cc}
  a_{1,0}, & a_{1,1} \\
  a_{2,0}, & a_{2,1} \\
  \dots & \dots \\
  a_{l,0}, & a_{l,1}
\end{array}%
\right)_R \left\langle l \right\rangle\{\sum_{j=1}^l (1-\deg_a a_{j,1}),~\sum_{j=1}^l (N+1-\deg_x{a_{j,1}})\}.
\end{eqnarray*}
\end{lemma}

\begin{lemma}\cite{KR1,KR2}\label{lemma-dual-Koszul}
Let $a_{1,0},a_{1,1},\dots,a_{l,0},a_{l,1}$ be as in Definition \ref{def-koszul-mf}. Then for any $\zed_2\oplus\zed^{\oplus2}$-graded matrix factorization $M$ of $\sum_{j=1}^l a_{j,0} a_{j,1}$ over $R$,
\begin{eqnarray*}
&& \Hom_R(\left(%
\begin{array}{cc}
  a_{1,0}, & a_{1,1} \\
  a_{2,0}, & a_{2,1} \\
  \dots & \dots \\
  a_{l,0}, & a_{l,1}
\end{array}%
\right)_R, M) \cong 
M \otimes_R \left(%
\begin{array}{cc}
  -a_{1,1}, & a_{1,0} \\
 -a_{2,1}, & a_{2,0} \\
  \dots & \dots \\
  -a_{l,1}, & a_{l,0}
\end{array}%
\right)_R \\
& \cong &  M \otimes_R \left(%
\begin{array}{cc}
  a_{1,0}, & -a_{1,1} \\
  a_{2,0}, & -a_{2,1} \\
  \dots & \dots \\
  a_{l,0}, & -a_{l,1}
\end{array}%
\right)_R \left\langle l \right\rangle\{\sum_{j=1}^l (1-\deg_a a_{j,1}),~\sum_{j=1}^l (N+1-\deg_x{a_{j,1}})\}
\end{eqnarray*}
as $\zed_2\oplus\zed^{\oplus2}$-graded chain complexes.
\end{lemma}

\begin{lemma}\cite[Proposition 2]{KR1}\label{entries-null-homotopic}
Let $a_{1,0},a_{1,1},\dots,a_{l,0},a_{l,1}$ be as in Definition \ref{def-koszul-mf} and
\[
M = \left(%
\begin{array}{ll}
  a_{1,0}, & a_{1,1} \\
  a_{2,0}, & a_{2,1} \\
  \dots & \dots \\
  a_{l,0}, & a_{l,1}
\end{array}%
\right)_R.
\]
If $r$ is an element of the ideal $(a_{1,0}, a_{1,1}, \dots, a_{l,0}, a_{l,1})$ of $R$, then the multiplication by $r$, as an endomorphism of $M$, is homotopic to $0$.
\end{lemma}

\begin{lemma}\cite{Ras2}\label{lemma-twist}
Suppose $a_{1,0},a_{1,1},a_{2,0},a_{2,1},k$ are homogeneous elements of $R$ satisfying $\deg a_{j,0}+\deg a_{j,1}=(2,2N+2)$ and $\deg k = \deg a_{1,0} +\deg a_{2,0} -(2,2N+2)$. Then
\[
\left(%
\begin{array}{cc}
  a_{1,0} & a_{1,1} \\
  a_{2,0} & a_{2,1} 
\end{array}%
\right)_R
\cong
\left(%
\begin{array}{cc}
  a_{1,0}+ka_{2,1} & a_{1,1} \\
  a_{2,0}-ka_{1,1} & a_{2,1} 
\end{array}%
\right)_R.
\]
\end{lemma}

\begin{lemma}\cite{KR1,Ras2}\label{lemma-row-op} 
Suppose $a_{1,0},a_{1,1},a_{2,0},a_{2,1},c$ are homogeneous elements of $R$ satisfying $\deg a_{j,0}+\deg a_{j,1}=(2,2N+2)$ and $\deg c = \deg a_{1,0} -\deg a_{2,0}$. Then 
\[
\left(%
\begin{array}{cc}
  a_{1,0} & a_{1,1} \\
  a_{2,0} & a_{2,1} 
\end{array}%
\right)_R
\cong
\left(%
\begin{array}{cc}
  a_{1,0}+ca_{2,0} & a_{1,1} \\
  a_{2,0} & a_{2,1}-ca_{1,1} 
\end{array}%
\right)_R.
\]
\end{lemma}

\begin{definition}\label{def-regular-sequence}
Let $a_1,\dots,a_k$ be elements of $R$. The sequence $\{a_1,\dots,a_k\}$ is called $R$-regular if $a_1\neq 0$ and $a_j$ is not a zero divisor in $R/(a_1,\dots,a_{j-1})$ for $j=2,\dots,k$.
\end{definition}

\begin{lemma}\cite{KR3,Ras2}\label{lemma-freedom}
Suppose that $\{a_1,\dots,a_k\}$ is an $R$-regular sequence of homogeneous elements of $R$. Assume that $f_1,\dots,f_k,g_1,\dots,g_k$ are homogeneous elements of $R$ such that $\deg f_j = \deg g_j = (2,2N+2) -\deg a_j$ and $\sum_{j=1}^k f_ja_j=\sum_{j=1}^kg_ja_j$. Then
\[
\left(%
\begin{array}{ll}
  f_1, & a_1 \\
  \dots & \dots \\
  f_k, & a_k
\end{array}%
\right)_R
\cong
\left(%
\begin{array}{ll}
  g_1, & a_1 \\
  \dots & \dots \\
  g_k, & a_k
\end{array}%
\right)_R.
\]
\end{lemma}

The proofs of the above lemmas are omitted here since they fairly easy and can be found in for example \cite{KR1,KR2,KR3,Ras2,Wu7}. The following are two versions of \cite[Proposition 9]{KR1}, which are very useful in computations.

\begin{proposition}[strong version]\label{prop-b-contraction}
Let $X$ be a homogeneous indeterminate such that $\deg X= (0, 2n)$ and $n \leq N+1$. Denote by $P:R[X]\rightarrow R$ the evaluation map at $X=0$. That is, $P(f(X))=f(0)$ $\forall ~ f(X) \in R[X]$.

Suppose that $a_1,\dots,a_l,b_1,\dots,b_l$ are homogeneous elements of $R[X]$ such that 
\begin{itemize}
	\item $\deg a_j +\deg b_j = (2, 2N+2)$ $\forall~j=1,\dots,l$,
	\item $\sum_{j=1}^l a_jb_j \in R$,
	\item $\exists~ i\in \{1,\dots,l\}$ such that $b_i=X$.
\end{itemize}
Then
\[
M=\left(%
\begin{array}{cc}
  a_1 & b_1 \\
  a_2 & b_2 \\
  \dots & \dots \\
  a_k & b_k
\end{array}%
\right)_{R[X]}
\text{ and }
M'=\left(%
\begin{array}{cc}
  P(a_1) & P(b_1) \\
  P(a_2) & P(b_2) \\
  \dots & \dots \\
  P(a_{i-1}) & P(b_{i-1}) \\
  P(a_{i+1}) & P(b_{i+1}) \\
  \dots & \dots \\
  P(a_k) & P(b_k)
\end{array}%
\right)_{R}
\]
are homotopic as $\zed_2\oplus\zed^{\oplus2}$-graded matrix factorizations over $R$.
\end{proposition}

\begin{proof}
See \cite[Proposition 3.19]{Wu-color}.
\end{proof}

\begin{proposition}[weak version]\label{prop-contraction-weak}
Let $I$ be an ideal of $R$ generated by homogeneous elements. Assume $w$, $a_0$ and $a_1$ are homogeneous elements of $R$ such that $\deg w=\deg a_0 +\deg a_1 = (2,2N+2)$ and $w+a_0a_1 \in I$. Then $w \in I+(a_0)$ and $w \in I+(a_1)$. 

Let $M$ be a $\zed_2\oplus\zed^{\oplus2}$-graded matrix factorization of $w$ over $R$, and $\widetilde{M}=M \otimes_R (a_0,a_1)_R$. Then ${\widetilde{M}/I\widetilde{M}}$, ${M/(I+(a_0))M}$ and ${M/(I+(a_1))M}$ are all $\zed_2 \oplus\zed^{\oplus2}$-graded chain complexes of $R$-modules.
\begin{enumerate}
	\item If $a_0$ is not a zero-divisor in $R/I$, then there is an $R$-linear quasi-isomorphism $f:{\widetilde{M}/I\widetilde{M}} \rightarrow {(M/(I+(a_0))M)\left\langle 1\right\rangle \{1-\deg_a a_0, N+1-\deg_x a_0 \}}$ that preserves the $\zed_2\oplus\zed^{\oplus 2}$-grading.
	\item If $a_1$ is not a zero-divisor in $R/I$, then there is an $R$-linear quasi-isomorphism $g:{\widetilde{M}/I\widetilde{M}} \rightarrow {M/(I+(a_1))M}$ that preserves the $\zed_2\oplus\zed^{\oplus 2}$-grading.
\end{enumerate}
\end{proposition}

\begin{proof}
This proposition is \cite[Proposition 9]{KR1}. Since the quasi-isomorphism $g$ in Part (2) will be used in the proof of Theorem \ref{thm-neg-stabilization}, we sketch a proof for Part (2) here.

Write $M= M_0 \xrightarrow{d_0} M_1 \xrightarrow{d_1} M_0$. Recall that $(a_0, a_1)_R= R \xrightarrow{a_0} R\{1-\deg_a a_0,~N+1-\deg_x{a_0}\} \xrightarrow{a_1} R$. Then $\widetilde{M}=M \otimes_R (a_0,a_1)_R$ is the matrix factorization
{\footnotesize \[
\left.%
\begin{array}{l}
  M_0\\
  \oplus \\
  M_1\{1-\deg_a a_0,~N+1-\deg_x{a_0}\}
\end{array}%
\right.
\xrightarrow{\tilde{d}_0}
\left.%
\begin{array}{l}
  M_1\\
  \oplus \\
  M_0\{1-\deg_a a_0,~N+1-\deg_x{a_0}\}
\end{array}%
\right.
\xrightarrow{\tilde{d}_1}
\left.%
\begin{array}{l}
  M_0\\
  \oplus \\
  M_1\{1-\deg_a a_0,~N+1-\deg_x{a_0}\}
\end{array}%
\right.,
\]}

\noindent where
\[
\tilde{d}_0 = \left(%
\begin{array}{cc}
  d_0 & -a_1\\
  a_0 & d_1 \\
\end{array}%
\right), \hspace{5pc}
\tilde{d}_1 = \left(%
\begin{array}{cc}
  d_1 & a_1\\
  -a_0 & d_0 \\
\end{array}%
\right).
\]

For $\ve \in \zed_2$, denote by $P_\ve:M_\ve/IM_\ve \rightarrow M_\ve/(I+(a_1))M_\ve$ the standard quotient map. We define $g:{\widetilde{M}/I\widetilde{M}} \rightarrow (M/(I+(a_1))M)$ by the mappings
\begin{eqnarray*}
\left.%
\begin{array}{l}
  M_0/IM_0\\
  \oplus \\
  (M_1/IM_1)\{1-\deg_a a_0,~N+1-\deg_x{a_0}\}
\end{array}%
\right.
& \xrightarrow{(P_0,0)} &
M_0/(I+(a_1))M_0, \\
\left.%
\begin{array}{l}
  M_1/IM_1\\
  \oplus \\
  (M_0/IM_0)\{1-\deg_a a_0,~N+1-\deg_x{a_0}\}
\end{array}%
\right.
& \xrightarrow{(P_1,0)} & M_1/(I+(a_1))M_1.
\end{eqnarray*}

It is straightforward to check that 
\begin{itemize}
	\item $g$ is a surjective $R$-linear chain map that preserves the $\zed_2\oplus\zed^{\oplus 2}$-grading,
	\item $\ker g$ is a homotopically trivial subcomplex of ${\widetilde{M}/I\widetilde{M}}$.
\end{itemize}
So the short exact sequence 
\[
0 \rightarrow \ker g \hookrightarrow {\widetilde{M}/I\widetilde{M}} \xrightarrow{g} (M/(I+(a_1))M) \rightarrow 0
\]
induces an exact triangle
\[
\xymatrix{
H(\widetilde{M}/I\widetilde{M})  \ar[rr]^<<<<<<<<<<<g  & & H(M/(I+(a_1))M) \ar[dl] \\
& 0 \ar[ul]& \\
},
\]
which implies that $g$ is a quasi-isomorphism.
\end{proof}

\subsection{Categories of matrix factorizations}

\begin{definition}\label{categories-def}
Let $w$ be a homogeneous element of $R$ with bidegree $(2,2N+2)$. 

Suppose $M$ is a $\zed_2\oplus\zed^{\oplus2}$-graded matrix factorization of $w$ over $R$. We say that $M$ is homotopically finite if there exists a finitely generated graded matrix factorization $\mathcal{M}$ over $R$ with potential $w$ such that $M\simeq \mathcal{M}$. 

We define categories $\mf^{\mathrm{all}}_{R,w}$, $\mf_{R,w}$, $\hmf^{\mathrm{all}}_{R,w}$ and $\hmf_{R,w}$ by the following table. 

\begin{center}
\small{
\begin{tabular}{|c|l|c|}
\hline
Category & Objects & Morphisms \\
\hline
$\mf^{\mathrm{all}}_{R,w}$ & all $\zed_2\oplus\zed^{\oplus2}$-graded matrix factorizations of $w$ over $R$ with the   & $\Hom_{\mf}$ \\
 & $\zed^{\oplus2}$-grading bounded below &  \\
\hline
$\mf_{R,w}$ & all homotopically finite $\zed_2\oplus\zed^{\oplus2}$-graded matrix factorizations of $w$    & $\Hom_{\mf}$ \\
 & over $R$ with the $\zed^{\oplus2}$-grading bounded below &  \\
\hline
$\hmf^{\mathrm{all}}_{R,w}$ & all $\zed_2\oplus\zed^{\oplus2}$-graded matrix factorizations of $w$ over $R$ with the   & $\Hom_{\hmf}$ \\
 & $\zed^{\oplus2}$-grading bounded below &  \\
\hline
$\hmf_{R,w}$ & all homotopically finite $\zed_2\oplus\zed^{\oplus2}$-graded matrix factorizations of $w$    & $\Hom_{\hmf}$ \\
 & over $R$ with the $\zed^{\oplus2}$-grading bounded below &  \\
\hline
\end{tabular}
}
\end{center}
\end{definition}

\begin{definition}\label{def-homology-mf}
Let $w$ be a homogeneous element of $R$ with bidegree $(2,2N+2)$. Denote by $\mathfrak{I}$ the maximal homogeneous ideal $(a,X_1,\dots,X_k)$ of $R$. Note that $w\in \mathfrak{I}$ and, for a $\zed_2\oplus\zed^{\oplus2}$-graded matrix factorization $M$ of $w$ over $R$, $M/\mathfrak{I}M$ is a chain complex of $\zed^{\oplus2}$-graded $\Q$-spaces with a $\zed_2$-homological grading. We define $H_R(M)$ to be the $\zed_2\oplus\zed^{\oplus2}$-graded homology of $M/\mathfrak{I}M$.

Denote by $H_R^{\ve,j,k}(M)$ the subspace of $H_R(M)$ of homogeneous elements of $\zed_2$-degree $\ve$ and $\zed^{\oplus2}$-degree $(j,k)$. If the $\zed^{\oplus 2}$-grading of $M$ is bounded below and $\dim H_R^{\ve,j,k}(M) < \infty$ for all $\ve,j,k$, then we define the graded dimension of $M$ over $R$ to be
\[
\gdim_R(M) = \sum_{\ve,j,k} \tau^{\ve} \alpha^j\xi^k \dim_\Q H_R^{\ve,j,k}(M) ~\in ~\zed[[\alpha,\xi]][\alpha^{-1},\xi^{-1},\tau]/(\tau^2-1).
\]
\end{definition}

\begin{lemma}\label{lemma-free-module-quotient-tensor}
Let $M$ be a $\zed^{\oplus2}$-graded free $R$-module whose $\zed^{\oplus2}$-grading is bounded below. Define $V=M/\mathfrak{I} M$. Then there is a homogeneous $R$-module isomorphism $F: V \otimes_\Q R \rightarrow M$ preserving the $\zed^{\oplus2}$-grading. In particular, if $\{v_\beta | \beta \in \mathcal{B}\}$ is a homogeneous $\Q$-basis for $V$, then $\{F(v_\beta \otimes 1) | \beta \in \mathcal{B}\}$ is a homogeneous $R$-basis for $M$.
\end{lemma}
\begin{proof}
By Lemma \ref{lemma-bi-homogeneous-basis-exists}, $M$ has a homogeneous basis $\{e_\alpha|\alpha\in\mathcal{A}\}$. Then, as $\zed^{\oplus2}$-graded vector spaces, $V \cong \bigoplus_{\alpha\in\mathcal{A}} \Q \cdot e_\alpha$. So, as graded $R$-modules, $M \cong \bigoplus_{\alpha\in\mathcal{A}} R \cdot e_\alpha \cong V \otimes_\Q R$. This proves the existence of $F$. The rest of the lemma follows easily.
\end{proof}

\begin{proposition}\cite[Proposition 7]{KR1}\label{prop-contractible-essential-decomp}
Let $w$ be a homogeneous element of $R$ with bidegree $(2,2N+2)$, and $M$ a $\zed_2\oplus\zed^{\oplus2}$-graded matrix factorization of $w$ over $R$. Assume the $\zed^{\oplus2}$-grading of $M$ is bounded below. Then there exist $\zed_2\oplus\zed^{\oplus2}$-graded matrix factorizations $M_c$ and $M_{es}$ of $w$ over $R$ such that
\begin{enumerate}[(i)]
	\item $M \cong M_c \oplus M_{es}$,
	\item $M_c \simeq 0$ and, therefore, $M\simeq M_{es}$,
	\item $M_{es} \cong H_R(M)\otimes_\Q R$ as $\zed_2\oplus\zed^{\oplus2}$-graded $R$-modules, and $H_R(M) \cong M_{es}/\mathfrak{I}M_{es}$ as $\zed_2\oplus\zed^{\oplus2}$-graded $\Q$-spaces.
\end{enumerate}
\end{proposition}

\begin{proof}
(Following \cite{KR1}.) Write $M$ as $M_0 \xrightarrow{d_0} M_1 \xrightarrow{d_1} M_0$. Then the chain complex $V:=M/\mathfrak{I}M$ is given by $V_0 \xrightarrow{\hat{d}_0} V_1 \xrightarrow{\hat{d}_1} V_0$, where $V_\ve=M_\ve/\mathfrak{I}M_\ve$ for $\ve=0,1$. By Lemma \ref{lemma-bi-homogeneous-basis-exists}, $M_\ve$ has a homogeneous $R$-basis $\{e_\sigma|\sigma\in \mathcal{S}_\ve\}$, which induces a homogeneous $\Q$-basis $\{\hat{e}_\sigma|\sigma\in\mathcal{S}_\ve\}$ for $V_\ve$. Under the homogeneous basis $\{e_\sigma|\sigma\in \mathcal{S}_\ve\}$, the entries of matrices of $d_0$ and $d_1$ are homogeneous elements of $R$. And the matrices of $\hat{d}_0$ and $\hat{d}_1$ are obtained by letting $a=X_1=\cdots=X_m=0$ in the matrices of $d_0$ and $d_1$, which preserves scalar entries and kills entries with positive degrees.

We call $\{(\hat{u}_\rho,\hat{v}_\rho)|\rho\in\mathcal{P}\}$ a ``good" set if
\begin{itemize}
	\item $\{\hat{u}_\rho|\rho\in\mathcal{P}\}$ is a set of linearly independent homogeneous elements in $V_0$,
	\item $\{\hat{v}_\rho|\rho\in\mathcal{P}\}$ is a set of linearly independent homogeneous elements in $V_1$,
	\item $\hat{d}_0(\hat{u}_\rho) =\hat{v}_\rho$ and $\hat{d}_1(\hat{v}_\rho) =0$.
\end{itemize}
Using Zorn's Lemma, we find a maximal ``good" set $G=\{(\hat{u}_\alpha,\hat{v}_\alpha)|\alpha\in\mathcal{A}\}$. Using Zorn's Lemma again, we extend $\{\hat{u}_\alpha|\alpha\in\mathcal{A}\}$ into a homogeneous basis $\{\hat{u}_\alpha|\alpha\in\mathcal{A}\cup \mathcal{B}_0\}$ for $V_0$, and $\{\hat{v}_\alpha|\alpha\in\mathcal{A}\}$ into a homogeneous basis $\{\hat{v}_\alpha|\alpha\in\mathcal{A}\cup \mathcal{B}_1\}$ for $V_1$. For each $\beta \in \mathcal{B}_0$, we can write $\hat{d}_0 \hat{u}_\beta = \sum_{\alpha\in\mathcal{A}\cup \mathcal{B}_1} c_{\alpha\beta} \cdot \hat{v}_\alpha$, where $c_{\alpha\beta}\in \Q$, and the right hand side is a finite sum. 

By Lemma \ref{lemma-free-module-quotient-tensor}, there is a homogeneous isomorphism $F_\ve:V_\ve \otimes_\Q R \xrightarrow{\cong}M_\ve$ preserving the $\zed^{\oplus 2}$-grading. Let $u_\alpha = F_0(\hat{u}_\alpha \otimes 1)$ and  $v_\alpha = F_1(\hat{v}_\alpha \otimes 1)$. Then $\{u_\alpha|\alpha\in\mathcal{A}\cup \mathcal{B}_0\}$ and $\{v_\alpha|\alpha\in\mathcal{A}\cup \mathcal{B}_1\}$ are homogeneous $R$-bases for $M_0$ and $M_1$. Recall that $\hat{d}_0(\hat{u}_\alpha) =\hat{v}_\alpha$ for $\alpha \in \mathcal{A}$. So we have that, for any $\alpha \in \mathcal{A}$,
\[
d u_\alpha = v_\alpha + \sum_{\beta \in \mathcal{A}\cup \mathcal{B}_1,~\beta \neq \alpha} f_{\beta \alpha} v_\beta,
\]
where $f_{\beta \alpha} \in \mathfrak{I}$ and the sum on the right hand side is a finite sum. That is, for each $\alpha$,
\begin{equation}\label{contractible-essential-decomp-finite-sum}
f_{\beta \alpha} = 0 \text{ for all but finitely many } \beta.
\end{equation}

For two pairs of integers $(i,j)$, $(k,l)$, we say that 
\begin{itemize}
	\item $(i,j) \preceq (k,l)$ if $i \leq k$ and $j\leq l$,
	\item $(i,j) \prec (k,l)$ if $(i,j) \preceq (k,l)$ and $(i,j) \neq (k,l)$.
\end{itemize}

For each pair of $(\alpha,\beta)$ with $\alpha \in \mathcal{A}$, $\beta \in \mathcal{A}\cup\mathcal{B}_1$ and $\alpha \neq \beta$, the requirement that $f_{\beta \alpha} \in \mathfrak{I}$ implies 
\begin{equation}\label{contractible-essential-decomp-vanish}
f_{\beta \alpha}\neq 0 \text{ only if } \deg v_\beta \prec \deg v_\alpha.
\end{equation}
For $\alpha\in \mathcal{A}$ and $k>0$, let 
\[
C_{\ast\alpha}^k =\{(\gamma_0,\dots,\gamma_k)\in \mathcal{A}^{k+1} | ~\gamma_k = \alpha, ~  \deg v_{\gamma_0} \prec \cdots \prec\deg v_{\gamma_k},~f_{\gamma_0\gamma_1}\cdots f_{\gamma_{k-1}\gamma_k} \neq0\}.
\] 
By \eqref{contractible-essential-decomp-finite-sum}, $C_{\ast\alpha}^k$ is a finite set. For each $\alpha$, $C_{\ast\alpha}^k = \emptyset$ for large $k$'s since the $\zed^{\oplus 2}$-grading of $M$ is bounded below. For $\alpha, \beta \in \mathcal{A}$ and $k>0$, let 
\[
C_{\beta\alpha}^k =\{(\gamma_0,\dots,\gamma_k)\in C_{\ast\alpha}^k | \gamma_0 = \beta\}.
\] 
Then $\cup_{\beta \in \mathcal{A}} C_{\beta\alpha}^k = C_{\ast\alpha}^k$. So each $C_{\beta\alpha}^k$ is finite. And, for each $k$, $C_{\beta\alpha}^k\neq \emptyset$ for only finitely many $\beta$. Also, by definition, it is easy to see that $C_{\beta\alpha}^k\neq\emptyset$ only if $\deg v_\beta \prec \deg v_\alpha$. Moreover, for each $\alpha$, there is a $k_0>0$ such that $C_{\beta\alpha}^k=\emptyset$ for any $\beta$ whenever $k>k_0$.

Now, for $\alpha,\beta \in \mathcal{A}$, define $t_{\beta\alpha} \in R$ by
\[
t_{\beta\alpha} = 
\begin{cases}
1 & \text{if } \beta=\alpha, \\
\sum_{k\geq 1} (-1)^k \sum_{(\gamma_0,\dots,\gamma_k) \in C_{\beta\alpha}^k} f_{\gamma_0\gamma_1}\cdots f_{\gamma_{k-1}\gamma_k}   & \text{if } \deg v_\beta \prec \deg v_\alpha,\\
0 & \text{otherwise.} 
\end{cases}
\]
From the above discussion, we know that the sum on the right hand side is always a finite sum. So $t_{\beta\alpha}$ is well defined. Furthermore, given an $\alpha \in \mathcal{A}$, $t_{\beta\alpha}=0$ for all but finitely many $\beta$. So, for $\alpha \in \mathcal{A}$, $u_\alpha':=\sum_{\beta\in \mathcal{A}} t_{\beta\alpha} u_\beta$ is well defined. It is straightforward to check that:
\begin{itemize}
	\item $\{u_\alpha'|\alpha\in\mathcal{A}\}\cup \{u_\beta|\beta\in\mathcal{B}_0\}$ is also a homogeneous $R$-basis for $M_0$,
	\item For $\alpha \in \mathcal{A}$, $d u_\alpha' = v_\alpha + \sum_{\beta \in \mathcal{B}_1} f_{\beta \alpha}' v_\beta$, where the right hand side is a finite sum, and $f_{\beta \alpha}' \in \mathfrak{I}$. 
\end{itemize}
Now let
\[
v_\alpha'=
\begin{cases}
v_\alpha + \sum_{\beta \in \mathcal{B}_1} f_{\beta \alpha}' v_\beta & \text{if } \alpha \in \mathcal{A}, \\
v_\alpha & \text{if } \alpha \in \mathcal{B}_1.
\end{cases}
\]
Then $\{v_\alpha'|\alpha\in\mathcal{A}\cup \mathcal{B}_1\}$ is a homogeneous $R$-basis for $M_1$. We have
\[
\begin{cases}
d u_\alpha' = v_\alpha' & \text{if } \alpha \in \mathcal{A}, \\
d u_\beta = \sum_{\alpha \in \mathcal{A}} g_{\alpha\beta} v_\alpha' + \sum_{\gamma \in \mathcal{B}_1} g_{\gamma\beta} v_\gamma' & \text{if } \beta \in \mathcal{B}_0,
\end{cases}  
\]
where the sums on the right hand side are finite sums. For $\beta \in \mathcal{B}_0$, we let $u_\beta'= u_\beta - \sum_{\alpha \in \mathcal{A}} g_{\alpha\beta} u_\alpha'$. Then $\{u_\alpha'|\alpha\in\mathcal{A}\cup\mathcal{B}_0\}$ is again a homogeneous $R$-basis for $M_0$, and 
\[
\begin{cases}
d u_\alpha' = v_\alpha' & \text{if } \alpha \in \mathcal{A}, \\
d u_\beta' = \sum_{\gamma \in \mathcal{B}_1} g_{\gamma\beta} v_\gamma' & \text{if } \beta \in \mathcal{B}_0,
\end{cases}  
\]
where the sum on the right hand side is a finite sum. Using that $d_1d_0 = w \cdot\id_{M_0}$ and $d_0d_1 = w \cdot\id_{M_1}$, one can check that
\[
\begin{cases}
d v_\alpha' = w\cdot v_\alpha' & \text{if } \alpha \in \mathcal{A}, \\
d v_\beta' = \sum_{\gamma \in \mathcal{B}_0} h_{\gamma\beta} u_\gamma' & \text{if } \beta \in \mathcal{B}_1,
\end{cases}  
\]
where the sum on the right hand side is a finite sum.

Define $M_{(1,w)}$ to be the submodule of $M$ spanned by $\{u_\alpha'|\alpha\in\mathcal{A}\} \cup \{v_\alpha'|\alpha\in\mathcal{A}\}$, and $M'$ the submodule of $M$ spanned by $\{u_\beta'|\beta\in\mathcal{B}_0\} \cup \{v_\beta'|\beta\in\mathcal{B}_1\}$. Then $M_{(1,w)}$ and $M'$ are both $\zed_2\oplus\zed^{\oplus 2}$-graded matrix factorizations of $w$ over $R$ and $M=M_{(1,w)} \oplus M'$. Note that
\begin{enumerate}[(a)]
	\item $M_{(1,w)}$ is a direct sum of components of the form $(1,w)_R \{j,k\}$,
	\item Under the standard projection $M\rightarrow M/\mathfrak{I}M$, we have, for $\alpha \in \mathcal{A}$, $u_\alpha' \mapsto \hat{u}_\alpha$ and $v_\alpha' \mapsto \hat{v}_\alpha$.
\end{enumerate}
In particular, (b) above means that $M'$ does not have direct sum components of the form $(1,w)_R \{j,k\}$. Otherwise, we can enlarge the ``good" set $G$, which contradicts the fact that $G$ is maximal. We then apply a similar argument to $M'$ and find a decomposition $M' = M_{(w,1)} \oplus M_{es}$ of $\zed_2\oplus\zed^{\oplus 2}$-graded matrix factorizations satisfying
\begin{itemize}
	\item $M_{(w,1)}$ is a direct sum of components of the form $(w,1)_R \{j,k\}$,
	\item $M_{es}$ has no direct sum component of the forms $(1,w)_R \{j,k\}$ or $(w,1)_R \{j,k\}$.
\end{itemize}
Let $M_c = M_{(1,w)} \oplus M_{(w,1)}$. Then $M = M_c \oplus M_{es}$. $M_c \simeq 0$ since $(1,w)_R \{j,k\}$ and $(w,1)_R \{j,k\}$ are both homotopic to $0$. So $M\simeq M_{es}$. It is clear that, under any homogeneous basis for $M_{es}$, all entries of the matrices representing the differential map of $M_{es}$ must be in $\mathfrak{I}$. Otherwise, a simple change of basis would show that $M_{es}$ has a direct sum component of the form $(1,w)_R \{j,k\}$ or $(w,1)_R \{j,k\}$. Therefore, $H_R(M) \cong H_R(M_{es}) \cong M_{es}/\mathfrak{I}M_{es}$. So, by Lemma \ref{lemma-free-module-quotient-tensor}, $M_{es} \cong H_R(M) \otimes_\Q R$ as graded modules.
\end{proof}

\begin{corollary}\cite[Corollary 4]{KR1}\label{cor-homology-detects-homotopy}
Let $w$ be a homogeneous element of $R$ with bidegree $(2,2N+2)$, and $M$ a $\zed_2\oplus\zed^{\oplus2}$-graded matrix factorization of $w$ over $R$ whose $\zed^{\oplus2}$-grading is bounded below. We have:
\begin{enumerate}
	\item $M\simeq 0$ if and only if $H_R(M)=0$ or, equivalently, $\gdim_R (M)=0$.
	\item $M$ is homotopically finite if and only if $\dim_\Q H_R(M)$ is finite.
\end{enumerate}
\end{corollary}

\begin{proof}
For Part (1), note that, by Proposition \ref{prop-contractible-essential-decomp}, $M\simeq 0$ if and only if $M_{es}\simeq 0$ if and only if $H_R(M)=0$. 

Now consider (2). If $M$ is homotopically finite, then there is a finitely generated $\zed_2\oplus\zed^{\oplus2}$-graded matrix factorization $\mathcal{M}$ of $w$ over $R$ such that $M\simeq \mathcal{M}$. Note that $\mathcal{M}/\mathfrak{I}\mathcal{M}$ is finite dimensional over $\Q$. Thus, $H_R(M)\cong H_R(\mathcal{M})$ is finite dimensional over $\Q$. On the other hand, if $H_R(M)$ is finite dimensional over $\Q$, then, by Proposition \ref{prop-contractible-essential-decomp}, $M_{es}\cong H_R(M) \otimes_\Q R$ is finitely generated over $R$. But $M\simeq M_{es}$. So $M$ is homotopically finite.
\end{proof}

\begin{definition}\label{def-fully-additive}
An additive category $\mathscr{C}$ is said to be fully additive if every idempotent endomorphism in $\mathscr{C}$ splits. That is, for every object $C$ of $\mathscr{C}$ and every endomorphism $f$ of $C$ satisfying $f\circ f=f$, there exist objects $C_0$ and $C_1$ of $\mathscr{C}$ and an isomorphism 
$\left.%
\begin{array}{l}
  C_0\\
  \oplus \\
  C_1
\end{array}%
\right. \xrightarrow{(J_0,J_1)} C$ such that $f\circ J_0=0$ and $f\circ J_1=J_1$.
\end{definition}

\begin{lemma}\label{lemma-fully-additive-split}
Let $\mathscr{C}$ be a fully additive category. Assume that 
\begin{itemize}
	\item $A$ and $C$ are objects of $\mathscr{C}$,
	\item $A\xrightarrow{f} C$ and $C\xrightarrow{g} A$ are morphisms of $\mathscr{C}$ such that $g\circ f = \id_A$.
\end{itemize}
Then there exits an object $C_0$ of $\mathscr{C}$ such that $C \cong A\oplus C_0$.
\end{lemma}

\begin{proof}
Consider the morphism $C\xrightarrow{f\circ g}C$. We have $(f\circ g)\circ(f\circ g) = f\circ \id_A \circ g =f \circ g$. Since $\mathscr{C}$ is fully additive, there exist objects $C_0$ and $C_1$ of $\mathscr{C}$ and an isomorphism 
$\left.%
\begin{array}{l}
  C_0\\
  \oplus \\
  C_1
\end{array}%
\right. \xrightarrow{(J_0,J_1)} C$ such that $f\circ g\circ J_0=0$ and $f\circ g\circ J_1=J_1$. Denote by 
$C \xrightarrow{\left(%
\begin{array}{l}
  P_0\\
  P_1
\end{array}%
\right)} 
\left.%
\begin{array}{l}
  C_0\\
  \oplus \\
  C_1
\end{array}%
\right.$ the inverse of $(J_0,J_1)$. Note that $P_1\circ f \circ g\circ J_1 =P_1\circ J_1 = \id_{C_1}$. Moreover, $\id_A=g\circ f =g\circ  \id_C \circ f= g\circ  (J_0\circ P_0 + J_1\circ P_1) \circ f$. So 
\begin{eqnarray*}
\id_A & = & \id_A \circ \id_A = g\circ f \circ g\circ  (J_0\circ P_0 + J_1\circ P_1) \circ f \\
& = & g\circ (f \circ g\circ J_0\circ P_0 + f \circ g\circ J_1\circ P_1) \circ f = g\circ  J_1\circ P_1 \circ f.\\ 
\end{eqnarray*}
Thus, the morphisms $A \xrightarrow{P_1\circ f}C_1$ and $C_1 \xrightarrow{g\circ J_1} A$ are isomorphisms. This shows that $A \cong C_1$ and, therefore, $C \cong C_0 \oplus C_1 \cong A \oplus C_0$.
\end{proof}

\begin{proposition}\cite[Proposition 24]{KR1}\label{fully-additive-hmf}
Let $w$ be a homogeneous element of $R$ with bidegree $(2,2N+2)$. Then $\mf^{\mathrm{all}}_{R,w}$, $\mf_{R,w}$ and $\hmf_{R,w}$ are all fully additive.
\end{proposition}

\begin{proof}(Following \cite{KR1}.)
The category of $\zed^{\oplus2}$-graded $R$-modules is of course fully additive. By the Quillen-Suslin Theorem, we know that any projective $R$-module is a free $R$-module. Thus, the category of $\zed^{\oplus2}$-graded free $R$-modules is also fully additive. From this, it is easy to deduce that $\mf^{\mathrm{all}}_{R,w}$ and $\mf_{R,w}$ are both fully additive.

Next, we prove that $\hmf_{R,w}$ is fully additive. 

Let $M$ be an object of $\hmf_{R,w}$ and $f:M\rightarrow M$ a morphism of $\zed_2\oplus\zed^{\oplus2}$-graded matrix factorization such that $f\circ f\simeq f$. Denote by $P: M \rightarrow M_{es}$ and $J:M_{es} \rightarrow M$ the homotopy equivalences from Part (ii) of Proposition \ref{prop-contractible-essential-decomp}. Then $f$ induces a morphism $f_{es}=P\circ f\circ J: M_{es} \rightarrow M_{es}$, which satisfies $f_{es} \circ f_{es} \simeq f_{es}$.

Let $\alpha:\Hom_{\mf}(M_{es},M_{es}) \rightarrow \Hom_{\hmf}(M_{es},M_{es})$ be the natural projection taking each morphism to its homotopy class, and $\beta:\Hom_{\mf}(M_{es},M_{es}) \rightarrow \Hom_\Q(H_R(M),H_R(M))$ the map taking each morphism to the induced map on the homology.  Then $\ker \alpha$ and $\ker \beta$ are ideals of the ring $\Hom_{\mf}(M_{es},M_{es})$, and $\ker \alpha \subset \ker \beta$. 

Note that $M$ is homotopically finite. So, by Proposition \ref{prop-contractible-essential-decomp} and Corollary \ref{cor-homology-detects-homotopy}, $M_{es}$ is finitely generated. Let $\{e_1,\dots,e_n\}$ be a homogeneous basis for $M_{es}$. For any $h \in \ker \beta$, denote by $H$ its matrix under this basis. By Proposition \ref{prop-contractible-essential-decomp}, $H_R(M) \cong M_{es}/\mathfrak{I}M_{es}$. Since $\beta (h)=0$, we know that all entries of $H$ are elements of $\mathfrak{I}$. Thus, if $h \in (\ker \beta)^l$, then all entries of $H$ are elements of $\mathfrak{I}^l$. Recall that $\deg_a a=2$, $\deg_x X_i\geq 2$ and $\mathfrak{I}=(a,X_1,\dots,X_k)$. A simple degree count shows that the matrix of a homogeneous endomorphism of $M_{es}$ preserving the $\zed^{\oplus 2}$-grading can not contain non-zero entries from $\mathfrak{I}^K$, where
\[
K:= \max \{\max\{\deg_a e_i -\deg_a e_j~|~1\leq i,j\leq n\},~\max\{\deg_x e_i -\deg_x e_j~|~1\leq i,j\leq n\}\}.
\]
Thus, $(\ker \beta)^K =0$ and, therefore, $(\ker \alpha)^K =0$. This shows that $\ker \alpha$ is a nilpotent ideal of $\Hom_{\mf}(M_{es},M_{es})$. By \cite[Theorem 1.7.3]{Benson-book}, nilpotent ideals have the lifting idempotents property. Thus, there is an endomorphism $g_{es}\in \Hom_{\mf}(M_{es}, M_{es})$ satisfying $g_{es}\simeq f_{es}$ and $g_{es}\circ g_{es}=g_{es}$.

But $\mf_{R,w}$ is fully additive. So $g_{es}$ splits $M_{es}$ into a direct sum of two finitely generated $\zed_2\oplus\zed^{\oplus2}$-graded matrix factorizations. This direct sum is a splitting of $M$ by $f$ in the category $\hmf_{R,w}$.
\end{proof}

\section{Matrix Factorizations Associated to MOY Graphs}\label{sec-MOY-mf}

In this section, we define matrix factorizations associated to MOY graphs, which are the building blocks of the chain complex $\fC_N$ used to define the homology $\fH_N$. Throughout this section, we fix a non-negative integer $N$ and let $a$ be a homogeneous indeterminate of bidegree $\deg a = (2,0)$.

\subsection{Symmetric polynomials} In this subsection, we recall some facts about symmetric polynomials, which will be used in our definition of $\fH_N$.

\begin{definition}
A finite collection of homogeneous indeterminates of bidegree $(0,2)$ is called an alphabet. We denote by $\Sym(\mathbb{X})$ the ring of symmetric polynomials over $\Q$ in the alphabet $\mathbb{X}=\{x_1,\dots,x_m\}$. More generally, given a collection $\{\mathbb{X}_1,\dots,\mathbb{X}_l\}$ of pairwise disjoint alphabets, we denote by $\Sym(\mathbb{X}_1|\cdots|\mathbb{X}_l)$ the ring of polynomials in $\mathbb{X}_1\cup\cdots\cup\mathbb{X}_l$ over $\Q$ that are symmetric in each $\mathbb{X}_i$.\footnote{$\Sym(\mathbb{X}_1|\cdots|\mathbb{X}_l)$ is bigger than $\Sym(\mathbb{X}_1\cup\cdots\cup\mathbb{X}_l)$. In fact, $\Sym(\mathbb{X}_1|\cdots|\mathbb{X}_l)$ is a finitely generated free $\Sym(\mathbb{X}_1\cup\cdots\cup\mathbb{X}_l)$-module.} That is, $\Sym(\mathbb{X}_1|\cdots|\mathbb{X}_l)= \Sym(\mathbb{X}_1)\otimes_\Q\cdots\otimes_\Q\Sym(\mathbb{X}_l)$. 
\end{definition}

For an alphabet $\mathbb{X}=\{x_1,\dots,x_m\}$, we denote by $X_k$, $h_k(\mathbb{X})$ and $p_k(\mathbb{X})$ the elementary, complete and power sum symmetric polynomials in $\mathbb{X}$. That is,
\begin{eqnarray}
\label{eq-def-elementary} X_k & = & {\begin{cases}
                                     \sum_{1\leq i_1<i_2<\cdots<i_k\leq m} x_{i_1}x_{i_1}\cdots x_{i_k} & \text{if } 1\leq k \leq m, \\
                                     1 & \text{if } k=0, \\
                                     0 & \text{if } k < 0 \text{ or } k > m,
                                     \end{cases}} \\
\label{eq-def-complete} h_k(\mathbb{X}) & = & {\begin{cases}
                                               \sum_{1\leq i_1\leq i_2\leq \cdots<\leq i_k\leq m} x_{i_1}x_{i_1}\cdots x_{i_k},  & \text{if } k \geq 1, \\
                                               1 & \text{if } k=0, \\
                                               0 & \text{if } k<0,
                                               \end{cases}} \\
\label{eq-def-powersum} p_{k}(\mathbb{X}) & = & {\begin{cases}
                                                 \sum_{j=1}^m x_j^{k}  & \text{if } k \geq 0, \\
                                                 0 & \text{if } k<0.
                                                 \end{cases}}
\end{eqnarray}
Recall that $\Sym(\mathbb{X}) = \Q[X_1\dots,X_m]$. So there are unique $m$-variable polynomials $h_{m,k}$ and $p_{m,k}$ such that
\begin{eqnarray}
\label{eq-def-poly-h} h_k(\mathbb{X})& = & h_{m,k}(X_1,\dots,X_m), \\
\label{eq-def-poly-p} p_{k}(\mathbb{X}) & = & p_{m,k}(X_1,\dots,X_m).
\end{eqnarray}

\begin{lemma}\cite[Lemma 5.1]{Wu-color}\label{lemma-power-derive}
\[
\frac{\partial}{\partial X_j} p_{m,k}(X_1,\dots,X_m) = (-1)^{j+1} k h_{m,k-j}(X_1,\dots,X_m).
\]
\end{lemma}

\subsection{Matrix factorizations associated to MOY graphs} We now recall the definition of MOY graphs and define $\zed_2\oplus\zed^{\oplus2}$-graded matrix factorizations associated to MOY graphs. Although the definitions in this subsection are for general MOY graphs, we will only need $1,2,3$-colored MOY graphs in our construction of $\fH_N$. 

\begin{definition}\label{def-MOY-graph}
An abstract MOY graph is an oriented graph with every edge colored by a non-negative integer such that, for every vertex $v$ with valence at least $2$, the sum of the colors of the edges entering $v$ is equal to the sum of the colors of the edges leaving $v$.

A vertex of valence $1$ in an abstract MOY graph is called an end point. A vertex of valence greater than $1$ is called an internal vertex. An abstract MOY graph $\Gamma$ is said to be closed if it has no end points. We say that an abstract MOY graph is trivalent if all of its internal vertices have valence $3$.

A MOY graph is an embedding of an abstract MOY graph into $\mathbb{R}^2$ such that, through each internal vertex $v$, there is a straight line $L_v$ so that all the edges entering $v$ enter through one side of $L_v$ and all edges leaving $v$ leave through the other side of $L_v$.

A marking of a MOY graph $\Gamma$ consists of the following:
\begin{enumerate}
	\item A finite collection of marked points on $\Gamma$ such that
	\begin{itemize}
	\item every edge of $\Gamma$ has at least one marked point;
	\item all the end points (vertices of valence $1$) are marked;
	\item none of the internal vertices (vertices of valence at least $2$) are marked.
  \end{itemize}
  \item An assignment of pairwise disjoint alphabets to the marked points such that the alphabet assigned to a marked point on an edge of color $m$ has $m$ independent indeterminates.
\end{enumerate}
\end{definition}

\begin{figure}[ht]

\setlength{\unitlength}{1pt}

\begin{picture}(360,80)(-180,-40)


\put(0,0){\vector(-1,1){15}}

\put(-15,15){\line(-1,1){15}}

\put(-23,25){\tiny{$i_1$}}

\put(-33,32){\small{$\mathbb{X}_1$}}

\put(0,0){\vector(-1,2){7.5}}

\put(-7.5,15){\line(-1,2){7.5}}

\put(-11,25){\tiny{$i_2$}}

\put(-18,32){\small{$\mathbb{X}_2$}}

\put(3,25){$\cdots$}

\put(0,0){\vector(1,1){15}}

\put(15,15){\line(1,1){15}}

\put(31,25){\tiny{$i_k$}}

\put(27,32){\small{$\mathbb{X}_k$}}


\put(4,-2){$v$}

\multiput(-50,0)(5,0){19}{\line(1,0){3}}

\put(-70,0){$L_v$}

\put(45,0){\tiny{$i_1+i_2+\cdots +i_k = j_1+j_2+\cdots +j_l$}}


\put(-30,-30){\vector(1,1){15}}

\put(-15,-15){\line(1,1){15}}

\put(-26,-30){\tiny{$j_1$}}

\put(-33,-40){\small{$\mathbb{Y}_1$}}

\put(-15,-30){\vector(1,2){7.5}}

\put(-7.5,-15){\line(1,2){7.5}}

\put(-13,-30){\tiny{$j_2$}}

\put(-18,-40){\small{$\mathbb{Y}_2$}}

\put(3,-30){$\cdots$}

\put(30,-30){\vector(-1,1){15}}

\put(15,-15){\line(-1,1){15}}

\put(31,-30){\tiny{$j_l$}}

\put(27,-40){\small{$\mathbb{Y}_l$}}

\end{picture}

\caption{}\label{general-MOY-vertex}

\end{figure}

For a MOY graph $\Gamma$ with a marking, cut it at its marked points. This gives a collection of marked MOY graphs, each of which is a star-shaped neighborhood of a vertex in $\Gamma$ and is marked only at its endpoints. (If an edge of $\Gamma$ has two or more marked points, then some of these pieces may be oriented arcs from one marked point to another. In this case, we consider such an arc as a neighborhood of an additional vertex of valence $2$ in the middle of that arc.)

Let $v$ be a vertex of $\Gamma$ with coloring and marking around it given as in Figure \ref{general-MOY-vertex}. Set $m=i_1+i_2+\cdots +i_k = j_1+j_2+\cdots +j_l$. Define 
\[
R=\Q[a]\otimes_\Q\Sym(\mathbb{X}_1|\dots|\mathbb{X}_k|\mathbb{Y}_1|\dots|\mathbb{Y}_l).
\] 
Note that $R$ is the $\zed^{\oplus 2}$-graded polynomial ring generated by $a$ and the elementary symmetric polynomials of $\mathbb{X}_1,\dots,~\mathbb{X}_k,~\mathbb{Y}_1,\dots,~\mathbb{Y}_l$, with bigrading given by $\deg a = (2,0)$ and $\deg x =(0,2)$ for all $x \in {\mathbb{X}_1\cup\cdots\cup\mathbb{X}_k\cup\mathbb{Y}_1\cup\cdots\cup\mathbb{Y}_l}$.

Write $\mathbb{X}=\mathbb{X}_1\cup\cdots\cup \mathbb{X}_k$ and $\mathbb{Y}=\mathbb{Y}_1\cup\cdots\cup \mathbb{Y}_l$, each of which is an alphabet of $m$ indeterminates. Denote by $X_j$ and $Y_j$ the $j$-th elementary symmetric polynomials in $\mathbb{X}$ and $\mathbb{Y}$. For $j=1,\dots,m$, define
\begin{equation}\label{eq-def-U-j}
U_j = \frac{p_{m,N+1}(Y_1,\dots,Y_{j-1},X_j,\dots,X_m) - p_{m,N+1}(Y_1,\dots,Y_j,X_{j+1},\dots,X_m)}{X_j-Y_j}.
\end{equation}

We associate to the vertex $v$ the $\zed_2\oplus\zed^{\oplus 2}$-graded matrix factorization
\[
\fC_N(v)=\left(%
\begin{array}{cc}
  aU_1 & X_1-Y_1 \\
  aU_2 & X_2-Y_2 \\
  \dots & \dots \\
  aU_m & X_m-Y_m
\end{array}%
\right)_R
\{0,-\sum_{1\leq s<t \leq k} i_si_t\},
\]
of $\sum_{j=1}^m (X_j-Y_j)U_j = ap_{N+1}(\mathbb{X})-ap_{N+1}(\mathbb{Y})$ over $R$, where $p_{N+1}(\mathbb{X})$ and $p_{N+1}(\mathbb{Y})$ are the $(N+1)$-th power sum symmetric polynomials in $\mathbb{X}$ and $\mathbb{Y}$.

\begin{remark}\label{MOY-freedom}
Since 
\[
\Sym(\mathbb{X}|\mathbb{Y})=\Q[X_1,\dots,X_m,Y_1,\dots,Y_m]=\Q[X_1-Y_1,\dots,X_m-Y_m,Y_1,\dots,Y_m],
\] 
it is clear that $\{X_1-Y_1,\dots,X_m-Y_m\}$ is $\Sym(\mathbb{X}|\mathbb{Y})$-regular. (See Definition \ref{def-regular-sequence}.) But $R$ is a free $\Sym(\mathbb{X}|\mathbb{Y})$-module. (See for example \cite{Lascoux-notes}.) So $\{X_1-Y_1,\dots,X_m-Y_m\}$ is also $R$-regular. Thus, by Lemma \ref{lemma-freedom}, the isomorphism type of $C(v)$ does not depend on the particular choice of $U_1,\dots,U_m$ as long as they are homogeneous with the right degrees and the potential of $C(v)$ remains $\sum_{j=1}^m a(X_j-Y_j)U_j = ap_{N+1}(\mathbb{X})-ap_{N+1}(\mathbb{Y})$. From now on, we will only specify our choice for $U_1,\dots,U_m$ when it is actually used in the computation. Otherwise, we will simply denote the entries in the left column of $\fC_N(v)$ by $\ast$'s. 
\end{remark}

\begin{definition}\label{def-MOY-mf}
\[
\fC_N(\Gamma) := \bigotimes_{v} \fC_N(v),
\]
where $v$ runs through all the interior vertices of $\Gamma$ (including those additional $2$-valent vertices.) Here, the tensor product is done over the common end points. More precisely, for two sub-MOY graphs $\Gamma_1$ and $\Gamma_2$ of $\Gamma$ intersecting only at (some of) their end points, let $\mathbb{W}_1,\dots,\mathbb{W}_n$ be the alphabets associated to these common end points. Then, in the above tensor product, $\fC_N(\Gamma_1)\otimes \fC_N(\Gamma_2)$ is the tensor product $\fC_N(\Gamma_1)\otimes_{\Q[a]\otimes_\Q\Sym(\mathbb{W}_1|\dots|\mathbb{W}_n)} \fC_N(\Gamma_2)$.

Note that $\fC_N(\Gamma)$ is a $\zed_2\oplus\zed^{\oplus 2}$-graded matrix factorization.  We denote by $\fC_N^{\ve,j,k}(\Gamma)$ the homogeneous component of $\fC_N(\Gamma)$ of $\zed_2$-degree $\ve$, $a$-degree $j$ and $x$-degree $k$.

If $\Gamma$ is closed, that is, has no end points, then $\fC_N(\Gamma)$ is considered a $\zed_2\oplus\zed^{\oplus 2}$-graded matrix factorization of $0$ over $\Q[a]$. 

Assume $\Gamma$ has end points. Let $\mathbb{E}_1,\dots,\mathbb{E}_n$ be the alphabets assigned to all end points of $\Gamma$, among which $\mathbb{E}_1,\dots,\mathbb{E}_k$ are assigned to exits and $\mathbb{E}_{k+1},\dots,\mathbb{E}_n$ are assigned to entrances. Define the boundary ring of $\Gamma$ to be $R_\partial=\Q[a]\otimes_\Q\Sym(\mathbb{E}_1|\cdots|\mathbb{E}_n)$. Then $\fC_N(\Gamma)$ is viewed as a $\zed_2\oplus\zed^{\oplus 2}$-graded matrix factorization over $R_\partial$ of $w= \sum_{i=1}^k ap_{N+1}(\mathbb{E}_i) - \sum_{j=k+1}^n ap_{N+1}(\mathbb{E}_j)$.

We allow the MOY graph to be empty. In this case, we define $\fC_N(\emptyset)=\Q[a] \rightarrow 0 \rightarrow \Q[a]$.
\end{definition}

\begin{lemma}\label{lemma-marking-independence}
If $\Gamma$ is a MOY graph, then the homotopy type of $\fC_N(\Gamma)$ does not depend on the choice of the marking.
\end{lemma}
\begin{proof}
We only need to show that adding or removing an extra marked point corresponds to a homotopy equivalence of $\zed_2\oplus\zed^{\oplus 2}$-graded matrix factorizations. This follows easily from Proposition \ref{prop-b-contraction}.
\end{proof}

\begin{definition}\label{def-homology-MOY-closed}
Let $\Gamma$ be a closed MOY graph with a marking. Then $\fC_N(\Gamma)$ is a $\zed_2\oplus\zed^{\oplus 2}$-graded chain complex of free $\Q[a]$-modules. Define $\fH_N(\Gamma)$ to be the homology of $\fC_N(\Gamma)$. 

Note that $\fH_N(\Gamma)$ is a $\zed_2\oplus\zed^{\oplus 2}$-graded $\Q[a]$-modules. Denote by $\fH_N^{\ve,j,k}(\Gamma)$ the homogeneous component of $\fH_N(\Gamma)$ of $\zed_2$-degree $\ve$, $a$-degree $j$ and $x$-degree $k$. Define the $\zed_2\oplus\zed^{\oplus 2}$-graded dimension of $\Gamma$ to be $\gdim ~\Gamma = \sum_{\ve,j,k} \tau^\ve \alpha^j \xi^k \dim_\Q \fH_N^{\ve,j,k}(\Gamma) \in \zed[[\alpha,\xi]][\alpha^{-1},\xi^{-1},\tau]/(\tau^2-1)$.
\end{definition}

\begin{remark}
In the language of Definition \ref{def-homology-mf}, $\gdim ~\Gamma = \gdim_\Q(\fC_N(\Gamma))\neq \gdim_{\Q[a]}(\fC_N(\Gamma))$. Here, ${\gdim ~\Gamma}$ is well defined because $\fC_N(\Gamma)$ is a finitely generated $\zed_2\oplus \zed^{\oplus2}$-graded matrix factorization over a polynomial ring with finitely many indeterminates, which implies that:
\begin{enumerate}
	\item $\dim_\Q \fH_N^{\ve,j,k}(\Gamma) < \infty$ $\forall~\ve,j,k$,
	\item the $\zed^{\oplus 2}$-grading of $\fH_N(\Gamma)$ is bounded below.
\end{enumerate}
\end{remark}

\subsection{Edge sliding} Consider the MOY graphs in Figure \ref{contract-expand-figure}. We call the local changes $\Gamma_1 \rightleftharpoons \Gamma_1'$ and $\Gamma_2 \rightleftharpoons \Gamma_2'$ edge slidings. In this subsection, we demonstrate that an edge sliding induces, up to homotopy and scaling, a unique homotopy equivalence. 

\begin{figure}[ht]

\setlength{\unitlength}{1pt}

\begin{picture}(360,100)(-180,-50)


\put(-100,25){$\Gamma_1$:}

\put(-60,10){\vector(0,1){10}}

\put(-60,20){\vector(-1,1){20}}

\put(-60,20){\vector(1,1){10}}

\put(-50,30){\vector(-1,1){10}}

\put(-50,30){\vector(1,1){10}}

\put(-61,3){\tiny{$3$}}

\put(-55,21){\tiny{$2$}}

\put(-80,42){\tiny{$1$}}

\put(-60,42){\tiny{$1$}}

\put(-40,42){\tiny{$1$}}


\put(20,25){$\Gamma'_1$:}

\put(60,10){\vector(0,1){10}}

\put(60,20){\vector(1,1){20}}

\put(60,20){\vector(-1,1){10}}

\put(50,30){\vector(1,1){10}}

\put(50,30){\vector(-1,1){10}}

\put(59,3){\tiny{$3$}}

\put(48,21){\tiny{$2$}}

\put(80,42){\tiny{$1$}}

\put(60,42){\tiny{$1$}}

\put(40,42){\tiny{$1$}}


\put(-100,-25){$\Gamma_2$:}

\put(-60,-30){\vector(0,-1){10}}

\put(-80,-10){\vector(1,-1){20}}

\put(-50,-20){\vector(-1,-1){10}}

\put(-60,-10){\vector(1,-1){10}}

\put(-40,-10){\vector(-1,-1){10}}

\put(-61,-47){\tiny{$3$}}

\put(-55,-29){\tiny{$2$}}

\put(-80,-8){\tiny{$1$}}

\put(-60,-8){\tiny{$1$}}

\put(-40,-8){\tiny{$1$}}


\put(20,-25){$\Gamma'_2$:}

\put(60,-30){\vector(0,-1){10}}

\put(80,-10){\vector(-1,-1){20}}

\put(50,-20){\vector(1,-1){10}}

\put(60,-10){\vector(-1,-1){10}}

\put(40,-10){\vector(1,-1){10}}

\put(59,-47){\tiny{$3$}}

\put(48,-29){\tiny{$2$}}

\put(80,-8){\tiny{$1$}}

\put(60,-8){\tiny{$1$}}

\put(40,-8){\tiny{$1$}}

\end{picture}

\caption{}\label{contract-expand-figure}

\end{figure}

\begin{lemma}\label{lemma-edge-sliding}
Suppose that $\Gamma_1$, $\Gamma'_1$, $\Gamma_2$ and $\Gamma'_2$ are the MOY graphs shown in Figure \ref{contract-expand-figure}. Mark corresponding end points of $\Gamma_i$ and $\Gamma_i'$ with the same alphabet and denote by $R_\partial$ the common boundary ring. Then $\fC_N(\Gamma_1) \simeq \fC_N(\Gamma'_1)$ and $\fC_N(\Gamma_2) \simeq \fC_N(\Gamma'_2)$ as $\zed_2\oplus\zed^{\oplus 2}$-graded matrix factorizations over $R_\partial$.
\end{lemma}

\begin{figure}[ht]
\[
\xymatrix{
\input{edge-slide-middle}  && \input{edge-slide-right} 
}
\]
\caption{}\label{proof-edge-sliding-figure}

\end{figure}

\begin{proof}
Let $\Gamma$ be the MOY graph in Figure \ref{proof-edge-sliding-figure}. We mark $\Gamma$ and $\Gamma_1$ as in Figure \ref{proof-edge-sliding-figure}, where $\mathbb{X}_1 = \{x_1,x_2,x_3\}$ and $\mathbb{X}_2=\{x_4,x_5\}$. Write $R_\partial=\Q[a,x_6,x_7,x_8]\otimes_\Q \Sym(\mathbb{X}_1)$ and $R=\Q[a,x_6,x_7,x_8]\otimes_\Q \Sym(\mathbb{X}_1|\mathbb{X}_2)$. From Proposition \ref{prop-b-contraction}, we know that, as $\zed_2\oplus\zed^{\oplus 2}$-graded matrix factorizations over $R_\partial$,
\begin{eqnarray*}
\fC_N(\Gamma_1) & = & \left(%
\begin{array}{cc}
  \ast & x_7+x_8-x_4-x_5 \\
  \ast & x_7x_8-x_4x_5 \\
  \ast & x_4+x_5+x_6- x_1-x_2-x_3 \\
  \ast & x_4x_5+x_5x_6+x_6x_4- x_1x_2-x_2x_3-x_3x_1 \\
  \ast & x_4x_5x_6- x_1x_2x_3
\end{array}%
\right)_{R}\{0,-3\} \\
& \simeq &  \left(%
\begin{array}{cc}
  \ast & x_7+x_8+x_6- x_1-x_2-x_3 \\
  \ast & x_7x_8+x_8x_6+x_6x_7- x_1x_2-x_2x_3-x_3x_1 \\
  \ast & x_7x_8x_6- x_1x_2x_3
\end{array}%
\right)_{R_\partial}\{0,-3\} \\
& = & \fC_N(\Gamma).
\end{eqnarray*}
Similarly, $\fC_N(\Gamma_1') \simeq \fC_N(\Gamma)$. So $\fC_N(\Gamma_1) \simeq \fC_N(\Gamma'_1)$ and, similarly, $\fC_N(\Gamma_2) \simeq \fC_N(\Gamma'_2)$.
\end{proof}

\begin{lemma}\label{lemma-edge-sliding-unique}
Suppose that $\Gamma_1$, $\Gamma'_1$, $\Gamma_2$ and $\Gamma'_2$ are the MOY graphs shown in Figure \ref{contract-expand-figure}. Then the homotopy equivalences $\fC_N(\Gamma_1) \xrightarrow{\simeq} \fC_N(\Gamma'_1)$, $\fC_N(\Gamma_1') \xrightarrow{\simeq} \fC_N(\Gamma_1)$, $\fC_N(\Gamma_2) \xrightarrow{\simeq} \fC_N(\Gamma'_2)$ and $\fC_N(\Gamma_2') \xrightarrow{\simeq} \fC_N(\Gamma_2)$ of $\zed_2\oplus\zed^{\oplus 2}$-graded matrix factorizations from Lemma \ref{lemma-edge-sliding} are unique up to homotopy and scaling by non-zero scalars.
\end{lemma}

\begin{proof}
We only prove the uniqueness of $\fC_N(\Gamma_1) \xrightarrow{\simeq} \fC_N(\Gamma'_1)$ here. The proof of uniqueness of the other homotopy equivalences are similar and left to the reader. Let $R_\partial$ be as in the proof of Lemma \ref{lemma-edge-sliding}. To prove the uniqueness of $\fC_N(\Gamma_1) \xrightarrow{\simeq} \fC_N(\Gamma'_1)$, we only need to prove that, in the category 
\[
\hmf_{R_\partial,a(x_6^{N+1}+x_7^{N+1}+x_8^{N+1}-x_1^{N+1}-x_2^{N+1}-x_3^{N+1})},
\] 
we have $\Hom_\hmf (\fC_N(\Gamma_1),\fC_N(\Gamma_1')) \cong \Q$. Since $\fC_N(\Gamma_1) \simeq \fC_N(\Gamma_1') \simeq \fC_N(\Gamma)$, we just need to check that $\Hom_\hmf (\fC_N(\Gamma),\fC_N(\Gamma)) \cong \Q$. 

First, it is easy to see that $\fC_N(\Gamma)$ is finitely generated over $R_\partial$ and $\gdim_{R_\partial} (\fC_N(\Gamma)) \neq 0$. By Corollary \ref{cor-homology-detects-homotopy}, this means that $\fC_N(\Gamma)$ is not homotopic to $0$. Thus, the identity map of $\fC_N(\Gamma)$ is not homotopic to $0$. This implies that $\dim_\Q \Hom_\hmf (\fC_N(\Gamma),\fC_N(\Gamma)) \geq 1$. 

On the other hand, by Lemma \ref{lemma-dual-Koszul},
\[
 \Hom_{R_\partial} (\fC_N(\Gamma),\fC_N(\Gamma)) 
 \cong  \left(%
\begin{array}{cc}
  aU_1 & x_7+x_8+x_6- x_1-x_2-x_3 \\
  aU_2& x_7x_8+x_8x_6+x_6x_7- x_1x_2-x_2x_3-x_3x_1 \\
  aU_3 & x_7x_8x_6- x_1x_2x_3 \\
  aU_1 & -(x_7+x_8+x_6- x_1-x_2-x_3) \\
  aU_2 & -(x_7x_8+x_8x_6+x_6x_7- x_1x_2-x_2x_3-x_3x_1) \\
  aU_3 & -(x_7x_8x_6- x_1x_2x_3)
\end{array}%
\right)_{R_\partial}\left\langle 3\right\rangle\{3,3N-9\},
\]
where $U_1$, $U_2$ and $U_3$ are given by equation \eqref{eq-def-U-j}. Let $R'=\Q[a,x_6,x_7,x_8]$. By Proposition \ref{prop-b-contraction} and Lemma \ref{lemma-power-derive}, we have that, as $\zed_2\oplus\zed^{\oplus 2}$-graded matrix factorizations of $0$ over $R'$,
\[
\left(%
\begin{array}{cc}
  aU_1 & x_7+x_8+x_6- x_1-x_2-x_3 \\
  aU_2& x_7x_8+x_8x_6+x_6x_7- x_1x_2-x_2x_3-x_3x_1 \\
  aU_3 & x_7x_8x_6- x_1x_2x_3 \\
  aU_1 & -(x_7+x_8+x_6- x_1-x_2-x_3) \\
  aU_2 & -(x_7x_8+x_8x_6+x_6x_7- x_1x_2-x_2x_3-x_3x_1) \\
  aU_3 & -(x_7x_8x_6- x_1x_2x_3)
\end{array}%
\right)_{R_\partial}
\simeq
\left(%
\begin{array}{cc}
  ah_N(\{x_6,x_7,x_8\}) & 0 \\
  ah_{N-1}(\{x_6,x_7,x_8\}) & 0 \\
  ah_{N-2}(\{x_6,x_7,x_8\}) & 0
\end{array}%
\right)_{R'}.
\]
Thus, as $\zed_2\oplus\zed^{\oplus 2}$-graded matrix factorizations of $0$ over $R'$,
\[
\Hom_{R_\partial} (\fC_N(\Gamma),\fC_N(\Gamma)) \simeq M := \left(%
\begin{array}{cc}
  ah_N(\{x_6,x_7,x_8\}) & 0 \\
  ah_{N-1}(\{x_6,x_7,x_8\}) & 0 \\
  ah_{N-2}(\{x_6,x_7,x_8\}) & 0
\end{array}%
\right)_{R'} \left\langle 3 \right\rangle\{3,3N-9\}.
\]
As a $\zed_2\oplus \zed^{\oplus 2}$-graded $R'$-module,
\[
M \cong (R'\oplus R'\left\langle 1\right\rangle\left\{1, N-1\right\}) \otimes_{R'} (R'\oplus R'\left\langle 1\right\rangle\left\{1, N-3\right\}) \otimes_{R'} (R'\oplus R'\left\langle 1\right\rangle\left\{1, N-5\right\}).
\]
From this, it is easy to see that the homogeneous component of $M$ of $\zed_2\oplus \zed^{\oplus 2}$-degree $(0,0,0)$ is $1$-dimensional over $\Q$. This implies that $\dim_\Q H^{0,0,0}(\Hom_{R_\partial} (\fC_N(\Gamma),\fC_N(\Gamma))) \leq 1$.
But, by Lemma \ref{lemma-Hom-space}, 
\[
\Hom_\hmf (\fC_N(\Gamma),\fC_N(\Gamma)) \cong H^{0,0,0}(\Hom_{R_\partial} (\fC_N(\Gamma),\fC_N(\Gamma))).
\] 
Therefore, $\dim_\Q \Hom_\hmf (\fC_N(\Gamma),\fC_N(\Gamma)) \leq 1$.

Putting the above together, we get $\Hom_\hmf (\fC_N(\Gamma),\fC_N(\Gamma)) \cong \Q$. 
\end{proof}

\subsection{Edge splitting and merging} Let $\Gamma$ and $\Gamma_1$ be the MOY graphs in Figure \ref{decomp-II-figure}. We call the change $\Gamma_1 \rightharpoonup \Gamma$ an edge splitting and the change $\Gamma \rightharpoonup \Gamma_1$ an edge merging. In this subsection, we define the morphisms induced by edge splitting and merging.

\begin{figure}[ht]

\setlength{\unitlength}{1pt}

\begin{picture}(360,70)(-180,-10)


\put(-60,0){\vector(0,1){15}}

\qbezier(-60,15)(-70,15)(-70,25)

\put(-70,25){\vector(0,1){10}}

\qbezier(-70,35)(-70,45)(-60,45)

\put(-71,26){\line(1,0){2}}

\qbezier(-60,15)(-50,15)(-50,25)

\put(-50,25){\vector(0,1){10}}

\qbezier(-50,35)(-50,45)(-60,45)

\put(-51,26){\line(1,0){2}}

\put(-60,45){\vector(0,1){15}}

\put(-65,55){\tiny{$2$}}

\put(-58,54){\small{$\{x_1,x_2\}$}}

\put(-65,0){\tiny{$2$}}

\put(-58,0){\small{$\{x_3,x_4\}$}}

\put(-77,30){\tiny{$1$}}

\put(-80,22){\small{$x_5$}}

\put(-47,30){\tiny{$1$}}

\put(-48,22){\small{$x_6$}}

\put(-63,-10){$\Gamma$}


\put(60,0){\vector(0,1){60}}

\put(62,54){\small{$\{x_1,x_2\}$}}

\put(55,30){\tiny{$2$}}

\put(62,0){\small{$\{x_3,x_4\}$}}

\put(57,-10){$\Gamma_1$}

\end{picture}

\caption{}\label{decomp-II-figure}

\end{figure}

\begin{lemma}\cite[Proposition 30]{KR1}\label{lemma-decomp-II}
Let $R_\partial=\Q[a]\otimes_\Q\Sym(\{x_1,x_2\}|\{x_3,x_4\})$. Then 
\[
\fC_N(\Gamma) \simeq \fC_N(\Gamma_1)\{0,-1\} \oplus \fC_N(\Gamma_1)\{0,1\}
\] 
as $\zed_2\oplus\zed^{\oplus 2}$-graded matrix factorizations of $a(x_1^{N+1}+x_2^{N+1}-x_3^{N+1}-x_4^{N+1})$ over $R_\partial$.
\end{lemma}

\begin{figure}[ht]

\[
\input{decomp-II-proof-fig}
\]

\caption{}\label{decomp-II-proof-figure}

\end{figure}

\begin{proof}(Following \cite{KR1}.)
Set 
\begin{eqnarray*}
R & = & \Q[a]\otimes_\Q\Sym(\{x_1,x_2\}|\{x_3,x_4\}|\{x_5,x_6\}), \\
\widetilde{R}& = & \Q[a]\otimes_\Q\Sym(\{x_1,x_2\}|\{x_3,x_4\})\otimes_\Q \Q[x_5,x_6].
\end{eqnarray*}
Then $R$ is a subring of $\widetilde{R}$ and, as $\zed^{\oplus 2}$-graded $R$-modules, $\widetilde{R}= R\cdot 1 \oplus R \cdot x_5 \cong R \oplus R \{0,2\}$.

By definition,
\begin{equation}\label{eq-proof-lemma-decomp-II-1}
\fC_N(\Gamma) = \left(%
\begin{array}{cc}
  \ast & x_1+x_2-x_5-x_6 \\
  \ast & x_1x_2-x_5x_6 \\
  \ast & x_5+x_6-x_3-x_4 \\
  \ast & x_5x_6-x_3x_4 
\end{array}%
\right)_{\widetilde{R}} \{0,-1\}.
\end{equation}
We add an extra marked point to $\Gamma_1$ as in Figure \ref{decomp-II-proof-figure}. Then, by Lemma \ref{lemma-marking-independence}, we have 
\begin{equation}\label{eq-proof-lemma-decomp-II-2}
\fC_N(\Gamma_1) \simeq \left(%
\begin{array}{cc}
  \ast & x_1+x_2-x_5-x_6 \\
  \ast & x_1x_2-x_5x_6 \\
  \ast & x_5+x_6-x_3-x_4 \\
  \ast & x_5x_6-x_3x_4 
\end{array}%
\right)_{R}. 
\end{equation}

Combine the above, we get that $\fC_N(\Gamma) \simeq \fC_N(\Gamma_1)\{0,-1\} \oplus \fC_N(\Gamma_1)\{0,1\}$ as $\zed_2\oplus\zed^{\oplus 2}$-graded matrix factorizations of $a(x_1^{N+1}+x_2^{N+1}-x_3^{N+1}-x_4^{N+1})$ over $R$. Since $R_\partial$ is a subring of $R$, this proves the lemma.
\end{proof}

\begin{lemma}\label{lemma-hmf-decomp-II}
Let $\Gamma$, $\Gamma_1$ and $R_\partial$ be as in Lemma \ref{lemma-decomp-II}. Then, in the category $\hmf_{R_\partial,a(x_1^{N+1}+x_2^{N+1}-x_3^{N+1}-x_4^{N+1})}$, we have that, for $k\geq 1$,
\[
\Hom_\hmf(\fC_N(\Gamma),\fC_N(\Gamma_1)\{0,k\}) \cong \Hom_\hmf(\fC_N(\Gamma_1),\fC_N(\Gamma)\{0,k\}) \cong \begin{cases}
\Q & \text{if } k=1,\\
0 & \text{if } k >1.
\end{cases}
\]
\end{lemma}

\begin{proof}
By Lemma \ref{lemma-decomp-II}, we have $\fC_N(\Gamma) \simeq \fC_N(\Gamma_1)\{0,-1\} \oplus \fC_N(\Gamma_1)\{0,1\}$. So 
\begin{eqnarray}
\label{eq-lemma-hmf-decomp-II-1} && \Hom_\hmf(\fC_N(\Gamma),\fC_N(\Gamma_1)\{0,k\}) \cong \Hom_\hmf(\fC_N(\Gamma_1),\fC_N(\Gamma)\{0,k\}) \\
\nonumber &\cong & \Hom_\hmf(\fC_N(\Gamma_1),\fC_N(\Gamma_1)\{0,k-1\}) \oplus \Hom_\hmf(\fC_N(\Gamma_1),\fC_N(\Gamma_1)\{0,k+1\}) \\
\nonumber &\cong & H^{0,0,1-k}(\Hom_{R_\partial} (\fC_N(\Gamma_1),\fC_N(\Gamma_1))) \oplus H^{0,0,-1-k}(\Hom_{R_\partial} (\fC_N(\Gamma_1),\fC_N(\Gamma_1))).
\end{eqnarray}
By Lemma \ref{lemma-dual-Koszul}, we have that
\[
\Hom_{R_\partial} (\fC_N(\Gamma_1),\fC_N(\Gamma_1)) \cong \left(%
\begin{array}{cc}
  aU_1 & x_1+x_2-x_3-x_4 \\
  aU_2 & x_1x_2-x_3x_4 \\
  aU_1 & x_3+x_4-x_1-x_2 \\
  aU_2 & x_3x_4-x_1x_2 
\end{array}%
\right)_{R_\partial}\{2,2N-4\},
\]
where $U_1$ and $U_2$ are given by equation \eqref{eq-def-U-j}. By Proposition \ref{prop-b-contraction} and Lemma \ref{lemma-power-derive}, as $\zed_2\oplus\zed^{\oplus 2}$-graded matrix factorizations over $\hat{R}:=\Q[a] \otimes_\Q \Sym(\{x_1,x_2\})$,
\[
\left(%
\begin{array}{cc}
  aU_1 & x_1+x_2-x_3-x_4 \\
  aU_2 & x_1x_2-x_3x_4 \\
  aU_1 & x_3+x_4-x_1-x_2 \\
  aU_2 & x_3x_4-x_1x_2 
\end{array}%
\right)_{R_\partial}
\cong
\left(%
\begin{array}{cc}
  ah_N(x_1,x_2) & 0 \\
  ah_{N-1}(x_1,x_2) & 0 
\end{array}%
\right)_{\hat{R}}.
\]
So, as $\zed_2 \oplus \zed^{\oplus 2}$-graded chain complexes over $\hat{R}$,
\[
\Hom_{R_\partial} (\fC_N(\Gamma_1),\fC_N(\Gamma_1)) \cong \left(%
\begin{array}{cc}
  ah_N(x_1,x_2) & 0 \\
  ah_{N-1}(x_1,x_2) & 0 
\end{array}%
\right)_{\hat{R}} \{2,2N-4\}.
\]
But, as a $\zed_2 \oplus \zed^{\oplus 2}$-graded $\hat{R}$-module,
\[
\left(%
\begin{array}{cc}
  ah_N(x_1,x_2) & 0 \\
  ah_{N-1}(x_1,x_2) & 0 
\end{array}%
\right)_{\hat{R}} \{2,2N-4\} \cong (\hat{R}\{1,N-1\} \oplus \hat{R}\left\langle 1\right\rangle) \otimes_{\hat{R}} (\hat{R}\{1,N-3\} \oplus \hat{R}\left\langle 1\right\rangle).
\]
This implies that
\begin{equation}
\label{eq-lemma-hmf-decomp-II-2} \dim_\Q H^{0,0,-l}(\Hom_{R_\partial} (\fC_N(\Gamma_1),\fC_N(\Gamma_1))) \begin{cases}
\leq 1 & \text{if } l=0,\\
=0  & \text{if } l<0.
\end{cases}
\end{equation}
It is easy to check that $H_{R_\partial}(\fC_N(\Gamma_1)) \ncong 0$. This means that $\fC_N(\Gamma_1)$ is not homotopic to $0$ and, therefore, $\id_{\fC_N(\Gamma_1)}$ is not homotopic to $0$. Thus, 
\begin{equation}\label{eq-lemma-hmf-decomp-II-3}
\dim_\Q H^{0,0,0}(\Hom_{R_\partial} (\fC_N(\Gamma_1),\fC_N(\Gamma_1))) =1.
\end{equation}
Now the lemma follows from equations \eqref{eq-lemma-hmf-decomp-II-1}, \eqref{eq-lemma-hmf-decomp-II-2} and \eqref{eq-lemma-hmf-decomp-II-3}.
\end{proof}

\begin{lemma}\label{lemma-phi}
Let $\Gamma$, $\Gamma_1$ and $R_\partial$ be as in Lemma \ref{lemma-decomp-II}. Then, up to homotopy and scaling, there exist unique homotopically non-trivial $R_\partial$-linear homogeneous morphisms of matrix factorizations $\fC_N(\Gamma_1)\xrightarrow{\phi}\fC_N(\Gamma)$ and $\fC_N(\Gamma)\xrightarrow{\bar{\phi}}\fC_N(\Gamma_1)$ of $\zed_2 \oplus \zed^{\oplus 2}$-degree $(0,0,-1)$. For these two morphisms, we have 
\begin{eqnarray*}
\bar{\phi}\circ \phi & \simeq & 0, \\
\bar{\phi} \circ \mathsf{m}(x_5) \circ \phi & = & -\bar{\phi} \circ \mathsf{m}(x_6) \circ \phi \approx \id_{\fC_N(\Gamma_1)},
\end{eqnarray*}
where $\mathsf{m}(x_5)$ is the endomorphism of $\fC_N(\Gamma)$ given by the multiplication by $x_5$.
\end{lemma}

\begin{proof}
The existence and uniqueness of $\phi$ and $\bar{\phi}$ follow from Lemma \ref{lemma-hmf-decomp-II}. It remains to show that $\phi$ and $\bar{\phi}$ satisfy the equations in the lemma. 

Denote by $\Sym(x_5,x_6) \xrightarrow{\jmath} \Q[x_5,x_6]$ the standard inclusion that maps every element of $\Sym(x_5,x_6)$ to itself in $\Q[x_5,x_6]$. Set $R=R_\partial\otimes_\Q \Sym(x_5,x_6)$ as in the proof of Lemma \ref{lemma-decomp-II}. Using expressions \eqref{eq-proof-lemma-decomp-II-1} and \eqref{eq-proof-lemma-decomp-II-2} of $\fC_N(\Gamma)$ and $\fC_N(\Gamma_1)$, one can see that $\jmath$ induces an $R$-linear homogeneous morphism of matrix factorizations $\fC_N(\Gamma_1)\xrightarrow{\varphi}\fC_N(\Gamma)$ of $\zed_2 \oplus \zed^{\oplus 2}$-degree $(0,0,-1)$. Denote by $\Q[x_5,x_6] \xrightarrow{\pi} \Sym(x_5,x_6)$ the divided difference $\pi(f(x_5,x_6)) = \frac{f(x_5,x_6)-f(x_6,x_5)}{x_5-x_6} \in \Sym(x_5,x_6)$. Note that $\pi$ is $\Sym(x_5,x_6)$-linear. Using expressions \eqref{eq-proof-lemma-decomp-II-1} and \eqref{eq-proof-lemma-decomp-II-2} again, one can see that $\pi$ induces an $R$-linear homogeneous morphism of matrix factorizations $\fC_N(\Gamma)\xrightarrow{\bar{\varphi}}\fC_N(\Gamma_1)$ of $\zed_2 \oplus \zed^{\oplus 2}$-degree $(0,0,-1)$. It is easy to check that $\pi \circ \jmath =0$ and $\pi \circ \mathsf{m}(x_5) \circ \jmath =-\pi \circ \mathsf{m}(x_6) \circ \jmath=\id_{\Sym(x_5,x_6)}$. It follows that $\bar{\varphi}\circ \varphi \simeq 0$ and $\bar{\varphi} \circ \mathsf{m}(x_5) \circ \varphi=-\bar{\varphi} \circ \mathsf{m}(x_6) \circ \varphi \approx \id_{\fC_N(\Gamma_1)}$. In particular, the last equation implies that $\varphi$ and $\bar{\varphi}$ are homotopically non-trivial. Thus, by the uniqueness of $\phi$ and $\bar{\phi}$, we have $\phi \approx \varphi$ and $\bar{\phi} \approx \bar{\varphi}$. This completes the proof.
\end{proof}

\subsection{$\chi$-morphisms} In this subsection, we construct two pairs of $\chi$-morphisms induced by the changes in Figures \ref{def-chi-fig} and \ref{def-tilde-chi-fig}. The morphisms $\chi^0$ and $\chi^1$ are direct generalizations of those defined in \cite{KR1,KR2} and will be used in the definition of the chain complexes associated to closed braids. The morphisms $\tilde{\chi}^0$ and $\tilde{\chi}^1$ will be used in the proof of the invariance of $\fH_N$ under Reidemeister move III. The constructions of these morphisms are given in more general settings in \cite[Section 8]{Wu-color}. 

As in \cite{KR2}, our construction of the $\chi$-morphisms are based on the following simple observation.

\begin{lemma}\cite{KR2}\label{lemma-jumping-factor}
Let $r,s,t$ be homogeneous elements of $R=\Q[a,X_1,\dots,X_k]$ with $\deg r + \deg s + \deg t =(2,2N+2)$. Then there exist homogeneous morphisms of matrix factorizations
\begin{eqnarray*}
f: (r,st)_R \rightarrow (rs,t)_R, && \\
g: (rs,t)_R \rightarrow (r,st)_R, &&
\end{eqnarray*}
such that 
\begin{enumerate}[(i)]
  \item $f$ and $g$ preserve the $\zed_2$-grading, 
  \item $\deg f =(0,0)$ and $\deg g = \deg s$,
	\item $g \circ f = s \cdot \id_{(r,st)_R}$ and $f \circ g = s \cdot \id_{(rs,t)_R}$.
\end{enumerate}
\end{lemma}

\begin{proof}
\[
\xymatrix{
R \ar[rr]^>>>>>>>>>>>>>>>>>{r} \ar[d]^{1} && R\{1-\deg_a r, N+1 - \deg_x r\} \ar[rr]^>>>>>>>>>>>>>>>>>{st} \ar[d]^{s} && R \ar[d]^{1} \\
R \ar[rr]^>>>>>>>>{rs} \ar[d]^{s} && R\{1-\deg_a r -\deg_a s, N+1 - \deg_x r - \deg_x s\} \ar[rr]^>>>>>>>>{t} \ar[d]^{1} && R \ar[d]^{s} \\
R \ar[rr]^>>>>>>>>>>>>>>>>>{r} && R\{1-\deg_a r, N+1 - \deg_x r\} \ar[rr]^>>>>>>>>>>>>>>>>>{st} && R
}
\]
In the above diagram, note that:
\begin{itemize}
	\item The top and bottom rows are both $(r,st)_R$. 
	\item The middle row is $(rs,t)_R$.
	\item All squares commute.
\end{itemize}
So this diagram defines morphisms $(r,st)_R \xrightarrow{f} (rs,t)_R$ and $(rs,t)_R \xrightarrow{g} (r,st)_R$. It is easy to verify that $f$ and $g$ satisfy all the requirements in the lemma.
\end{proof}

\begin{figure}[ht]
$
\xymatrix{
\input{crossing-1-1-res-0} \ar@<8ex>[rr]^{\chi^0} && \input{crossing-1-1-res-1} \ar@<-6ex>[ll]^{\chi^1}
}
$
\caption{}\label{def-chi-fig}

\end{figure}

\begin{lemma}\cite{KR1,KR2}\label{lemma-def-chi}
Let $\Gamma_0$ and $\Gamma_1$ be the MOY graphs in Figure \ref{def-chi-fig}. Then there exist homogeneous morphisms of matrix factorizations $\fC_N(\Gamma_0) \xrightarrow{\chi^0} \fC_N(\Gamma_1)$ and $\fC_N(\Gamma_1) \xrightarrow{\chi^1} \fC_N(\Gamma_0)$ satisfying:
\begin{enumerate}
  \item $\chi^0$ and $\chi^1$ are homotopically non-trivial,
	\item the $\zed_2 \oplus \zed^{\oplus 2}$-degrees of $\chi^0$ and $\chi^1$ are both $(0,0,1)$,
	\item $\chi^1 \circ \chi^0 \simeq (x_2-x_1)\id_{\fC_N(\Gamma_0)}$ and $\chi^0 \circ \chi^1 \simeq (x_2-x_1)\id_{\fC_N(\Gamma_1)}$.
\end{enumerate}
Moreover, up to homotopy and scaling, 
\begin{itemize}
	\item $\chi^0$ is the unique homotopically non-trivial homogeneous morphisms from $\fC_N(\Gamma_0)$ to $\fC_N(\Gamma_1)$ of $\zed_2 \oplus \zed^{\oplus 2}$-degree $(0,0,1)$,
	\item $\chi^1$ is the unique homotopically non-trivial homogeneous morphisms from $\fC_N(\Gamma_1)$ to $\fC_N(\Gamma_0)$ of $\zed_2 \oplus \zed^{\oplus 2}$-degree $(0,0,1)$.
\end{itemize}
\end{lemma}

\begin{proof}
We prove the existence of $\chi^0$ and $\chi^1$ first. Let $R_\partial=\Q[a,x_1,x_2,y_1,y_2]$. By Proposition \ref{prop-b-contraction},
\[
\fC_N(\Gamma_1) \simeq \left(%
\begin{array}{cc}
  aU_1 & x_1+y_1-x_2-y_2 \\
  aU_2 & x_1y_1-x_2y_2
\end{array}%
\right)_{R_\partial}\{0,-1\},
\]
where $U_1$ and $U_2$ are given by equation \eqref{eq-def-U-j}. By Lemma \ref{lemma-row-op},
\[
\left(%
\begin{array}{cc}
  aU_1 & x_1+y_1-x_2-y_2 \\
  aU_2 & x_1y_1-x_2y_2
\end{array}%
\right)_{R_\partial}
\cong 
\left(%
\begin{array}{cc}
  a(U_1+x_1U_2) & x_1+y_1-x_2-y_2 \\
  aU_2 & (x_2-x_1)(x_1-y_2)
\end{array}%
\right)_{R_\partial}.
\]
Thus, we have a pair of homotopy equivalences 
\[
\xymatrix{
\fC_N(\Gamma_1) \ar@<.5ex>[rr]^>>>>>>>>>>{\rho} && {M:=\left(%
\begin{array}{cc}
  a(U_1+x_1U_2) & x_1+y_1-x_2-y_2 \\
  aU_2 & (x_2-x_1)(x_1-y_2)
\end{array}%
\right)_{R_\partial} \{0,-1\}} \ar@<.5ex>[ll]^<<<<<<<<<<{\bar{\rho}}
}
\]
that are homotopy inverses of each other. On the other hand, by Lemmas \ref{lemma-row-op} and \ref{lemma-freedom}, we know that
\[
\fC_N(\Gamma_0) 
\cong \left(%
\begin{array}{cc}
  a(U_1+x_1U_2) & y_1-x_2 \\
  a(U_1+x_2U_2) & x_1-y_2
\end{array}%
\right)_{R_\partial}
\cong \left(%
\begin{array}{cc}
  a(U_1+x_1U_2) & x_1+y_1-x_2-y_2 \\
  a(x_2-x_1)U_2 & x_1-y_2
\end{array}%
\right)_{R_\partial}.
\]
This gives a pair of isomorphisms
\[
\xymatrix{
\fC_N(\Gamma_0) \ar@<.5ex>[rr]^>>>>>>>>>>{\eta} && {M':=\left(%
\begin{array}{cc}
  a(U_1+x_1U_2) & x_1+y_1-x_2-y_2 \\
  a(x_2-x_1)U_2 & x_1-y_2
\end{array}%
\right)_{R_\partial}} \ar@<.5ex>[ll]^<<<<<<<<<<{\eta^{-1}}.
}
\]
By Lemma \ref{lemma-jumping-factor}, there are homogeneous morphisms 
\[
\xymatrix{
{M} \ar@<.5ex>[rr]^<<<<<<<<<<{f} && {M'} \ar@<.5ex>[ll]^<<<<<<<<<<{g}
}
\]
satisfying:
\begin{itemize}
	\item the $\zed_2 \oplus \zed^{\oplus 2}$-degrees of $f$ and $g$ are both $(0,0,1)$,
	\item $f \circ g \simeq (x_2-x_1)\id_{M'}$ and $g \circ f \simeq (x_2-x_1)\id_M$.
\end{itemize}
Moreover, it is straightforward to check that $H_{R_\partial}(M) \xrightarrow{f} H_{R_\partial}(M')$ and $H_{R_\partial}(M') \xrightarrow{g} H_{R_\partial}(M)$ are both non-zero. So $f$ and $g$ are homotopically non-trivial. Define $\chi^0 = \bar{\rho}\circ g \circ \eta$ and $\chi^0 = \eta^{-1}\circ f \circ \rho$. It is easy to verify that these homogeneous morphisms satisfy conditions (1-3) in the lemma.

It remains to prove the uniqueness of $\chi^0$ and $\chi^1$. This comes down to showing that, in the category 
\[
\hmf_{R_\partial, a(x_1^{N+1} +y_1^{N+1} -x_2^{N+1} -y_2^{N+1})},
\] 
we have 
\[
\Hom_\hmf (\fC_N(\Gamma_0),\fC_N(\Gamma_1)\{0,-1\}) \cong \Hom_\hmf (\fC_N(\Gamma_1),\fC_N(\Gamma_0)\{0,-1\}) \cong \Q.
\]
Next, we prove that $\Hom_\hmf (\fC_N(\Gamma_0),\fC_N(\Gamma_1)\{0,-1\}) \cong \Q$. The proof of $\Hom_\hmf (\fC_N(\Gamma_1),\fC_N(\Gamma_0)\{0,-1\}) \cong \Q$ is similar and left to the reader.

Since $\chi^0$ is not homotopic to $0$, we have that $\dim_\Q H^{0,0,1}(\Hom_{R_\partial}(M',M)) \geq 1$. Note that, by Lemma \ref{lemma-Hom-space}, 
\[
\Hom_\hmf (\fC_N(\Gamma_0),\fC_N(\Gamma_1)\{0,-1\}) \cong \Hom_\hmf(M',M\{0,-1\})\cong H^{0,0,1}(\Hom_{R_\partial}(M',M)).
\] 
By Lemma \ref{lemma-dual-Koszul} and Proposition \ref{prop-b-contraction}, we have that, as $\zed_2\oplus\zed^{\oplus2}$-graded matrix factorizations over $R=\Q[a,x_1,y_1]$,
\begin{eqnarray*}
\Hom_{R_\partial}(M',M) & \cong & \left(%
\begin{array}{cc}
  a(U_1+x_1U_2) & x_1+y_1-x_2-y_2 \\
  aU_2 & (x_2-x_1)(x_1-y_2) \\
  a(U_1+x_1U_2) & -(x_1+y_1-x_2-y_2) \\
  a(x_2-x_1)U_2 & -(x_1-y_2)
\end{array}%
\right)_{R_\partial} \{2,2N-3\} \\
& \cong & \left(%
\begin{array}{cc}
  a(V_1+x_1V_2) & 0 \\
  aV_2 & 0 
\end{array}%
\right)_{R_\partial} \{2,2N-3\},
\end{eqnarray*}
where $V_i=U_i|_{y_2=x_1,~x_2=y_1} \in R$. But, as a $\zed_2\oplus\zed^{\oplus2}$-graded $R$-module,
\[
\left(%
\begin{array}{cc}
  a(V_1+x_1V_2) & 0 \\
  aV_2 & 0 
\end{array}%
\right)_{R_\partial} \{2,2N-3\} \cong
(R\{1,N-1\} \oplus R \left\langle 1\right\rangle) \otimes_R (R\{1,N-3\} \oplus R \left\langle 1\right\rangle)\{0,1\}.
\]
This implies that $\dim_\Q H^{0,0,1}(\Hom_{R_\partial}(M',M)) \leq 1$. Thus, $\Hom_\hmf (\fC_N(\Gamma_0),\fC_N(\Gamma_1)\{0,-1\}) \cong \Q$.
\end{proof}

\begin{figure}[ht]
$
\xymatrix{
\input{crossing-1-2-res-0} \ar@<8ex>[rr]^{\tilde{\chi}^0} && \input{crossing-1-2-res-1} \ar@<-6ex>[ll]^{\tilde{\chi}^1}
}
$
\caption{}\label{def-tilde-chi-fig}

\end{figure}

\begin{lemma}\label{lemma-def-tilde-chi}
Let $\tilde{\Gamma}_0$ and $\tilde{\Gamma}_1$ be the MOY graphs in Figure \ref{def-tilde-chi-fig}. Then there exist homogeneous morphisms of matrix factorizations $\fC_N(\tilde{\Gamma}_0) \xrightarrow{\tilde{\chi}^0} \fC_N(\tilde{\Gamma}_1)$ and $\fC_N(\tilde{\Gamma}_1) \xrightarrow{\tilde{\chi}^1} \fC_N(\tilde{\Gamma}_0)$ satisfying:
\begin{enumerate}
  \item $\tilde{\chi}^0$ and $\tilde{\chi}^1$ are homotopically non-trivial,
	\item the $\zed_2 \oplus \zed^{\oplus 2}$-degrees of $\tilde{\chi}^0$ and $\tilde{\chi}^1$ are both $(0,0,1)$,
	\item $\tilde{\chi}^1 \circ \tilde{\chi}^0 \simeq (x_2-x_1)\id_{\fC_N(\tilde{\Gamma}_0)}$ and $\tilde{\chi}^0 \circ \tilde{\chi}^1 \simeq (x_2-x_1)\id_{\fC_N(\tilde{\Gamma}_1)}$.
\end{enumerate}
Moreover, up to homotopy and scaling, 
\begin{itemize}
	\item $\tilde{\chi}^0$ is the unique homotopically non-trivial homogeneous morphisms from $\fC_N(\tilde{\Gamma}_0)$ to $\fC_N(\tilde{\Gamma}_1)$ of $\zed_2 \oplus \zed^{\oplus 2}$-degree $(0,0,1)$,
	\item $\tilde{\chi}^1$ is the unique homotopically non-trivial homogeneous morphisms from $\fC_N(\tilde{\Gamma}_1)$ to $\fC_N(\tilde{\Gamma}_0)$ of $\zed_2 \oplus \zed^{\oplus 2}$-degree $(0,0,1)$.
\end{itemize}
\end{lemma}

\begin{proof}
We follow the approach used in the proof of Lemma \ref{lemma-def-chi}. The only difference is that the computations are now more complex. To simplify the exposition, we introduce the notation ``$\ast_{j,k}$", which means a homogeneous element of $\zed^{\oplus 2}$-degree $(j,k)$. From Lemma \ref{lemma-freedom}, we know that the isomorphism types of all the Koszul matrix factorizations appearing in this proof are independent of the choice of $\ast_{j,k}$ as long as these matrix factorizations remain $\zed_2\oplus\zed^{\oplus 2}$-graded matrix factorizations of $a(x_1^{N+1}+y_1^{N+1}+y_2^{N+1}-x_2^{N+1}-y_3^{N+1}-y_4^{N+1})$.

Let $R_\partial=\Q[a,x_1,x_2]\otimes_\Q \Sym(\{y_1,y_2\}|\{y_3,y_4\})$. By Proposition \ref{prop-b-contraction}, we have 
\[
\fC_N(\tilde{\Gamma}_1) \simeq \left(%
\begin{array}{cc}
  \ast_{2,2N} & x_1+y_1+y_2-x_2-y_3-y_4 \\
  \ast_{2,2N-2} & x_1y_1+x_1y_2+y_1y_2-x_2y_3-x_2y_4-y_3y_4 \\
  \ast_{2,2N-4} & x_1y_1y_2-x_2y_3y_4
\end{array}%
\right)_{R_\partial}\{0,-2\}.
\]
Note that
\begin{eqnarray*}
&& (x_1y_1y_2-x_2y_3y_4) -x_1(x_1y_1+x_1y_2+y_1y_2-x_2y_3-x_2y_4-y_3y_4) + x_1^2(x_1+y_1+y_2-x_2-y_3-y_4) \\
& = & -(x_2-x_1)(x_1^2-x_1(y_3+y_4)+y_3y_4)
\end{eqnarray*}
and 
\begin{eqnarray*}
&& (x_1y_1+x_1y_2+y_1y_2-x_2y_3-x_2y_4-y_3y_4) - x_1(x_1+y_1+y_2-x_2-y_3-y_4) \\
& = & y_1y_2 +(x_2-x_1)(x_1-y_3-y_4)-y_3y_4.
\end{eqnarray*}
So, by Lemma \ref{lemma-twist},
\begin{eqnarray*}
&& {\left(%
\begin{array}{cc}
  \ast_{2,2N} & x_1+y_1+y_2-x_2-y_3-y_4 \\
  \ast_{2,2N-2} & x_1y_1+x_1y_2+y_1y_2-x_2y_3-x_2y_4-y_3y_4 \\
  \ast_{2,2N-4} & x_1y_1y_2-x_2y_3y_4
\end{array}%
\right)_{R_\partial}} \\
& \cong &
{\left(%
\begin{array}{cc}
  \ast_{2,2N} & x_1+y_1+y_2-x_2-y_3-y_4 \\
  \ast_{2,2N-2} & y_1y_2 +(x_2-x_1)(x_1-y_3-y_4)-y_3y_4 \\
  \ast_{2,2N-4} & -(x_2-x_1)(x_1^2-x_1(y_3+y_4)+y_3y_4)
\end{array}%
\right)_{R_\partial}}.
\end{eqnarray*}
Thus, we have a pair of homotopy equivalences 
\[
\xymatrix{
\fC_N(\tilde{\Gamma}_1) \ar@<.5ex>[rr]^>>>>>>>>>>{\rho} && {M:=\left(%
\begin{array}{cc}
  \ast_{2,2N} & x_1+y_1+y_2-x_2-y_3-y_4 \\
  \ast_{2,2N-2} & y_1y_2 +(x_2-x_1)(x_1-y_3-y_4)-y_3y_4 \\
  \ast_{2,2N-4} & -(x_2-x_1)(x_1^2-x_1(y_3+y_4)+y_3y_4)
\end{array}%
\right)_{R_\partial} \{0,-2\}} \ar@<.5ex>[ll]^<<<<<<<<<<{\bar{\rho}}
}
\]
that are homotopy inverses of each other.

By definition, 
\[
\fC_N(\tilde{\Gamma}_0) 
= \left(%
\begin{array}{cc}
  \ast_{2,2N} & x_1+z-y_3-y_4 \\
  \ast_{2,2N-2} & x_1z-y_3y_4 \\
  \ast_{2,2N} & y_1+y_2-x_2-z \\
  \ast_{2,2N-2} & y_1y_2-x_2z
\end{array}%
\right)_{R_\partial[z]}\{0,-1\}.
\]
We remove the indeterminate $z$ by applying Proposition \ref{prop-b-contraction} to the first row of this Koszul matrix factorization. This gives
\begin{eqnarray*}
\fC_N(\tilde{\Gamma}_0) & \simeq & \left(%
\begin{array}{cc}
  \ast_{2,2N-2} & x_1(y_3+y_4-x_1)-y_3y_4 \\
  \ast_{2,2N} & y_1+y_2-x_2-(y_3+y_4-x_1) \\
  \ast_{2,2N-2} & y_1y_2-x_2(y_3+y_4-x_1)
\end{array}%
\right)_{R_\partial}\{0,-1\} \\
& \cong & \left(%
\begin{array}{cc}
  \ast_{2,2N} & y_1+y_2-x_2-(y_3+y_4-x_1) \\
  \ast_{2,2N-2} & y_1y_2-x_2(y_3+y_4-x_1) \\
  \ast_{2,2N-2} & x_1(y_3+y_4-x_1)-y_3y_4
\end{array}%
\right)_{R_\partial}\{0,-1\} \\
& \cong & \left(%
\begin{array}{cc}
  \ast_{2,2N} & x_1+y_1+y_2-x_2-y_3-y_4 \\
  \ast_{2,2N-2} & y_1y_2 +(x_2-x_1)(x_1-y_3-y_4)-y_3y_4 \\
  \ast_{2,2N-2} & -(x_1^2-x_1(y_3+y_4)+y_3y_4)
\end{array}%
\right)_{R_\partial}\{0,-1\},
\end{eqnarray*}
where, in the last step, we applied Lemma \ref{lemma-twist} to the second and third rows. Thus, we have a pair of homotopy equivalences 
\[
\xymatrix{
\fC_N(\tilde{\Gamma}_0) \ar@<.5ex>[rr]^>>>>>>>>>>{\eta} && {M':=\left(%
\begin{array}{cc}
  \ast_{2,2N} & x_1+y_1+y_2-x_2-y_3-y_4 \\
  \ast_{2,2N-2} & y_1y_2 +(x_2-x_1)(x_1-y_3-y_4)-y_3y_4 \\
  \ast_{2,2N-2} & -(x_1^2-x_1(y_3+y_4)+y_3y_4)
\end{array}%
\right)_{R_\partial} \{0,-1\}} \ar@<.5ex>[ll]^<<<<<<<<<<{\bar{\eta}}
}
\]
that are homotopy inverses of each other.

By Lemma \ref{lemma-jumping-factor}, there are homogeneous morphisms 
\[
\xymatrix{
{M} \ar@<.5ex>[rr]^<<<<<<<<<<{f} && {M'} \ar@<.5ex>[ll]^<<<<<<<<<<{g}
}
\]
satisfying
\begin{itemize}
	\item the $\zed_2 \oplus \zed^{\oplus 2}$-degrees of $f$ and $g$ are both $(0,0,1)$,
	\item $f \circ g \simeq (x_2-x_1)\id_{M'}$ and $g \circ f \simeq (x_2-x_1)\id_M$.
\end{itemize}
Moreover, it is straightforward to check that $H_{R_\partial}(M) \xrightarrow{f} H_{R_\partial}(M')$ and $H_{R_\partial}(M') \xrightarrow{g} H_{R_\partial}(M)$ are both non-zero. So $f$ and $g$ are homotopically non-trivial. Define $\tilde{\chi}^0 = \bar{\rho}\circ g \circ \eta$ and $\tilde{\chi}^0 = \bar{\eta}\circ f \circ \rho$. It is easy to see that these are homogeneous morphisms satisfying conditions (1-3) in the lemma. 

The above proves the existence of $\tilde{\chi}^0$ and $\tilde{\chi}^1$. The uniqueness of $\tilde{\chi}^0$ and $\tilde{\chi}^1$ follows from the fact that,in the category 
\[
\hmf_{R_\partial, a(x_1^{N+1} +y_1^{N+1}+y_2^{N+1} -x_2^{N+1} -y_3^{N+1}-y_4^{N+1})},
\] 
we have 
\[
\Hom_\hmf (\fC_N(\tilde{\Gamma}_0),\fC_N(\tilde{\Gamma}_1)\{0,-1\}) \cong \Hom_\hmf (\fC_N(\tilde{\Gamma}_1),\fC_N(\tilde{\Gamma}_0)\{0,-1\}) \cong \Q.
\]
The computations of these $\Hom_\hmf$ spaces are very similar to the corresponding computation in the proof of Lemma \ref{lemma-def-chi}. We leave details to the reader.
\end{proof}

\section{Definition of $\fH_N$}\label{sec-def}

We define in this section a chain complex $\fC_N(B)$ of matrix factorizations for every closed braid $B$. $\fH_N(B)$ is then defined to be the homology of $\fC_N(B)$. Up to a grading shift, $\fC_0(B)$ is the chain complex defined in \cite{KR2} and $\fH_0(B)$ is the HOMFLYPT homology. We will prove in Sections \ref{sec-R1}-\ref{sec-R3} below that, for $N\geq 1$, $\fH_N$ is an invariant for transverse links but not for smooth links.

In this section, we again fix a non-negative integer $N$ and let $a$ be a homogeneous indeterminate of bidegree $\deg a = (2,0)$. 

\subsection{The chain complex associated to a marked tangle diagram} In this subsection, we define the chain complex associated to a marked oriented tangle. Although this chain complex can be defined for any oriented tangle, we can only establish its homotopy invariance under transverse Markov moves. So, in the end, this definition will only be applied to closed braids to define a transverse link invariants via transverse braids. 

\begin{definition}\label{def-marking-tangle}
Let $T$ be an oriented tangle diagram. We call a segment of $T$ between two adjacent crossings/end points an arc. We color all arcs of $T$ by $1$. A marking of $T$ consists of:
\begin{enumerate}
	\item a collections of marked points on $T$ such that
	      \begin{itemize}
	            \item none of the crossings of $T$ are marked, 
	            \item all end points are marked,
	            \item every arc of $T$ contains at least one marked point,
        \end{itemize}
	\item an assignment of pairwise distinct homogeneous indeterminates of bidegree $(0,2)$ to the marked points such that every marked point is assigned a unique indeterminate.
\end{enumerate} 
\end{definition}

Let $T$ be an oriented tangle with a marking. Cut $T$ at all of its marked points. This cuts $T$ into a collection $\{T_1,\dots,T_l\}$ of simple tangles, each of which is of one of the three types in Figure \ref{tangle-pieces-fig} and is marked only at its end points.

\begin{figure}[ht]
$
\xymatrix{
\input{arc} && \input{crossing+} && \input{crossing-} 
}
$
\caption{}\label{tangle-pieces-fig}

\end{figure}

Note that $A$ is itself a MOY graph. We define the chain complex associated to $A$ to be
\begin{equation}\label{eq-def-chain-arc}
\fC_N(A)= 0 \rightarrow \underbrace{\fC_N(A)}_{0} \rightarrow 0, 
\end{equation}
where the $\fC_N(A)$ on the right hand side is the matrix factorization associated to the MOY graph $A$, and the under-brace indicates the homological grading.

\begin{figure}[ht]
$
\xymatrix{
&& \input{crossing+}  \ar@<-12ex>[lld]_{0}\ar@<12ex>[rrd]^{+1} && \\
 \input{crossing-1-1-res-0}&&&& \input{crossing-1-1-res-1} \\
&& \input{crossing-} \ar[llu]_{0} \ar[rru]^{-1} &&
}
$
\caption{}\label{crossing-res-fig}

\end{figure}

To define the chain complexes $\fC_N(C_\pm)$, consider the resolutions of $C_\pm$ in Figure \ref{crossing-res-fig}. We call the resolution $C_\pm \leadsto \Gamma_0$ a $0$-resolution and the resolution $C_\pm \leadsto \Gamma_1$ a $\pm 1$-resolution. We define 
\begin{eqnarray}
\label{eq-def-chain-crossing+} \fC_N(C_+) & = & 0 \rightarrow \underbrace{\fC_N(\Gamma_1)\left\langle 1\right\rangle\{1,N\}}_{-1} \xrightarrow{\chi^1} \underbrace{\fC_N(\Gamma_0)\left\langle 1\right\rangle\{1,N-1\}}_{0} \rightarrow 0, \\
\label{eq-def-chain-crossing-} \fC_N(C_-) & = & 0 \rightarrow \underbrace{\fC_N(\Gamma_0)\left\langle 1\right\rangle\{-1,-N+1\}}_{0} \xrightarrow{\chi^0} \underbrace{\fC_N(\Gamma_1)\left\langle 1\right\rangle\{-1,-N\}}_{1} \rightarrow 0,
\end{eqnarray}
where the morphisms $\chi^0$ and $\chi^1$ are defined in Lemma \ref{lemma-def-chi} and, again, the under-braces indicate the homological gradings.

Note that the differential maps of $\fC_N(A)$, $\fC_N(C_+)$ and $\fC_N(C_-)$ are homogeneous morphisms of matrix factorizations that preserve the $\zed_2 \oplus \zed^{\oplus 2}$-grading. Of course, these differential maps raise the homological grading by $1$.

\begin{definition}\label{def-chain-tangle}
We define the chain complex $\fC_N(T)$ associated to $T$ to be $\fC_N(T) := \bigotimes_{i=1}^{l} \fC_N(T_i)$, where the tensor product is over the common end points. That is, if $x_1,\dots, x_m$ are the indeterminates assigned to the common end points of $T_i$ and $T_j$, then $T_i \otimes T_j := T_i \otimes_{\Q[a,x_1,\dots, x_m]} T_j$.

Suppose that
\begin{itemize}
  \item the indeterminates assigned to end points of $T$ are $y_1,\dots,y_{2n}$,
	\item $T$ points outward at end points assigned with $y_1,\dots,y_n$,
	\item $T$ points inward at end points assigned with $y_{n+1},\dots,y_{2n}$.
\end{itemize}
Define $R_\partial = \Q[a,y_1,\dots,y_{2n}]$. We view $\fC_N(T)$ as a chain complex over the category
\[
\hmf^{\mathrm{all}}_{R_\partial, a(\sum_{i=1}^n y_i^{N+1} - \sum_{j=n+1}^{2n} y_j^{N+1})}.
\]
In particular, for a link diagram $L$, we view $\fC_N(L)$ as a chain complex over the category $\hmf_{\Q[a],0}$.

We denote by $d_\chi$ the differential map of $\fC_N(L)$ induced by the differential maps of the chain complexes \eqref{eq-def-chain-arc}, \eqref{eq-def-chain-crossing+} and \eqref{eq-def-chain-crossing-} via the tensor product. $d_\chi$ is a homogeneous morphism of matrix factorizations preserving the $\zed_2$-, $a$-, $x$-gradings and raising the homological grading by $1$.

The underlying matrix factorization of $\fC_N(T)$ endows it with a homogeneous differential $d_{mf}$, which preserves homological grading and shifts the $\zed_2$-grading by $1$, the $a$-grading by $1$ and the $x$-grading by $N+1$.

$d_\chi$ commutes with $d_{mf}$ since $d_\chi$ is a morphism of matrix factorizations. 
\end{definition}

\begin{definition}\label{def-homology-braid}
Suppose that $B$ is a closed braid with a marking. Then $\fC_N(B)$ is a chain complex over the category $\hmf^{\mathrm{all}}_{\Q[a],0}$. We define $\fH_N(B) := H(H(\fC_N(B),d_{mf}),d_\chi)$. Note that both $d_\chi$ and $d_{mf}$ are homogeneous homomorphisms of $\zed_2 \oplus \zed^{\oplus 3}$-graded $\Q[a]$-modules. So $\fH_N(B)$ inherits the $\zed_2 \oplus \zed^{\oplus 3}$-graded $\Q[a]$-module structure of $\fC_N(B)$, where the $\zed_2$-grading is the $\zed_2$-grading of matrix factorizations and the three $\zed$-gradings are the homological, the $a$- and the $x$-gradings.
\end{definition}

\begin{remark}
Comparing Definitions \ref{def-chain-tangle} and \ref{def-homology-braid} to the corresponding definitions in \cite{KR2}, it is easy to see that, up to a grading shift, $\fH_0$ is the HOMFLYPT homology defined in \cite{KR2}, which is a smooth link invariant. In the current paper, we focus on the case $N \geq 1$ and show that, if $N \geq 1$, then $\fH_N$ is an invariant for transverse links but not for smooth links. 
\end{remark}

\subsection{Markings do not matter} 

\begin{lemma}\label{lemma-marking-independent}
Suppose that $T$ is an oriented tangle diagram with a marking and $T'$ is the same oriented tangle diagram with a different marking. Assume that:
\begin{itemize}
  \item each pair of corresponding end points of $T$ and $T'$ are marked by the same indeterminate,
  \item the indeterminates assigned to end points of $T$ are $y_1,\dots,y_{2n}$,
	\item $T$ points outward at end points assigned with $y_1,\dots,y_n$,
	\item $T$ points inward at end points assigned with $y_{n+1},\dots,y_{2n}$.
\end{itemize}
Then $\fC_N(T) \cong \fC_N(T')$ as chain complexes over the category $\hmf^{\mathrm{all}}_{R_\partial, a(\sum_{i=1}^n y_i^{N+1} - \sum_{j=n+1}^{2n} y_j^{N+1})}$, where $R_\partial = \Q[a,y_1,\dots,y_{2n}]$.
\end{lemma}

\begin{figure}[ht]
$
\xymatrix{
\input{crossing+} && \input{crossing-} \\
\input{crossing+ex} && \input{crossing-ex}
}
$
\caption{}\label{remove-marked-point-fig}

\end{figure}

\begin{proof}
To prove this lemma, we only need to show that adding and removing an extra marked point near a crossing does not change the isomorphism type of the chain complex associated to that crossing. There are four arcs near a crossing. So, in principle, one needs to discuss where the extra marked point is added/removed. Here we prove only that $\fC_N(C_\pm) \cong \fC_N(C_\pm')$ as chain complexes over the category $\hmf_{R,a(x_1^{N+1}+y_1^{N+1}-x_2^{N+1}-y_2^{N+1})}$, where $C_\pm$ and $C_\pm'$ are depicted in Figure \ref{remove-marked-point-fig}, and $R=\Q[a,x_1,x_2,y_1,y_2]$. The proofs for the other cases are very similar and left to the reader.

\begin{figure}[ht]
$
\xymatrix{
\input{crossing-1-1-res-0} && \input{crossing-1-1-res-1} \\
\input{crossing-1-1-res-0ex} && \input{crossing-1-1-res-1ex}
}
$
\caption{}\label{remove-marked-point-proof-fig}

\end{figure}

Consider the marked MOY graphs in Figure \ref{remove-marked-point-proof-fig}. Cutting $\Gamma_0'$ at the pointed marked by $y_3$, we get two marked MOY graphs, $\Gamma_0''$ and $A$, where $A$ is the arc from $y_3$ to $y_1$, and $\Gamma_0''$ is the remainder of $\Gamma_0'$. Then $\fC_N(\Gamma_0') \cong \fC_N(\Gamma_0'')\otimes_{\Q[a,y_3]} \fC_N(A)$. Similarly, cutting $\Gamma_1'$ at the pointed marked by $y_3$, we get $\Gamma_1'=\Gamma_1''\cup A$ and $\fC_N(\Gamma_1') \cong \fC_N(\Gamma_1'')\otimes_{\Q[a,y_3]} \fC_N(A)$. 

By Proposition \ref{prop-b-contraction}, we know:
\begin{itemize}
	\item there are a pair of homotopy equivalences $\xymatrix{\fC_N(\Gamma_0'')\otimes_{\Q[a,y_3]} \fC_N(A) \ar@<.5ex>[r]^>>>>>{\bar{f}}& \fC_N(\Gamma_0)\ar@<.5ex>[l]^<<<<<{f}}$ of $\zed_2\oplus\zed^{\oplus 2}$-graded matrix factorizations that are homotopy inverses of each other,
	\item there are a pair of homotopy equivalences $\xymatrix{\fC_N(\Gamma_1'')\otimes_{\Q[a,y_3]} \fC_N(A) \ar@<.5ex>[r]^>>>>>{\bar{g}}& \fC_N(\Gamma_1)\ar@<.5ex>[l]^<<<<<{g}}$ of $\zed_2\oplus\zed^{\oplus 2}$-graded matrix factorizations that are homotopy inverses of each other.
\end{itemize}
Let $\xymatrix{\fC_N(\Gamma_0'') \ar@<.5ex>[r]^{(\chi^0)''}& \fC_N(\Gamma_1'')\ar@<.5ex>[l]^{(\chi^1)''}}$ be the $\chi$-morphisms associated to $\Gamma_0''$ and $\Gamma_1''$ defined in Lemma \ref{lemma-def-chi}. Using the definitions of $(\chi^0)''$, $(\chi^1)''$ and Proposition \ref{prop-contraction-weak}, one can check that 
\[
\xymatrix{H_R(\fC_N(\Gamma_0'')\otimes_{\Q[a,y_3]} \fC_N(A)) \ar@<.5ex>[r]^{(\chi^0)''\otimes\id}& H_R(\fC_N(\Gamma_1'')\otimes_{\Q[a,y_3]} \fC_N(A))\ar@<.5ex>[l]^{(\chi^1)''\otimes\id}}
\] 
are non-zero homomorphisms. Thus, $(\chi^0)''\otimes\id$ and $(\chi^1)''\otimes\id$ are homotopically non-trivial. This implies that that morphisms 
\[
\xymatrix{
\fC_N(\Gamma_0) \ar@<.5ex>[rr]^{\bar{g}\circ((\chi^0)''\otimes\id)\circ f} && \fC_N(\Gamma_1)\ar@<.5ex>[ll]^{\bar{f}\circ((\chi^1)''\otimes\id)\circ g}}
\]
are homotopically non-trivial. Note that these are homogeneous morphisms of $\zed_2\oplus\zed^{\oplus 2}$-degree $(0,0,1)$. Let $\xymatrix{\fC_N(\Gamma_0) \ar@<.5ex>[r]^{\chi^0}& \fC_N(\Gamma_1)\ar@<.5ex>[l]^{\chi^1}}$ be the $\chi$-morphisms associated to $\Gamma_0$ and $\Gamma_1$ defined in Lemma \ref{lemma-def-chi}. By Lemma \ref{lemma-def-chi}, we know that, up to homotopy and scaling, $\chi^0$ and $\chi^1$ are the unique homotopically non-trivial homogeneous morphisms between $\fC_N(\Gamma_0)$ and $\fC_N(\Gamma_1)$ with $\zed_2\oplus\zed^{\oplus 2}$-degree $(0,0,1)$.
Thus, there exist non-zero scalars $c_0,c_1 \in \Q$ such that $\bar{g}\circ((\chi^0)''\otimes\id)\circ f \simeq c_0 \chi^0$ and $\bar{f}\circ((\chi^1)''\otimes\id)\circ g\simeq c_1 \chi^1$.

The above shows that, over the category $\hmf_{R,a(x_1^{N+1}+y_1^{N+1}-x_2^{N+1}-y_2^{N+1})}$, we have commutative diagrams 
\[
\xymatrix{
0 \ar[r] & \fC_N(\Gamma_1'')\otimes_{\Q[a,y_3]} \fC_N(A)\left\langle 1\right\rangle\{1,N\} \ar[rr]^{(\chi^1)''\otimes\id} \ar[d]^{c_1\bar{g}} && \fC_N(\Gamma_0'')\otimes_{\Q[a,y_3]} \fC_N(A)\left\langle 1\right\rangle\{1,N-1\} \ar[r] \ar[d]^{\bar{f}} & 0 \\
0 \ar[r] & \fC_N(\Gamma_1)\left\langle 1\right\rangle\{1,N\}\ar[rr]^{\chi^1}  && \fC_N(\Gamma_0)\left\langle 1\right\rangle\{1,N-1\} \ar[r] & 0
}
\]
and 
\[
\xymatrix{
0 \ar[r]& \fC_N(\Gamma_0'')\otimes_{\Q[a,y_3]} \fC_N(A)\left\langle 1\right\rangle\{-1,-N+1\} \ar[rr]^{(\chi^0)''\otimes\id} \ar[d]^{c_0\bar{f}} && \fC_N(\Gamma_1'')\otimes_{\Q[a,y_3]} \fC_N(A))\left\langle 1\right\rangle\{-1,-N\} \ar[r] \ar[d]^{\bar{g}}& 0 \\
0 \ar[r]& \fC_N(\Gamma_0)\left\langle 1\right\rangle\{-1,-N+1\} \ar[rr]^{\chi^0}  && \fC_N(\Gamma_1)\left\langle 1\right\rangle\{-1,-N\} \ar[r] & 0
}
\]

Since $\bar{f}$ and $\bar{g}$ are homotopy equivalences of $\zed_2\oplus\zed^{\oplus 2}$-graded matrix factorizations, the above diagram give isomorphisms $\fC_N(C_\pm) \cong \fC_N(C_\pm')$ for chain complexes over the category $\hmf_{R,a(x_1^{N+1}+y_1^{N+1}-x_2^{N+1}-y_2^{N+1})}$.
\end{proof}

\section{Reidemeister Move I}\label{sec-R1}

In this section, we prove the invariance of $\fH_N$ under positive stabilizations/de-stabilizations and establish the long exact sequence induced by a negative stabilization. 

\subsection{Algebraic lemmas} First, we recall the Gaussian Elimination Lemma in \cite{Bar-fast}, which will be used frequently in the proof of the invariance of $\fH_N$.

\begin{lemma}\cite[Lemma 4.2]{Bar-fast}\label{lemma-gaussian-elimination}
Let $\mathscr{C}$ be an additive category, and
\[
\mathtt{I}=\cdots\rightarrow C\xrightarrow{\left(%
\begin{array}{c}
  \alpha\\
  \beta \\
\end{array}%
\right)}
\left.%
\begin{array}{c}
  A\\
  \oplus \\
  D
\end{array}%
\right.
\xrightarrow{
\left(%
\begin{array}{cc}
  \phi & \delta\\
  \gamma & \varepsilon \\
\end{array}%
\right)}
\left.%
\begin{array}{c}
  B\\
  \oplus \\
  E
\end{array}%
\right.
\xrightarrow{
\left(%
\begin{array}{cc}
  \mu & \nu\\
\end{array}%
\right)} F \rightarrow \cdots
\]
a chain complex over $\mathscr{C}$. Assume that $A\xrightarrow{\phi} B$ is an isomorphism in $\mathscr{C}$ with inverse $\phi^{-1}$. Then $\mathtt{I}$ is homotopic to 
\[
\mathtt{II}=
\cdots\rightarrow C \xrightarrow{\beta} D
\xrightarrow{\varepsilon-\gamma\phi^{-1}\delta} E\xrightarrow{\nu} F \rightarrow \cdots.
\]
We call $-\gamma\phi^{-1}\delta$ the correction term in the differential. If $\delta$ or $\gamma$ is $0$, then the correction term is $0$ and $\mathtt{I}$ is homotopic to 
\[
\mathtt{II}=
\cdots\rightarrow C \xrightarrow{\beta} D
\xrightarrow{\varepsilon} E\xrightarrow{\nu} F \rightarrow \cdots.
\]
\end{lemma}

Next, we give a lemma about commutativity of certain morphisms between Koszul matrix factorizations that will allow us to simplify morphisms $\chi^0$ and $\chi^1$ and prove the invariance of $\fH_N$ under positive stabilization.

\begin{lemma}\label{lemma-chi-twist-commute}
Let $R=\Q[a,X_1,\dots,X_k]$ be a $\zed^{\oplus 2}$-graded polynomial ring with $\deg a = (2,0)$ and $\deg X_i = (0,2n_i)$, where $n_i$ is a positive integer. Suppose that $u,v,x,y,z,p$ are homogeneous elements of $R$ such that
\begin{itemize}
	\item $\deg u + \deg v = (2,2N+2)$,
	\item $\deg x + \deg y + \deg z = (2,2N+2)$,
	\item $\deg p = \deg x - \deg v = \deg u -\deg y - \deg z$.
\end{itemize}
Consider the $\zed_2\oplus\zed^{\oplus 2}$-graded matrix factorizations 
\[
\xymatrix{
{M_0=\left(%
\begin{array}{cc}
  u & v \\
  xy & z
\end{array}%
\right)_R,} && {M_0'=\left(%
\begin{array}{cc}
  u +pyz& v \\
  xy-ypv & z
\end{array}%
\right)_R,} \\
{M_1=\left(%
\begin{array}{cc}
  u & v \\
  x & yz
\end{array}%
\right)_R,} && {M_1'=\left(%
\begin{array}{cc}
  u + pyz& v \\
  x-pv & yz
\end{array}%
\right)_R.}
}
\]
Denote by $\xymatrix{M_0 \ar@<0.5ex>[r]^{g} & M_1\ar@<0.5ex>[l]^{f}}$ and $\xymatrix{M_0' \ar@<0.5ex>[r]^{g'} & M_1'\ar@<0.5ex>[l]^{f'}}$ the morphisms obtained by applying Lemma \ref{lemma-jumping-factor} to the second rows of these Koszul matrix factorizations. Let $M_0 \xrightarrow{\psi_0} M_0'$ and $M_1 \xrightarrow{\psi_1} M_1'$ be the isomorphisms from Lemma \ref{lemma-twist}. Then the following diagrams commute.
\[
\xymatrix{
M_0 \ar[d]_{g} \ar[r]^{\psi_0} & M_0'\ar[d]^{g'}  &&&& M_1 \ar[d]_{f} \ar[r]^{\psi_1} & M_1' \ar[d]^{f'} \\
M_1 \ar[r]^{\psi_1}& M_1' &&&& M_0 \ar[r]^{\psi_0}& M_0'
}
\]
\end{lemma}

\begin{proof}
This lemma becomes fairly obvious once we explicitly write down these matrix factorizations and morphisms. For simplicity, we omit the grading shifts in this proof. 

Note that
\begin{eqnarray*}
M_0 & = & {\left.%
\begin{array}{c}
  R_{00} \\
  \oplus \\
  R_{11}
\end{array}%
\right.} \xrightarrow{\left(%
\begin{array}{cc}
  u & -z \\
  xy & v
\end{array}%
\right)}
{\left.%
\begin{array}{c}
  R_{10} \\
  \oplus \\
  R_{01}
\end{array}%
\right.}
\xrightarrow{\left(%
\begin{array}{cc}
  v & z \\
  -xy & u
\end{array}%
\right)}
{\left.%
\begin{array}{c}
  R_{00} \\
  \oplus \\
  R_{11}
\end{array}%
\right.}, \\
M_1 & = & {\left.%
\begin{array}{c}
  R_{00} \\
  \oplus \\
  R_{11}
\end{array}%
\right.} \xrightarrow{\left(%
\begin{array}{cc}
  u & -yz \\
  x & v
\end{array}%
\right)}
{\left.%
\begin{array}{c}
  R_{10} \\
  \oplus \\
  R_{01}
\end{array}%
\right.}
\xrightarrow{\left(%
\begin{array}{cc}
  v & yz \\
  -x & u
\end{array}%
\right)}
{\left.%
\begin{array}{c}
  R_{00} \\
  \oplus \\
  R_{11}
\end{array}%
\right.}, \\
M_0' & = & {\left.%
\begin{array}{c}
  R_{00} \\
  \oplus \\
  R_{11}
\end{array}%
\right.} \xrightarrow{\left(%
\begin{array}{cc}
  u+pyz & -z \\
  xy-ypv & v
\end{array}%
\right)}
{\left.%
\begin{array}{c}
  R_{10} \\
  \oplus \\
  R_{01}
\end{array}%
\right.}
\xrightarrow{\left(%
\begin{array}{cc}
  v & z \\
  -xy+ypv & u+pyz
\end{array}%
\right)}
{\left.%
\begin{array}{c}
  R_{00} \\
  \oplus \\
  R_{11}
\end{array}%
\right.}, \\
M_1' & = & {\left.%
\begin{array}{c}
  R_{00} \\
  \oplus \\
  R_{11}
\end{array}%
\right.} \xrightarrow{\left(%
\begin{array}{cc}
  u+pyz & -yz \\
  x-pv & v
\end{array}%
\right)}
{\left.%
\begin{array}{c}
  R_{10} \\
  \oplus \\
  R_{01}
\end{array}%
\right.}
\xrightarrow{\left(%
\begin{array}{cc}
  v & yz \\
  -x+pv & u+pyz
\end{array}%
\right)}
{\left.%
\begin{array}{c}
  R_{00} \\
  \oplus \\
  R_{11}
\end{array}%
\right.},
\end{eqnarray*}
where each $R_{\ve\delta}$ is a copy of $R$ and the lower indices are just for keeping track of the position of each component. Using these explicit matrix factorizations, we have the following.
\begin{itemize}
	\item The morphisms $\xymatrix{M_0 \ar@<0.5ex>[r]^{g} & M_1\ar@<0.5ex>[l]^{f}}$ are given by 
	\[
	\xymatrix{
	{\left.%
\begin{array}{c}
  R_{00} \\
  \oplus \\
  R_{11}
\end{array}%
\right.} \ar@<0.5ex>[rr]^{G_0} && {\left.%
\begin{array}{c}
  R_{00} \\
  \oplus \\
  R_{11}
\end{array}%
\right.} \ar@<0.5ex>[ll]^{F_0} && \text{and} && 
{\left.%
\begin{array}{c}
  R_{10} \\
  \oplus \\
  R_{01}
\end{array}%
\right.} \ar@<0.5ex>[rr]^{G_1} && {\left.%
\begin{array}{c}
  R_{10} \\
  \oplus \\
  R_{01}
\end{array}%
\right.} \ar@<0.5ex>[ll]^{F_1},
	}
	\]
	where, as matrices over $R$, 
	\[
	G_0 = G_1 = {\left(%
  \begin{array}{cc}
  y & 0 \\
  0 & 1
  \end{array}%
  \right)}, \hspace{5pc}
  F_0 = F_1 = {\left(%
  \begin{array}{cc}
  1 & 0 \\
  0 & y
\end{array}%
\right)}.
	\]
	\item The morphisms $\xymatrix{M_0' \ar@<0.5ex>[r]^{g'} & M_1'\ar@<0.5ex>[l]^{f'}}$ are given by 
	\[
	\xymatrix{
	{\left.%
\begin{array}{c}
  R_{00} \\
  \oplus \\
  R_{11}
\end{array}%
\right.} \ar@<0.5ex>[rr]^{G_0'} && {\left.%
\begin{array}{c}
  R_{00} \\
  \oplus \\
  R_{11}
\end{array}%
\right.} \ar@<0.5ex>[ll]^{F_0'} && \text{and} && 
{\left.%
\begin{array}{c}
  R_{10} \\
  \oplus \\
  R_{01}
\end{array}%
\right.} \ar@<0.5ex>[rr]^{G_1'} && {\left.%
\begin{array}{c}
  R_{10} \\
  \oplus \\
  R_{01}
\end{array}%
\right.} \ar@<0.5ex>[ll]^{F_1'},
	}
	\]
	where, as matrices over $R$, 
	\[
	G_0' = G_1' = {\left(%
  \begin{array}{cc}
  y & 0 \\
  0 & 1
  \end{array}%
  \right)}, \hspace{5pc}
  F_0' = F_1' = {\left(%
  \begin{array}{cc}
  1 & 0 \\
  0 & y
\end{array}%
\right)}.
	\]
	\item The isomorphism $M_0 \xrightarrow{\psi_0} M_0'$ is given by 
	\[
	\xymatrix{
	{\left.%
\begin{array}{c}
  R_{00} \\
  \oplus \\
  R_{11}
\end{array}%
\right.} \ar[rr]^{\Psi_{00}} && {\left.%
\begin{array}{c}
  R_{00} \\
  \oplus \\
  R_{11}
\end{array}%
\right.}  && \text{and} &&
{\left.%
\begin{array}{c}
  R_{10} \\
  \oplus \\
  R_{01}
\end{array}%
\right.} \ar[rr]^{\Psi_{01}} && {\left.%
\begin{array}{c}
  R_{10} \\
  \oplus \\
  R_{01}
\end{array}%
\right.},
	}
	\]
	where, as matrices over $R$, 
	\[
	\Psi_{00} = {\left(%
  \begin{array}{cc}
  1 & 0 \\
  py & 1
  \end{array}%
  \right)}, \hspace{5pc}
  \Psi_{01} = {\left(%
  \begin{array}{cc}
  1 & 0 \\
  0 & 1
\end{array}%
\right)}.
	\]
	\item The isomorphism $M_1 \xrightarrow{\psi_1} M_1'$ is given by 
	\[
	\xymatrix{
	{\left.%
\begin{array}{c}
  R_{00} \\
  \oplus \\
  R_{11}
\end{array}%
\right.} \ar[rr]^{\Psi_{10}} && {\left.%
\begin{array}{c}
  R_{00} \\
  \oplus \\
  R_{11}
\end{array}%
\right.}  && \text{and} &&
{\left.%
\begin{array}{c}
  R_{10} \\
  \oplus \\
  R_{01}
\end{array}%
\right.} \ar[rr]^{\Psi_{11}} && {\left.%
\begin{array}{c}
  R_{10} \\
  \oplus \\
  R_{01}
\end{array}%
\right.},
	}
	\]
	where, as matrices over $R$, 
	\[
	\Psi_{10} = {\left(%
  \begin{array}{cc}
  1 & 0 \\
  p & 1
  \end{array}%
  \right)}, \hspace{5pc}
  \Psi_{11} = {\left(%
  \begin{array}{cc}
  1 & 0 \\
  0 & 1
\end{array}%
\right)}.
	\]
\end{itemize}

It is easy to verify that $\Psi_{1\ve}G_\ve = G_\ve'\Psi_{0\ve}$ and $\Psi_{0\ve}F_\ve = F_\ve'\Psi_{1\ve}$ for $\ve=0,1$. This proves the lemma.
\end{proof}

\subsection{A closer look at the $\chi$-morphisms} Let us recall the definitions of $\chi^0$ and $\chi^1$. From the proof of Lemma \ref{lemma-def-chi}, we know that, for the MOY graphs $\Gamma_0$ and $\Gamma_1$ in Figure \ref{def-chi-fig}, 
\begin{eqnarray*}
\fC_N(\Gamma_0) 
& \cong & {\left(%
\begin{array}{cc}
  a(U_1+x_1U_2) & x_1+y_1-x_2-y_2 \\
  a(x_2-x_1)U_2 & x_1-y_2
\end{array}%
\right)_{R_\partial}}, \\
\fC_N(\Gamma_1)\{0,1\} & \simeq & {\left(%
\begin{array}{cc}
  a(U_1+x_1U_2) & x_1+y_1-x_2-y_2 \\
  aU_2 & (x_2-x_1)(x_1-y_2)
\end{array}%
\right)_{R_\partial}},
\end{eqnarray*}
where $R_\partial=\Q[a,x_1,x_2,y_1,y_2]$, and 
\begin{eqnarray*}
U_1 & = & \frac{p_{2,N+1}(x_1+y_1,x_1y_1)-p_{2,N+1}(x_2+y_2,x_1y_1)}{x_1+y_1-x_2-y_2}, \\
U_2 & = & \frac{p_{2,N+1}(x_2+y_2,x_1y_1)-p_{2,N+1}(x_2+y_2,x_2y_2)}{x_1y_1-x_2y_2},
\end{eqnarray*}
with the polynomial $p_{2,N+1}$ given by equation \eqref{eq-def-poly-p}. Under the above homotopy equivalences, the morphisms $\chi^0$ and $\chi^1$ are defined by applying Lemma \ref{lemma-jumping-factor} to the second rows in the Koszul matrix factorizations on the right hand side. 

\begin{figure}[ht]
$
\xymatrix{
\input{crossing-1-1-res-0-closed} \ar@<8ex>[rr]^{\chi^0} && \input{crossing-1-1-res-1-closed} \ar@<-6ex>[ll]^{\chi^1}
}
$
\caption{}\label{kink-res-fig}

\end{figure}

Now consider the MOY graphs $\hat{\Gamma}_0$ and $\hat{\Gamma}_1$ in Figure \ref{kink-res-fig}. From the above discussion, we have the following lemma.

\begin{lemma}\label{lemma-RI-chi}
\begin{eqnarray*}
\fC_N(\hat{\Gamma}_0) 
& \cong & {\left(%
\begin{array}{cc}
  a(V_1+x_1V_2) & y_1-x_2 \\
  a(x_2-x_1)V_2 & 0
\end{array}%
\right)_{R}}, \\
\fC_N(\hat{\Gamma}_1)\{0,1\} & \simeq & {\left(%
\begin{array}{cc}
  a(V_1+x_1V_2) & y_1-x_2 \\
  aV_2 & 0
\end{array}%
\right)_{R}},
\end{eqnarray*}
where $R=\Q[a,x_1,x_2,y_1]$ and $V_1=U_1|_{y_2=x_1}$, $V_2=U_2|_{y_2=x_1}$. Moreover, the morphisms $\xymatrix{\fC_N(\hat{\Gamma}_0) \ar@<0.5ex>[r]^{\chi^0} & \fC_N(\hat{\Gamma}_1)\ar@<0.5ex>[l]^{\chi^1}}$ are obtained by applying Lemma \ref{lemma-jumping-factor} to the second rows in the Koszul matrix factorizations on the right hand side. 
\end{lemma}

Note that $\fC_N(\hat{\Gamma}_0)$ and $\fC_N(\hat{\Gamma}_1)$ are matrix factorizations of $a(y_1^{N+1}-x_2^{N+1})$. So 
\[
V_1+x_1V_2 = h_N(y_1,x_2):= \frac{y_1^{N+1}-x_2^{N+1}}{y_1-x_2}.
\] 
Also, 
\[
V_2=U_2|_{y_2=x_1} = U_2|_{y_2=x_1, y_1=x_2} + (y_1-x_2)p
\]
for some homogeneous element $p$ of $R$ of bidegree $(0,2N-4)$. By Lemma \ref{lemma-power-derive}, we have that 
\[
U_2|_{y_2=x_1, y_1=x_2} = h_{N-1}(x_1,x_2):= \frac{x_1^{N}-x_2^{N}}{x_1-x_2}.
\]
Thus, we have the following lemma.

\begin{lemma}\label{lemma-RI-chi-explicit}
\begin{eqnarray*}
\fC_N(\hat{\Gamma}_0) 
& \cong & {M_0:=\left(%
\begin{array}{cc}
  ah_N(y_1,x_2) & y_1-x_2 \\
  a(x_2-x_1)h_{N-1}(x_1,x_2) & 0
\end{array}%
\right)_{R}}, \\
\fC_N(\hat{\Gamma}_1)\{0,1\} & \simeq & {M_1:=\left(%
\begin{array}{cc}
  ah_N(y_1,x_2) & y_1-x_2 \\
  ah_{N-1}(x_1,x_2) & 0
\end{array}%
\right)_{R}},
\end{eqnarray*}
where $R=\Q[a,x_1,x_2,y_1]$. Denote by $\xymatrix{M_0 \ar@<0.5ex>[r]^{g} & M_1\ar@<0.5ex>[l]^{f}}$ the morphisms obtained by applying Lemma \ref{lemma-jumping-factor} to the second rows of these matrix factorizations. Then there are commutative diagrams
\[
\xymatrix{
\fC_N(\hat{\Gamma}_0) \ar[d]_{\chi^0} \ar[r]^{\cong} & M_0\ar[d]^{g}  \\
\fC_N(\hat{\Gamma}_1)\{0,1\} \ar[r]^>>>>>{\simeq}& M_1
}
\hspace{2pc}\text{ and  }\hspace{2pc}
\xymatrix{
\fC_N(\hat{\Gamma}_1)\{0,1\} \ar[d]_{\chi^1} \ar[r]^>>>>>{\simeq} & M_1 \ar[d]^{f} \\
\fC_N(\hat{\Gamma}_0) \ar[r]^{\cong}& M_0
}.
\]
\end{lemma}

\begin{proof}
This lemma follows from Lemmas \ref{lemma-chi-twist-commute} and \ref{lemma-RI-chi}.
\end{proof}

\subsection{Positive stabilization} We prove in this subsection the invariance of $\fH_N$ under positive stabilizations/de-stabilizations. The main result of this subsection is Proposition \ref{prop-RI+inv}.

\begin{figure}[ht]
$
\xymatrix{
\input{curve}  && \input{kink+}
}
$
\caption{}\label{stab+fig}

\end{figure}

\begin{proposition}\label{prop-RI+inv}
Let $T$ and $T_+$ be the tangles in Figure \ref{stab+fig}. Then, for $N \geq 1$, $\fC_N(T) \simeq \fC_N(T_+)$ as chain complexes over the category $\hmf^{\mathrm{all}}_{\Q[a,y_1,x_2],a(y_1^{N+1}-x_2^{N+1})}$.
\end{proposition}

\begin{proof}
Let $\hat{\Gamma}_0$ and $\hat{\Gamma}_1$ be the MOY graphs in Figure \ref{kink-res-fig}. Set $R=\Q[a,x_1,x_2,y_1]$ and $R_\partial=\Q[a,x_2,y_1]$. From Lemma \ref{lemma-RI-chi-explicit}, we know that 
\begin{eqnarray}
\label{eq-RI-T+} \fC_N(T_+) & = & 0 \rightarrow \underbrace{\fC_N(\hat{\Gamma}_1)\left\langle 1\right\rangle\{1,N\}}_{-1} \xrightarrow{\chi^1} \underbrace{\fC_N(\hat{\Gamma}_0)\left\langle 1\right\rangle\{1,N-1\}}_{0} \rightarrow 0 \\
\nonumber & \cong & 0 \rightarrow \underbrace{M_1\left\langle 1\right\rangle\{1,N-1\}}_{-1} \xrightarrow{f} \underbrace{M_0\left\langle 1\right\rangle\{1,N-1\}}_{0} \rightarrow 0,
\end{eqnarray}
where 
\begin{eqnarray}
\label{eq-R1-def-M0} M_0 & = & {\left(%
\begin{array}{cc}
  ah_N(y_1,x_2) & y_1-x_2 \\
  a(x_2-x_1)h_{N-1}(x_1,x_2) & 0
\end{array}%
\right)_{R}}, \\
\label{eq-R1-def-M1} M_1 & = & {\left(%
\begin{array}{cc}
  ah_N(y_1,x_2) & y_1-x_2 \\
  ah_{N-1}(x_1,x_2) & 0
\end{array}%
\right)_{R}},
\end{eqnarray}
and $f$ is the morphism obtained by applying Lemma \ref{lemma-jumping-factor} to the second rows of $M_0$ and $M_1$. 

We prove the proposition by simplifying the chain complex 
\[
0 \rightarrow \underbrace{M_1}_{-1} \xrightarrow{f} \underbrace{M_0}_{0} \rightarrow 0.
\]

Note that:
\begin{equation}\label{eq-kink-res-0-ex}
M_0 =  {\left.%
\begin{array}{c}
  R_{00} \\
  \oplus \\
  R_{11}\{-2,2-2N\}
\end{array}%
\right.} \xrightarrow{D_0}
{\left.%
\begin{array}{c}
  R_{10}\{-1,1-N\} \\
  \oplus \\
  R_{01}\{-1,1-N\}
\end{array}%
\right.}
\xrightarrow{D_1}
{\left.%
\begin{array}{c}
  R_{00} \\
  \oplus \\
  R_{11}\{-2,2-2N\}
\end{array}%
\right.},
\end{equation}
where 
\begin{eqnarray*}
D_0 & = & {\left(%
\begin{array}{cc}
  ah_N(y_1,x_2) & 0 \\
  a(x_2-x_1)h_{N-1}(x_1,x_2) & y_1-x_2
\end{array}%
\right)}, \\
D_1 & = & {\left(%
\begin{array}{cc}
  y_1-x_2 & 0 \\
  -a(x_2-x_1)h_{N-1}(x_1,x_2) & ah_N(y_1,x_2)
\end{array}%
\right)},
\end{eqnarray*}
and
\begin{equation}\label{eq-kink-res-1-ex}
M_1 =  {\left.%
\begin{array}{c}
  R_{00} \\
  \oplus \\
  R_{11}\{-2,4-2N\}
\end{array}%
\right.} \xrightarrow{\Delta_0}
{\left.%
\begin{array}{c}
  R_{10}\{-1,1-N\} \\
  \oplus \\
  R_{01}\{-1,3-N\}
\end{array}%
\right.}
\xrightarrow{\Delta_1}
{\left.%
\begin{array}{c}
  R_{00} \\
  \oplus \\
  R_{11}\{-2,4-2N\}
\end{array}%
\right.},
\end{equation}
where 
\begin{eqnarray*}
\Delta_0 & = & {\left(%
\begin{array}{cc}
  ah_N(y_1,x_2) & 0 \\
  ah_{N-1}(x_1,x_2) & y_1-x_2
\end{array}%
\right)}, \\
\Delta_1 & = & {\left(%
\begin{array}{cc}
  y_1-x_2 & 0 \\
  -ah_{N-1}(x_1,x_2) & ah_N(y_1,x_2)
\end{array}%
\right)}.
\end{eqnarray*}
Moreover, the morphism $f:M_1 \rightarrow M_0$ is given by
\begin{eqnarray*}
{\left.%
\begin{array}{c}
  R_{00} \\
  \oplus \\
  R_{11}\{-2,4-2N\}
\end{array}%
\right.} & \xrightarrow{\left(%
\begin{array}{cc}
  1 & 0 \\
  0 & x_2-x_1
\end{array}%
\right)} &  {\left.%
\begin{array}{c}
  R_{00} \\
  \oplus \\
  R_{11}\{-2,2-2N\}
\end{array}%
\right.}, \\
{\left.%
\begin{array}{c}
  R_{10}\{-1,1-N\} \\
  \oplus \\
  R_{01}\{-1,3-N\}
\end{array}%
\right.}
& \xrightarrow{\left(%
\begin{array}{cc}
  1 & 0 \\
  0 & x_2-x_1
\end{array}%
\right)} &
{\left.%
\begin{array}{c}
  R_{10}\{-1,1-N\} \\
  \oplus \\
  R_{01}\{-1,1-N\}
\end{array}%
\right.}.
\end{eqnarray*}
In the above explicit forms, each $R_{\ve\delta}$ is a copy of $R$. The lower indices are just for keeping track of the position of each component.

Next, we decompose the matrix factorizations $M_0$ and $M_1$ over $R_\partial$. First, for $M_0$, denote by $u^{[k]}_{\ve\delta}$ the homogeneous element $(x_2-x_1)^k$ in $R_{\ve\delta}$ and by $v^{[k]}_{\ve\delta}$ the homogeneous element $(x_2-x_1)^k h_{N-1}(x_1,x_2)$ in $R_{\ve\delta}$. Let
\begin{eqnarray*}
B_{00} & = & \{u^{[k]}_{00}~|~ k\geq 0\}, \\
B_{10} & = & \{u^{[k]}_{10}~|~ k\geq 0\}, \\
B_{01} & = & \{u^{[k]}_{01}~|~ 0\leq k\leq N-1\}\cup \{v^{[k+1]}_{01}~|~ k\geq 0\}, \\
B_{11} & = & \{u^{[k]}_{11}~|~ 0\leq k\leq N-1\}\cup \{v^{[k+1]}_{11}~|~ k\geq 0\}.
\end{eqnarray*}
It is easy to check that $B_{\ve\delta}$ is a homogeneous $R_\partial$-basis for $R_{\ve\delta}$. So 
\begin{equation}\label{eq-basis-M0}
B:=B_{00}\cup B_{10}\cup B_{01}\cup B_{11} 
\end{equation}
is a homogeneous $R_\partial$-basis for $M_0$. Define 
\begin{eqnarray}
\label{eq-RI-Theta} \Theta_k & = & \mathrm{span}_{R_\partial}\{u^{[k]}_{00}, u^{[k]}_{10}, v^{[k+1]}_{01}, v^{[k+1]}_{11}\} \text{ for } k \geq 0,\\
\label{eq-RI-Omega} \Omega_k & = & \mathrm{span}_{R_\partial}\{u^{[k]}_{01}, u^{[k]}_{11}\} \text{ for } 0\leq k \leq N-1.
\end{eqnarray}
It is straightforward to check that the differential map of $M_0$ preserves $\Theta_k$ and $\Omega_k$. So $\Theta_k$ and $\Omega_k$ are $\zed_2\oplus\zed^{\oplus 2}$-graded matrix factorizations of $a(y_1^{N+1}-x_2^{N+1})$ over $R_\partial$. Since $B$ is a homogeneous $R_\partial$-basis for $M_0$, we have that, as $\zed_2\oplus\zed^{\oplus 2}$-graded matrix factorizations of $a(y_1^{N+1}-x_2^{N+1})$ over $R_\partial$, 
\begin{equation}\label{eq-RI-decomp-M0}
M_0 = (\bigoplus_{k=0}^{N-1} \Omega_k) \oplus (\bigoplus_{k=0}^{\infty} \Theta_k).
\end{equation}

Similarly, for $M_1$, denote by $\tilde{u}^{[k]}_{\ve\delta}$ the homogeneous element $(x_2-x_1)^k$ in $R_{\ve\delta}$ and by $\tilde{v}^{[k]}_{\ve\delta}$ the homogeneous element $(x_2-x_1)^k h_{N-1}(x_1,x_2)$ in $R_{\ve\delta}$. Let
\begin{eqnarray*}
\tilde{B}_{00} & = & \{\tilde{u}^{[k]}_{00}~|~ k\geq 0\}, \\
\tilde{B}_{10} & = & \{\tilde{u}^{[k]}_{10}~|~ k\geq 0\}, \\
\tilde{B}_{01} & = & \{\tilde{u}^{[k]}_{01}~|~ 0\leq k\leq N-2\}\cup \{\tilde{v}^{[k]}_{01}~|~ k\geq 0\}, \\
\tilde{B}_{11} & = & \{\tilde{u}^{[k]}_{11}~|~ 0\leq k\leq N-2\}\cup \{\tilde{v}^{[k]}_{11}~|~ k\geq 0\}.
\end{eqnarray*}
It is easy to see that $\tilde{B}_{\ve\delta}$ is a homogeneous $R_\partial$-basis for $R_{\ve\delta}$. So 
\begin{equation}\label{eq-basis-M1}
\tilde{B}:=\tilde{B}_{00}\cup \tilde{B}_{10}\cup \tilde{B}_{01}\cup \tilde{B}_{11}
\end{equation}
is a homogeneous $R_\partial$-basis for $M_1$. Define 
\begin{eqnarray}
\label{eq-RI-tilde-Theta} \tilde{\Theta}_k & = & \mathrm{span}_{R_\partial}\{\tilde{u}^{[k]}_{00}, \tilde{u}^{[k]}_{10}, \tilde{v}^{[k]}_{01}, \tilde{}v^{[k]}_{11}\} \text{ for } k \geq 0,\\
\label{eq-RI-tilde-Omega} \tilde{\Omega}_k & = & \mathrm{span}_{R_\partial}\{\tilde{u}^{[k]}_{01}, \tilde{u}^{[k]}_{11}\} \text{ for } 0\leq k \leq N-2.
\end{eqnarray}
It is straightforward to check that the differential map of $M_1$ preserves $\tilde{\Theta}_k$ and $\tilde{\Omega}_k$. So $\tilde{\Theta}_k$ and $\tilde{\Omega}_k$ are $\zed_2\oplus\zed^{\oplus 2}$-graded matrix factorizations of $a(y_1^{N+1}-x_2^{N+1})$ over $R_\partial$. Since $\tilde{B}$ is a homogeneous $R_\partial$-basis for $M_1$, we have that, as $\zed_2\oplus\zed^{\oplus 2}$-graded matrix factorizations of $a(y_1^{N+1}-x_2^{N+1})$ over $R_\partial$, 
\begin{equation}\label{eq-RI-decomp-M1}
M_1 = (\bigoplus_{k=0}^{N-2} \tilde{\Omega}_k) \oplus (\bigoplus_{k=0}^{\infty} \tilde{\Theta}_k).
\end{equation}

From the matrix form of $f$ given above, one can see that
\begin{itemize}
	\item $f(\tilde{u}^{[k]}_{00}) = u^{[k]}_{00}$, $f(\tilde{u}^{[k]}_{10}) = u^{[k]}_{10}$ for $k \geq 0$,
	\item $f(\tilde{v}^{[k]}_{01}) = v^{[k+1]}_{01}$, $f(\tilde{v}^{[k]}_{11}) = v^{[k+1]}_{11}$ for $k \geq 0$,
	\item $f(\tilde{u}^{[k]}_{01}) = u^{[k+1]}_{01}$, $f(\tilde{u}^{[k]}_{11}) = u^{[k+1]}_{11}$ for $k=0,\dots,N-2$.
\end{itemize}
So 
\begin{itemize}
	\item $f$ maps $\tilde{\Theta}_k$ isomorphically onto $\Theta_k$ for $k\geq 0$,
	\item $f$ maps $\tilde{\Omega}_k$ isomorphically onto $\Omega_{k+1}$ for $0\leq k\leq N-2$.
\end{itemize}
Thus, by Lemma \ref{lemma-gaussian-elimination} and decompositions \eqref{eq-RI-decomp-M0}, \eqref{eq-RI-decomp-M1}, we know that 
\begin{equation}\label{eq-RI-chain-reduction+}
0 \rightarrow \underbrace{M_1}_{-1} \xrightarrow{f} \underbrace{M_0}_{0} \rightarrow 0 \simeq 0 \rightarrow \underbrace{\Omega_0}_{0} \rightarrow 0
\end{equation}
as chain complexes over the category $\hmf^{\mathrm{all}}_{R_\partial,a(y_1^{N+1}-x_2^{N+1})}$. Note that 
\begin{eqnarray}
\label{eq-RI-Omega0}\Omega_0 & \cong & R_\partial\{-2,2-2N\} \xrightarrow{y_1-x_2} R_\partial\{-1,1-N\} \xrightarrow{ah_N(y_1,x_2)}  R_\partial\{-2,2-2N\} \\
\nonumber & = & (ah_N(y_1,x_2),y_1-x_2)_{R_\partial}\left\langle 1\right\rangle \{-1,1-N\}.
\end{eqnarray}
Combining \eqref{eq-RI-T+}, \eqref{eq-RI-chain-reduction+} and \eqref{eq-RI-Omega0}, we get that
\[
\fC_N(T_+) \simeq 0 \rightarrow (ah_N(y_1,x_2),y_1-x_2)_{R_\partial} \rightarrow 0 \simeq \fC_N(T)
\]
as chain complexes over the category $\hmf^{\mathrm{all}}_{R_\partial,a(y_1^{N+1}-x_2^{N+1})}$.
\end{proof}

\begin{figure}[ht]
$
\xymatrix{
\input{crossing-1-1-res-0-closed} && \input{crossing-1-1-res-1-closed} && \input{decomp-I-fig1}
}
$
\caption{}\label{decomp-I-fig}

\end{figure}

The following is a simple corollary of the proof of Proposition \ref{prop-RI+inv}.

\begin{corollary}\label{cor-decomp-I}
Let $\hat{\Gamma}_0$, $\hat{\Gamma}_1$ and $\hat{\Gamma}$ be the MOY graphs in Figure \ref{decomp-I-fig}. Then, as objects of the category $\hmf^{\mathrm{all}}_{\Q[a,y_1,x_2],a(y_1^{N+1}-x_2^{N+1})}$,
\[
\fC_N(\hat{\Gamma}_0) \simeq \fC_N(\hat{\Gamma})\left\langle 1\right\rangle \{-1,1-N\} \oplus \fC_N(\hat{\Gamma}_1)\{0,1\}.
\]
\end{corollary}

\begin{proof}
This corollary follows from Lemma \ref{lemma-RI-chi-explicit}, decompositions \eqref{eq-RI-decomp-M0}, \eqref{eq-RI-decomp-M1}, isomorphism \eqref{eq-RI-Omega0} and the fact that $f$ maps $\tilde{\Theta}_k$ isomorphically onto $\Theta_k$ and maps $\tilde{\Omega}_k$ isomorphically onto $\Omega_{k+1}$.
\end{proof}

\subsection{Negative stabilization}\label{subsec-neg-stab} The main result of this subsection is Proposition \ref{prop-RI-} below, which implies Theorem \ref{thm-neg-stabilization} and Corollary \ref{cor-unknot-neg-stab}. The proof of Proposition \ref{prop-RI-} is parallel to that of Proposition \ref{prop-RI+inv}.

\begin{figure}[ht]
$
\xymatrix{
\input{curve}  && \input{kink-}
}
$
\caption{}\label{stab-fig}

\end{figure}

Before stating Proposition \ref{prop-RI-}, we define a morphism of matrix factorizations. Note that the tangle $T$ can also be viewed as a MOY graph. Its matrix factorization is 
\[
\fC_N(T) = (ah_N(y_1,x_2), y_1-x_2)_{R_\partial} \cong R_\partial \xrightarrow{ah_N(y_1,x_2)} R_\partial  \{-1,-N+1\} \xrightarrow{y_1-x_2}  R_\partial,
\]
where $R_\partial=\Q[a,y_1,x_2]$.
Moreover,
\begin{eqnarray*}
\fC_N(T)\otimes_{\Q[a]} (0,a)_{\Q[a]} & = & \fC_N(T)\otimes_{R_\partial} (0,a)_{R_\partial} \\
& = & {\left(%
\begin{array}{cc}
  ah_N(y_1,x_2) & y_1-x_2 \\
  0 & a
\end{array}%
\right)_{R_\partial}} \\
& \cong &
{\left.%
\begin{array}{c}
  R_\partial \\
  \oplus \\
  R_\partial\{0,-2N\}
\end{array}%
\right.} \xrightarrow{\Delta'_0}
{\left.%
\begin{array}{c}
  R_\partial \{-1,-N+1\} \\
  \oplus \\
  R_\partial \{1,-N-1\}
\end{array}%
\right.}
\xrightarrow{\Delta'_1}
{\left.%
\begin{array}{c}
  R_\partial  \\
  \oplus \\
  R_\partial\{0,-2N\}
\end{array}%
\right.},
\end{eqnarray*}
where 
\begin{eqnarray*}
\Delta'_0 & = & {\left(%
\begin{array}{cc}
  ah_N(y_1,x_2) & -a \\
  0 & y_1-x_2
\end{array}%
\right)}, \\
\Delta'_1 & = & {\left(%
\begin{array}{cc}
  y_1-x_2 & a \\
  0 & ah_N(y_1,x_2)
\end{array}%
\right)}.
\end{eqnarray*}

\begin{definition}\label{def-RI-jmath}
Using the above explicit forms of $\fC_N(T)$ and $\fC_N(T)\otimes_{\Q[a]} (0,a)_{\Q[a]}$, we define an $R_\partial$-module map $\jmath:\fC_N(T) \rightarrow \fC_N(T)\otimes_{\Q[a]} (0,a)_{\Q[a]}$ by 
\begin{eqnarray*}
R_\partial & \xrightarrow{\left(%
\begin{array}{c}
  1 \\
  0 
\end{array}%
\right)} &
{\left.%
\begin{array}{c}
  R_\partial  \\
  \oplus \\
  R_\partial\{0,-2N\}
\end{array}%
\right.}, \\
R_\partial  \{-1,-N+1\} & \xrightarrow{\left(%
\begin{array}{c}
  1 \\
  0 
\end{array}%
\right)} &
{\left.%
\begin{array}{c}
  R_\partial \{-1,-N+1\} \\
  \oplus \\
  R_\partial \{1,-N-1\}
\end{array}%
\right.}.
\end{eqnarray*}
It is easy to check that $\jmath$ is a homogeneous morphism of matrix factorizations of $\zed_2 \oplus \zed^{\oplus2}$-degree $(0,0,0)$.
\end{definition}

\begin{proposition}\label{prop-RI-}
Let $T$ and $T_-$ be the tangles in Figure \ref{stab-fig}, $R_\partial=\Q[a,y_1,x_2]$ and $w=a(y_1^{N+1}-x_2^{N+1})$. Denote by $\jmath:\fC_N(T) \rightarrow \fC_N(T)\otimes_{R_\partial} (0,a)_{R_\partial}$ the morphism given in Definition \ref{def-RI-jmath}. Then, for $N \geq 1$, 
\[
\fC_N(T_-) \simeq 0 \rightarrow \underbrace{\fC_N(T)\{-2,0\}}_{0} \xrightarrow{\jmath} \underbrace{\fC_N(T)\otimes_{\Q[a]} (0,a)_{\Q[a]}\{-2,0\}}_{1} \rightarrow 0
\]
as chain complexes over the category $\hmf^{\mathrm{all}}_{R_\partial,w}$. 
\end{proposition}

\begin{proof}
Let $\hat{\Gamma}_0$ and $\hat{\Gamma}_1$ be the MOY graphs in Figure \ref{kink-res-fig}. Set $R=\Q[a,x_1,x_2,y_1]$. From Lemma \ref{lemma-RI-chi-explicit}, we know that 
\begin{eqnarray}
\label{eq-RI-T-} \fC_N(T_-) & = & 0 \rightarrow \underbrace{\fC_N(\hat{\Gamma}_0)\left\langle 1\right\rangle\{-1,-N+1\}}_{0} \xrightarrow{\chi^0} \underbrace{\fC_N(\hat{\Gamma}_1)\left\langle 1\right\rangle\{-1,-N\}}_{1} \rightarrow 0 \\
\nonumber & \cong & 0 \rightarrow \underbrace{M_0\left\langle 1\right\rangle\{-1,-N+1\}}_{0} \xrightarrow{g} \underbrace{M_1\left\langle 1\right\rangle\{-1,-N-1\}}_{1} \rightarrow 0,
\end{eqnarray}
where $M_0$, $M_1$ are the matrix factorizations defined in \eqref{eq-R1-def-M0}, \eqref{eq-R1-def-M1}, and $g$ is the morphism obtained by applying Lemma \ref{lemma-jumping-factor} to the second rows of $M_0$ and $M_1$.

We prove the proposition by simplifying the chain complex 
\[
0 \rightarrow \underbrace{M_0}_{0} \xrightarrow{g} \underbrace{M_1\{0,-2\}}_{1} \rightarrow 0.
\]
Again, we use the explicit forms \eqref{eq-kink-res-0-ex} and \eqref{eq-kink-res-1-ex} of $M_0$ and $M_1$. Under these explicit forms, the morphism $g:M_0 \rightarrow M_1$ is given by 
\begin{eqnarray*}
  {\left.%
\begin{array}{c}
  R_{00} \\
  \oplus \\
  R_{11}\{-2,2-2N\}
\end{array}%
\right.} 
& \xrightarrow{\left(%
\begin{array}{cc}
  x_2-x_1 & 0 \\
  0 & 1
\end{array}%
\right)} & 
{\left.%
\begin{array}{c}
  R_{00} \\
  \oplus \\
  R_{11}\{-2,4-2N\}
\end{array}%
\right.}, \\
{\left.%
\begin{array}{c}
  R_{10}\{-1,1-N\} \\
  \oplus \\
  R_{01}\{-1,1-N\}
\end{array}%
\right.}
& \xrightarrow{\left(%
\begin{array}{cc}
  x_2-x_1 & 0 \\
  0 & 1
\end{array}%
\right)} &
{\left.%
\begin{array}{c}
  R_{10}\{-1,1-N\} \\
  \oplus \\
  R_{01}\{-1,3-N\}
\end{array}%
\right.}.
\end{eqnarray*}
Let $B$ and $\tilde{B}$ be the $R_\partial$-bases for $M_0$ and $M_1$ given in \eqref{eq-basis-M0} and \eqref{eq-basis-M1}. We have
\begin{itemize}
	\item $g(u^{[k]}_{00}) = \tilde{u}^{[k+1]}_{00}$, $g(u^{[k]}_{10}) = \tilde{u}^{[k+1]}_{10}$ for $k\geq 0$,
	\item $g(v^{[k]}_{01}) = \tilde{v}^{[k]}_{01}$, $g(v^{[k]}_{11}) = \tilde{v}^{[k]}_{11}$ for $k\geq 0$,
  \item $g(u^{[k]}_{01}) = \tilde{u}^{[k]}_{01}$, $g(u^{[k]}_{11}) = \tilde{u}^{[k]}_{11}$ for $k=0,\dots,N-2$.	
\end{itemize}
So 
\begin{itemize}
	\item $g$ maps $\Theta_k$ isomorphically onto $\tilde{\Theta}_{k+1}$ for $k\geq 0$,
	\item $g$ maps $\Omega_k$ isomorphically onto $\tilde{\Omega}_{k}$ for $0\leq k\leq N-2$,
\end{itemize}
where the $\zed_2\oplus\zed^{\oplus2}$-graded matrix factorizations $\Theta_k$, $\Omega_k$, $\tilde{\Theta}_{k}$ and $\tilde{\Omega}_{k}$ are defined in \eqref{eq-RI-Theta}, \eqref{eq-RI-Omega}, \eqref{eq-RI-tilde-Theta} and \eqref{eq-RI-tilde-Omega}. 

Moreover, note that there are unique polynomials $p_0(x_2),p_1(x_2),\dots,p_{N-2}(x_2) \in \Q[x_2]$ such that 
\[
(x_2-x_1)^{N-1} = (-1)^{N-1}h_{N-1}(x_1,x_2) +\sum_{l=0}^{N-2} p_l(x_2)(x_2-x_1)^l.
\]
Therefore,
\begin{eqnarray}
\label{eq-RI-g-u-01} g(u^{[N-1]}_{01}) & = & (-1)^{N-1}\tilde{v}^{[0]}_{01}+\sum_{l=0}^{N-2} p_l(x_2)\tilde{u}^{[l]}_{01}, \\
\label{eq-RI-g-u-11} g(u^{[N-1]}_{11}) & = & (-1)^{N-1}\tilde{v}^{[0]}_{11}+\sum_{l=0}^{N-2} p_l(x_2)\tilde{u}^{[l]}_{11}.
\end{eqnarray}

By the Gaussian Elimination Lemma (Lemma \ref{lemma-gaussian-elimination},) the above properties of $g$ imply that
\begin{eqnarray}
\label{eq-RI-chain1}0 \rightarrow \underbrace{M_0}_{0} \xrightarrow{g} \underbrace{M_1\{0,-2\}}_{1} \rightarrow 0 & \simeq &  0\rightarrow \underbrace{\Omega_{N-1}}_{0} \xrightarrow{(-1)^{N-1}J} \underbrace{\tilde{\Theta}_{0}\{0,-2\}}_{1} \rightarrow 0, \\
\nonumber & \cong & 0\rightarrow \underbrace{\Omega_{N-1}}_{0} \xrightarrow{J} \underbrace{\tilde{\Theta}_{0}\{0,-2\}}_{1} \rightarrow 0,
\end{eqnarray}
where $J(u^{[N-1]}_{01}) = \tilde{v}^{[0]}_{01}$ and $J(u^{[N-1]}_{11}) = \tilde{v}^{[0]}_{11}$.

From \eqref{eq-kink-res-0-ex}, one can see that, as $\zed_2\oplus\zed^{\oplus 2}$-graded matrix factorizations over $R_\partial$,
\begin{eqnarray*}
&& \Omega_{N-1} \\
& = & R_\partial \cdot u^{[N-1]}_{11}\{-2,2-2N\} \xrightarrow{y_1-x_2} R_\partial \cdot u^{[N-1]}_{01}\{-1,1-N\}  \xrightarrow{ah_N(y_1,x_2)} R_\partial \cdot u^{[N-1]}_{11}\{-2,2-2N\} \\
& \cong & R_\partial \{-2,0\} \xrightarrow{y_1-x_2} R_\partial  \{-1,N-1\}  \xrightarrow{ah_N(y_1,x_2)} R_\partial \{-2,0\}
\end{eqnarray*}
Thus, 
\begin{eqnarray}
\label{eq-RI-O-N-1}\Omega_{N-1}\left\langle 1 \right\rangle \{-1,-N+1\} & \cong & R_\partial \{-2,0\} \xrightarrow{ah_N(y_1,x_2)} R_\partial  \{-3,-N+1\} \xrightarrow{y_1-x_2}  R_\partial \{-2,0\} \\
\nonumber & = & (ah_N(y_1,x_2), y_1-x_2)_{R_\partial}\{-2,0\} \\
\nonumber & = & \fC_N(T)\{-2,0\}.
\end{eqnarray}
Similarly, from \eqref{eq-kink-res-1-ex}, we have
\begin{eqnarray*}
\tilde{\Theta}_0 
& = & {\left.%
\begin{array}{c}
  R_\partial \cdot u^{[0]}_{00} \\
  \oplus \\
  R_\partial \cdot v^{[0]}_{11}\{-2,4-2N\}
\end{array}%
\right.} \xrightarrow{\Delta_0}
{\left.%
\begin{array}{c}
  R_\partial \cdot u^{[0]}_{10}\{-1,1-N\} \\
  \oplus \\
  R_\partial \cdot v^{[0]}_{01}\{-1,3-N\}
\end{array}%
\right.}
\xrightarrow{\Delta_1}
{\left.%
\begin{array}{c}
  R_\partial \cdot u^{[0]}_{00} \\
  \oplus \\
  R_\partial \cdot v^{[0]}_{11}\{-2,4-2N\}
\end{array}%
\right.} \\
& \cong & {\left.%
\begin{array}{c}
  R_\partial \\
  \oplus \\
  R_\partial \{-2,2\}
\end{array}%
\right.} \xrightarrow{\tilde{\Delta}_0}
{\left.%
\begin{array}{c}
  R_\partial \{-1,1-N\} \\
  \oplus \\
  R_\partial \{-1,N+1\}
\end{array}%
\right.}
\xrightarrow{\tilde{\Delta}_1}
{\left.%
\begin{array}{c}
  R_\partial \\
  \oplus \\
  R_\partial \{-2,2\}
\end{array}%
\right.},
\end{eqnarray*}
where 
\begin{eqnarray*}
\tilde{\Delta}_0 & = & {\left(%
\begin{array}{cc}
  ah_N(y_1,x_2) & 0 \\
  a & y_1-x_2
\end{array}%
\right)}, \\
\tilde{\Delta}_1 & = & {\left(%
\begin{array}{cc}
  y_1-x_2 & 0 \\
  -a & ah_N(y_1,x_2)
\end{array}%
\right)}.
\end{eqnarray*}
Thus, 
\begin{eqnarray}
\label{eq-RI-Theta-0}
\tilde{\Theta}_0 \left\langle 1 \right\rangle \{-1,-N-1\}
& \cong & {\left.%
\begin{array}{c}
  R_\partial\{-2,-2N\} \\
  \oplus \\
  R_\partial \{-2,0\}
\end{array}%
\right.} \xrightarrow{\tilde{\Delta}_1}
{\left.%
\begin{array}{c}
  R_\partial \{-1,-N-1\} \\
  \oplus \\
  R_\partial \{-3,-N+1\}
\end{array}%
\right.}
\xrightarrow{\tilde{\Delta}_0}
{\left.%
\begin{array}{c}
  R_\partial\{-2,-2N\} \\
  \oplus \\
  R_\partial \{-2,0\}
\end{array}%
\right.} \\
\nonumber & \cong & {\left.%
\begin{array}{c}
  R_\partial \{-2,0\} \\
  \oplus \\
  R_\partial\{-2,-2N\}
\end{array}%
\right.} \xrightarrow{\Delta'_0}
{\left.%
\begin{array}{c}
  R_\partial \{-3,-N+1\} \\
  \oplus \\
  R_\partial \{-1,-N-1\}
\end{array}%
\right.}
\xrightarrow{\Delta'_1}
{\left.%
\begin{array}{c}
  R_\partial \{-2,0\} \\
  \oplus \\
  R_\partial\{-2,-2N\}
\end{array}%
\right.} \\
\nonumber & \cong & {\left(%
\begin{array}{cc}
  ah_N(y_1,x_2) & y_1-x_2 \\
  0 & a
\end{array}%
\right)_{R_\partial}}\{-2,0\} \\
\nonumber & = & \fC_N(T)\otimes_{\Q[a]} (0,a)_{\Q[a]}\{-2,0\},
\end{eqnarray}
where 
\begin{eqnarray*}
\Delta'_0 & = & {\left(%
\begin{array}{cc}
  ah_N(y_1,x_2) & -a \\
  0 & y_1-x_2
\end{array}%
\right)}, \\
\Delta'_1 & = & {\left(%
\begin{array}{cc}
  y_1-x_2 & a \\
  0 & ah_N(y_1,x_2)
\end{array}%
\right)}.
\end{eqnarray*}

Note that, under isomorphisms \eqref{eq-RI-O-N-1} and \eqref{eq-RI-Theta-0}, the morphism $J:\Omega_{N-1} \rightarrow \tilde{\Theta}_{0}\{0,-2\}$ is identified with the morphism $\jmath:\fC_N(T)\{-2,0\} \rightarrow \fC_N(T)\otimes_{\Q[a]} (0,a)_{\Q[a]}\{-2,0\}$ defined in Definition \ref{def-RI-jmath}. Thus, by \eqref{eq-RI-T-} and \eqref{eq-RI-chain1}, we have
\[
\fC_N(T_-) \simeq 0 \rightarrow \underbrace{\fC_N(T)\{-2,0\}}_{0} \xrightarrow{\jmath} \underbrace{\fC_N(T)\otimes_{\Q[a]} (0,a)_{\Q[a]}\{-2,0\}}_{1} \rightarrow 0.
\]
\end{proof}

Theorem \ref{thm-neg-stabilization} and Corollary \ref{cor-unknot-neg-stab} follow easily from Propositions \ref{prop-RI-} and \ref{prop-contraction-weak}.

\begin{proof}[Proof of Theorem \ref{thm-neg-stabilization}]
By Proposition \ref{prop-RI-}, 
\[
\fC_N(L_-) \simeq 0 \rightarrow \underbrace{\fC_N(L)\{-2,0\}}_{0} \xrightarrow{\jmath} \underbrace{\fC_N(L)\otimes_{\Q[a]}(0,a)_{\Q[a]}\{-2,0\}}_{1} \rightarrow 0. 
\]
So the chain complex $(H(\fC_N(L_-),d_{mf}),d_\chi)$ is isomorphic to the total chain complex of
\[
0 \rightarrow \underbrace{H(\fC_N(L),d_{mf})\{-2,0\}}_{0} \xrightarrow{\jmath} \underbrace{H(\fC_N(L)\otimes_{\Q[a]}(0,a)_{\Q[a]},d_{mf})\{-2,0\}}_{1} \rightarrow 0.
\] 
By Part (2) of Proposition \ref{prop-contraction-weak}, there is a quasi-isomorphism 
\[
\alpha:(\fC_N(L)\otimes_{\Q[a]}(0,a)_{\Q[a]},d_{mf}) \rightarrow (\fC_N(L)/a\fC_N(L),d_{mf})
\]
that preserves the $\zed_2\oplus\zed^{\oplus 3}$-grading. From the definitions of $\jmath$, $\alpha$ and $\pi_0$, one can check that $\pi_0 = \alpha \circ\jmath$. Thus, the chain complex $(H(\fC_N(L_-),d_{mf}),d_\chi)$ is isomorphic to the total chain complex of 
\[
0 \rightarrow \underbrace{H(\fC_N(L),d_{mf})\{-2,0\}}_{0} \xrightarrow{\pi_0} \underbrace{H(\fC_N(L)/a\fC_N(L),d_{mf})\{-2,0\}}_{1} \rightarrow 0.
\] 
The above total chain complex is of course the mapping cone of $\pi_0$. The long exact sequence in Theorem \ref{thm-neg-stabilization} is the long exact sequence of this mapping cone.
\end{proof}

\begin{proof}[Proof of Corollary \ref{cor-unknot-neg-stab}]
We put a single marking $x$ on $U$. Then  
\[
\fC_N(U) = ((N+1)ax^N,0)_{\Q[a,x]} \cong (ax^N,0)_{\Q[a,x]} = \Q[a,x] \xrightarrow{ax^N} \Q[a,x]\{-1, -N+1\} \xrightarrow{0} \Q[a,x].
\]
From this, one can see that
\[
\fH_N^{1,0,\star,\star}(U) \cong \Q[a,x]/(ax^N)\{-1, -N+1\} \cong  (\bigoplus_{l=0}^{N-1} \Q[a]\{-1,-N+1+2l\}) \oplus (\bigoplus_{m=0}^{\infty} \Q[a]/(a) \{-1,N+1+2m\})
\]
and $\fH_N^{\ve,i,\star,\star}(U) \cong 0$ for any other pair of $(\ve,i) \in \zed_2 \oplus \zed$. This completes the computation of $\fH_N(U)$.

Next, we compute $\mathscr{H}_N(U) := H(H(\fC_N(U)/a\fC_N(U),d_{mf}),d_\chi)$. Since $U$ has no crossings, $d_\chi$ is $0$ for $U$. So
\[
\mathscr{H}_N(U) := H(H(\fC_N(U)/a\fC_N(U),d_{mf})) = H(\Q[a,x]/(a) \xrightarrow{0} \Q[a,x]/(a)\{-1, -N+1\} \xrightarrow{0} \Q[a,x]/(a))
\]
and, therefore
\[
\mathscr{H}^{\ve,i,\star,\star}_N(U) = \begin{cases}
\Q[a,x]/(a) & \text{if } \ve =0 \text{ and } i=0, \\
\Q[a,x]/(a)\{-1, -N+1\} & \text{if } \ve =1 \text{ and } i=0, \\
0 & \text{otherwise.}
\end{cases}
\]

For $\ve=1$, homomorphism $\pi_0: \fH_N^{1,\star,\star,\star}(U) \rightarrow \mathscr{H}^{1,\star,\star,\star}_N(U)$ is the standard quotient map $\Q[a,x]/(ax^N) \xrightarrow{\pi_0} \Q[a,x]/(ax^N, a) = \Q[a,x]/(a)$. So the long exact sequence in Theorem \ref{thm-neg-stabilization} reduces to
\[
0 \rightarrow \fH_N^{1,0,\star,\star}(U_-) \rightarrow \fH_N^{1,0,\star,\star}(U)\{-2,0\} \xrightarrow{\pi_0} \mathscr{H}^{1,0,\star,\star}_N(U)\{-2,0\} \rightarrow 0.
\]
Thus,
\[
\fH_N^{1,i,\star,\star}(U_-) \cong \begin{cases}
\ker \pi_0 = a\cdot\Q[a,x]/(ax^N) \{-3, -N+1\}\cong \bigoplus_{l=0}^{N-1} \Q[a]\{-1,-N+1+2l\} & \text{if } i=0,\\
0 & \text{if } i \neq 0.
\end{cases}
\]
For $\ve=0$, homomorphism $\pi_0: \fH_N^{0,\star,\star,\star}(U) \rightarrow \mathscr{H}^{0,\star,\star,\star}_N(U)$ is the zero map $0 \rightarrow \Q[a,x]/(a)$. So the long exact sequence in Theorem \ref{thm-neg-stabilization} reduces to
\[
0 \rightarrow \mathscr{H}^{0,0,\star,\star}_N(U)\{-2,0\} \rightarrow \fH_N^{0,1,\star,\star}(U_-) \rightarrow 0.
\]
Thus,
\[
\fH_N^{0,i,\star,\star}(U_-) \cong \begin{cases}
\Q[a,x]/(a)\{-2,0\} \cong \bigoplus_{m=0}^{\infty} \Q[a]/(a) \{-2,2m\} & \text{if } i=1,\\
0 & \text{if } i \neq 1.
\end{cases}
\]
This completes the computation of $\fH_N(U_-)$.
\end{proof}

\section{Reidemeister Move II}\label{sec-R2}

In this section, we prove the invariance of $\fH_N$ under braid-like Reidemeister II moves. The main result of this section is Proposition \ref{prop-RIIinv}. Our proof here is a straightforward adaptation of the proofs in \cite{KR1,KR2}.

\begin{figure}[ht]
$
\xymatrix{
\input{RII-0} &&& \input{RII-1} 
}
$
\caption{}\label{RII-inv-fig}

\end{figure}

\begin{proposition}\label{prop-RIIinv}
Let $T_0$ and $T_1$ be the tangles in Figure \ref{RII-inv-fig}. Then, for $N \geq 0$, $\fC_N(T_0) \simeq \fC_N(T_1)$ as chain complexes over the category $\hmf_{\Q[a,x_1,x_2,x_3,x_4],a(x_1^{N+1}+x_2^{N+1}-x_3^{N+1}-x_4^{N+1})}$.
\end{proposition}

\begin{figure}[ht]
$
\xymatrix{
\input{RII-res01} && \input{RII-res00} \\
\input{RII-res-11} && \input{RII-res-10}
}
$
\caption{}\label{RII-inv-res-fig}

\end{figure}

In the rest of this section, we let $R_\partial = \Q[a,x_1,x_2,x_3,x_4]$ and $w= a(x_1^{N+1}+x_2^{N+1}-x_3^{N+1}-x_4^{N+1})$. The resolutions of $T_1$ are listed in Figure \ref{RII-inv-res-fig}. To prove Proposition \ref{prop-RIIinv}, we need the following lemma.

\begin{lemma}\label{lemma-RII-hmf}
Let $\Gamma_{0,1}$ and $\Gamma_{-1,0}$ be the MOY graphs in Figure \ref{RII-inv-res-fig}. In the category $\hmf_{R_\partial,w}$, we have that
\[
\Hom_\hmf(\fC_N(\Gamma_{0,1}),\fC_N(\Gamma_{0,1})) \cong \Hom_\hmf(\fC_N(\Gamma_{-1,0}),\fC_N(\Gamma_{-1,0})) \cong \Q.
\]
Both of these spaces are spanned by the identity morphisms.
\end{lemma}

\begin{proof}
By Lemma \ref{lemma-marking-independence}, $\fC_N(\Gamma_{0,1})\simeq \fC_N(\Gamma_{-1,0})$. So $\Hom_\hmf(\fC_N(\Gamma_{0,1}),\fC_N(\Gamma_{0,1})) \cong \Hom_\hmf(\fC_N(\Gamma_{-1,0}),\fC_N(\Gamma_{-1,0}))$. Thus, we only need to compute $\Hom_\hmf(\fC_N(\Gamma_{0,1}),\fC_N(\Gamma_{0,1}))$. By Proposition \ref{prop-b-contraction}, we have
\[
\fC_N(\Gamma_{0,1}) \simeq \left(%
\begin{array}{cc}
  aU_1 & x_1+x_2-x_3-x_4 \\
  aU_2 & x_1x_2-x_3x_4
\end{array}%
\right)_{R_\partial}\{0,-1\},
\]
where $U_1$ and $U_2$ are given by equation \eqref{eq-def-U-j}. From this, it is easy to check that $\gdim_{R_\partial} \fC_N(\Gamma_{0,1}) \neq 0$. So $\fC_N(\Gamma_{0,1})$ is not homotopic to $0$ and, therefore, $\id_{\fC_N(\Gamma_{0,1})}$ is not homotopic to $0$. Thus, $\dim_\Q \Hom_\hmf(\fC_N(\Gamma_{0,1}),\fC_N(\Gamma_{0,1})) \geq 1$.

By Lemma \ref{lemma-dual-Koszul} and Proposition \ref{prop-contraction-weak}, we have that 
\begin{eqnarray*}
\Hom_{\hmf}(\fC_N(\Gamma_{0,1}),\fC_N(\Gamma_{0,1})) & \cong &  {H^{0,0,0}\left(\left(%
\begin{array}{cc}
  aU_1 & x_1+x_2-x_3-x_4 \\
  aU_2 & x_1x_2-x_3x_4 \\
  aU_1 & -x_1-x_2+x_3+x_4 \\
  aU_2 & -x_1x_2+x_3x_4
\end{array}%
\right)_{R_\partial}\{2,2N-4\}\right)} \\
& \cong & {H^{0,0,0}\left(\left(%
\begin{array}{cc}
  aV_1 & 0 \\
  aV_2 & 0 
\end{array}%
\right)_{R}\{2,2N-4\}\right)}
\end{eqnarray*}
where $R=R_\partial/(x_1+x_2-x_3-x_4, x_1x_2-x_3x_4)$, and $V_1$, $V_2$ are the images of $U_1$, $U_2$ in $R$. As a $\zed_2 \oplus \zed^{\oplus 2}$-graded $R$-module,
\[
\left(%
\begin{array}{cc}
  aV_1 & 0 \\
  aV_2 & 0 
\end{array}%
\right)_{R}\{2,2N-4\}
\cong
(R\{1,N-1\} \oplus R \left\langle 1\right\rangle) \otimes_R (R\{1,N-3\} \oplus R \left\langle 1\right\rangle).
\]
Note that the homogeneous component of the right hand side of $\zed_2 \oplus \zed^{\oplus 2}$-degree $(0,0,0)$ is $1$-dimensional over $\Q$. This implies that $\dim_\Q \Hom_\hmf(\fC_N(\Gamma_{0,1}),\fC_N(\Gamma_{0,1})) \leq 1$. Thus, $\Hom_\hmf(\fC_N(\Gamma_{0,1}),\fC_N(\Gamma_{0,1})) \cong \Q$.
\end{proof}

We are now ready to prove Proposition \ref{prop-RIIinv}.

\begin{proof}[Proof of Proposition \ref{prop-RIIinv}]
It is easy to check that $\fC_N(\Gamma_{0,0})$ and $\fC_N(\Gamma_{0,1})\simeq \fC_N(\Gamma_{-1,0})$ are all homotopically finite over $R_\partial$. So  $\fC_N(T_0)= 0\rightarrow \underbrace{\fC_N(\Gamma_{0,0})}_{0} \rightarrow 0$ is a chain complex over $\hmf_{R_\partial,w}$. By Lemma \ref{lemma-decomp-II}, we have
\begin{equation}\label{eq-RII-decomp}
\fC_N(\Gamma_{-1,1}) \simeq \fC_N(\Gamma_{0,1})\{0,1\} \oplus \fC_N(\Gamma_{-1,0})\{0,-1\}.
\end{equation}
So $\fC_N(\Gamma_{-1,1})$ is also homotopically finite over $R_\partial$. Denote by $\chi^1_u$, $\chi^0_u$ the $\chi$-morphisms associated to the upper $2$-colored edge in $\Gamma_{-1,1}$ and by $\chi^1_l$, $\chi^0_l$ the $\chi$-morphisms associated to the lower $2$-colored edge in $\Gamma_{-1,1}$. Then, 
\[
\fC_N(T_1) = 0 \rightarrow \underbrace{\fC_N(\Gamma_{0,1})\{0,1\}}_{-1} \xrightarrow{\left(%
\begin{array}{c}
  \chi^1_l \\
  \chi^0_u
\end{array}%
\right)} \underbrace{\left.%
\begin{array}{c}
  \fC_N(\Gamma_{0,0}) \\
  \oplus \\
  \fC_N(\Gamma_{-1,1})
\end{array}%
\right.}_{0} \xrightarrow{(\chi^0_u,~-\chi^1_l)}
\underbrace{\fC_N(\Gamma_{-1,0})\{0,-1\}}_{0} \rightarrow 0,
\] 
which is a chain complex over $\hmf_{R_\partial,w}$. 

To prove Proposition \ref{prop-RIIinv}, we need to use the morphisms involved in decomposition \eqref{eq-RII-decomp}. By Lemma \ref{lemma-phi}, 
\begin{itemize}
	\item the projection $\fC_N(\Gamma_{-1,1}) \rightarrow \fC_N(\Gamma_{0,1})\{0,1\}$ in decomposition \eqref{eq-RII-decomp} is the morphism $\bar{\phi}$ induced by edge merging,
	\item the inclusion $\fC_N(\Gamma_{-1,0})\{0,-1\} \rightarrow \fC_N(\Gamma_{-1,1})$ in decomposition \eqref{eq-RII-decomp} is the morphism $\phi$ induced by edge splitting.
\end{itemize}
In addition, we denote: 
\begin{itemize}
	\item by $\fC_N(\Gamma_{0,1})\{0,1\} \xrightarrow{J} \fC_N(\Gamma_{-1,1})$ the corresponding inclusion in decomposition \eqref{eq-RII-decomp},
	\item by $\fC_N(\Gamma_{-1,1}) \xrightarrow{P} \fC_N(\Gamma_{-1,0})\{0,-1\}$ the corresponding projection in decomposition \eqref{eq-RII-decomp}. 
\end{itemize}
Then, as a chain complex over $\hmf_{R_\partial,w}$,
\[
\fC_N(T_1) \cong 0 \rightarrow \underbrace{\fC_N(\Gamma_{0,1})\{0,1\}}_{-1} \xrightarrow{\left(%
\begin{array}{c}
  \chi^1_l \\
  \bar{\phi}\circ\chi^0_u \\
  P \circ \chi^0_u 
\end{array}%
\right)} \underbrace{{\left.%
\begin{array}{c}
  \fC_N(\Gamma_{0,0}) \\
  \oplus \\
  \fC_N(\Gamma_{0,1})\{0,1\} \\
  \oplus \\
  \fC_N(\Gamma_{-1,0})\{0,-1\}
\end{array}%
\right.}}_{0} \xrightarrow{(\chi^0_u,~-\chi^1_l\circ J,~-\chi^1_l\circ \phi)}
\underbrace{\fC_N(\Gamma_{-1,0})\{0,-1\}}_{1} \rightarrow 0.
\]

Consider the morphism $\bar{\phi}\circ\chi^0_u$. First, note that it is a homogeneous morphism of $\zed_2 \oplus \zed^{\oplus 2}$-degree $(0,0,0)$. Moreover, by Lemmas \ref{lemma-phi} and \ref{lemma-def-chi},  we have 
\[
\bar{\phi}\circ\chi^0_u \circ \chi^1_u \circ \phi \simeq \bar{\phi}\circ \mathsf{m}(x_6-x_1) \circ \phi \simeq \bar{\phi}\circ \mathsf{m}(x_6) \circ \phi - \mathsf{m}(x_1)\bar{\phi}\circ \phi \approx \id_{\fC_N(\Gamma_{0,1})},
\]
which implies that $\bar{\phi}\circ\chi^0_u$ is not homotopic to $0$. By Lemma \ref{lemma-RII-hmf}, this means that
\begin{equation}\label{eq-phi-bar-chi0}
\bar{\phi}\circ\chi^0_u \approx \id_{\fC_N(\Gamma_{0,1})}.
\end{equation}
Similarly,
\begin{equation}\label{eq-chi1-phi}
\chi^1_l\circ \phi \approx \id_{\fC_N(\Gamma_{-1,0})}.
\end{equation}
Thus, $\bar{\phi}\circ\chi^0_u$ and $\chi^1_l\circ \phi$ are both homotopy equivalences. Now, applying the Gaussian Elimination Lemma (Lemma \ref{lemma-gaussian-elimination}) to these two homotopy equivalences in the above chain complex, we get that $\fC_N(T_1) \simeq  0\rightarrow \underbrace{\fC_N(\Gamma_{0,0})}_{0} \rightarrow 0 = \fC_N(T_0)$.
\end{proof}

\section{Reidemeister Move III}\label{sec-R3}

In this section, we prove the invariance of $\fH_N$ under braid-like Reidemeister III moves. The main result of this section is Proposition \ref{prop-RIIIinv}. We follow the ideas in \cite{KR1,KR2}. However, we include graphical descriptions of all morphisms used in our proof, which makes our proof somewhat more explicit than those in \cite{KR1,KR2}.

\subsection{A decomposition} Now we establish for $\fC_N$ a version of ``Direct Sum Decomposition IV" from \cite{KR1}. The main result of this subsection is Proposition \ref{prop-RIIIinv-decomp}.

\begin{figure}[ht]
$
\xymatrix{
\input{RIII-decomp} && \input{RIII-decomp0} && \input{RIII-decomp1}
}
$
\caption{}\label{RIIIinv-decomp-fig}

\end{figure}

\begin{proposition}\label{prop-RIIIinv-decomp}
Assume:
\begin{itemize}
	\item $\Gamma$, $\Gamma_0$, $\Gamma_1$ are the MOY graphs in Figure \ref{RIIIinv-decomp-fig},
	\item $R_\partial = \Q[a,x_3,x_6] \otimes_\Q \Sym(\{x_1,x_2\}) \otimes_\Q \Sym(\{x_4,x_5\})$,
	\item $w=a(x_1^{N+1}+x_2^{N+1}+x_3^{N+1}-x_4^{N+1}-x_5^{N+1}-x_6^{N+1})$.
\end{itemize}
Then, for $N \geq 0$, $\fC_N(\Gamma) \cong \fC_N(\Gamma_0) \oplus \fC_N(\Gamma_1)$ as objects of $\hmf_{R_\partial,w}$.
\end{proposition}

\begin{lemma}\label{lemma-RIIIinv-decomp-obj}
$\fC_N(\Gamma)$, $\fC_N(\Gamma_0)$ and $\fC_N(\Gamma_1)$ are objects of $\hmf_{R_\partial,w}$. Moreover, 
\[
\gdim_{R_\partial} \fC_N(\Gamma) = \gdim_{R_\partial} \fC_N(\Gamma_0) + \gdim_{R_\partial} \fC_N(\Gamma_1).
\]
\end{lemma}

\begin{proof}
It is easy to see that the $\zed^{\oplus 2}$-gradings of $\fC_N(\Gamma)$, $\fC_N(\Gamma_0)$ and $\fC_N(\Gamma_1)$ are bounded below. By Definition \ref{def-MOY-mf} and Proposition \ref{prop-b-contraction}, we know that
\begin{eqnarray}
\label{eq-RIII-mf-Gamma0} \fC_N(\Gamma_0) & = & {\left(%
\begin{array}{cc}
  \ast_{2,2N} & x_1+x_2-x_4-x_5 \\
  \ast_{2,2N-2} & x_1x_2-x_4x_5 \\
  \ast_{2,2N} & x_3-x_6
\end{array}%
\right)_{R_\partial}}, \\
\label{eq-RIII-mf-Gamma1} \fC_N(\Gamma_1) & \simeq & {\left(%
\begin{array}{cc}
  \ast_{2,2N} & x_1+x_2+x_3-x_4-x_5-x_6 \\
  \ast_{2,2N-2} & x_1x_2+x_2x_3+x_3x_1-x_4x_5-x_5x_6-x_6x_4 \\
  \ast_{2,2N-4} & x_1x_2x_3-x_4x_5x_6
\end{array}%
\right)_{R_\partial}\{0,-2\}} ,
\end{eqnarray}
where $\ast_{j,k}$ stands for a homogeneous element of the base ring of $\zed^{\oplus 2}$-degree $(j,k)$. So $\fC_N(\Gamma_0)$ and $\fC_N(\Gamma_1)$ are homotopically finite over $R_\partial$ and, therefore, objects of $\hmf_{R_\partial,w}$. Note that the maximal homogeneous ideal of $R_\partial$ is 
\[
\mathfrak{I}=(a,x_1+x_2,x_1x_2,x_4+x_5,x_4x_5,x_3,x_6).
\]
So
\begin{eqnarray*}
\fC_N(\Gamma_0)/\mathfrak{I} \cdot \fC_N(\Gamma_0) & \cong & {\left(%
\begin{array}{cc}
  0_{2,2N} & 0 \\
  0_{2,2N-2} & 0 \\
  0_{2,2N} & 0
\end{array}%
\right)_{\Q}}, \\
\fC_N(\Gamma_1)/\mathfrak{I} \cdot \fC_N(\Gamma_1) & \simeq & {\left(%
\begin{array}{cc}
  0_{2,2N} & 0 \\
  0_{2,2N-2} & 0 \\
  0_{2,2N-4} & 0
\end{array}%
\right)_{\Q}\{0,-2\}} ,
\end{eqnarray*}
where $0_{j,k}$ is ``the zero element with $\zed^{\oplus 2}$-degree $(j,k)$".(This is only used to keep track of grading shifts.) Note that the differential maps of these chain complexes are $0$. Thus,
\begin{eqnarray}
\label{eq-RIII-gdim-Gamma0} \gdim_{R_\partial} \fC_N(\Gamma_0) & = & (1+\tau\alpha^{-1}\xi^{-N+1})^2(1+\tau\alpha^{-1}\xi^{-N+3}), \\
\label{eq-RIII-gdim-Gamma1} \gdim_{R_\partial} \fC_N(\Gamma_1) & = & \xi^{-2} (1+\tau\alpha^{-1}\xi^{-N+1})(1+\tau\alpha^{-1}\xi^{-N+3})(1+\tau\alpha^{-1}\xi^{-N+5}),
\end{eqnarray}

Next we consider $\fC_N(\Gamma)$. By Corollary \ref{cor-homology-detects-homotopy}, to show that $\fC_N(\Gamma)$ is homotopically finite over $R_\partial$, we only need to show that $\dim_\Q H_{R_\partial}(\fC_N(\Gamma))$ is finite. By Definition \ref{def-MOY-mf} and Proposition \ref{prop-b-contraction},
\[
\fC_N(\Gamma) \simeq \left(%
\begin{array}{cc}
  \ast_{2,2N} & x_1+x_2-x_7-x_8 \\
  \ast_{2,2N-2} & x_1x_2-x_7x_8 \\
  \ast_{2,2N} & x_7+x_9-x_4-x_5 \\
  \ast_{2,2N-2} & x_7x_9-x_4x_5 \\
  \ast_{2,2N} & x_8+x_3-x_9-x_6 \\
  \ast_{2,2N-2} & x_8x_3-x_9x_6  
\end{array}%
\right)_{R}\{0,-2\},
\]
where $R=R_\partial \otimes_\Q \Q[x_7,x_8,x_9]$. So 
\[
\fC_N(\Gamma)/\mathfrak{I}\cdot\fC_N(\Gamma) \simeq \left(%
\begin{array}{cc}
  0_{2,2N} & -x_7-x_8 \\
  0_{2,2N-2} & -x_7x_8 \\
  0_{2,2N} & x_7+x_9 \\
  0_{2,2N-2} & x_7x_9 \\
  0_{2,2N} & x_8-x_9 \\
  0_{2,2N-2} & 0 
\end{array}%
\right)_{\Q[x_7,x_8,x_9]}\{0,-2\},
\]
Applying Proposition \ref{prop-contraction-weak} successively to the matrix factorization on the right hand side, we get that
\begin{eqnarray*}
H_{R_\partial}(\fC_N(\Gamma)) & \cong & H\left(\left(%
\begin{array}{cc}
  0_{2,2N} & 0 \\
  0_{2,2N-2} & 0 \\
  0_{2,2N-2} & 0 
\end{array}%
\right)_{\Q[x_9]/(x_9^2)}\right)\{0,-2\} \\
& \cong &  H\left(\left(%
\begin{array}{cc}
  0_{2,2N} & 0 \\
  0_{2,2N-2} & 0 \\
  0_{2,2N-2} & 0 
\end{array}%
\right)_{\Q}\right)\{0,-2\} \oplus  H\left(\left(%
\begin{array}{cc}
  0_{2,2N} & 0 \\
  0_{2,2N-2} & 0 \\
  0_{2,2N-2} & 0 
\end{array}%
\right)_{\Q}\right).
\end{eqnarray*}
From this, we get
\begin{equation}\label{eq-RIII-gdim-Gamma}
\gdim_{R_\partial} \fC_N(\Gamma) = (1+\xi^{-2})(1+\tau\alpha^{-1}\xi^{-N+1})(1+\tau\alpha^{-1}\xi^{-N+3})^2.
\end{equation}
Equation \eqref{eq-RIII-gdim-Gamma} implies that $\dim_\Q H_{R_\partial}(\fC_N(\Gamma))=16$. So $\fC_N(\Gamma)$ is homotopically finite over $R_\partial$. Moreover, comparing equations \eqref{eq-RIII-gdim-Gamma0}, \eqref{eq-RIII-gdim-Gamma1} to \eqref{eq-RIII-gdim-Gamma}, we have $\gdim_{R_\partial} \fC_N(\Gamma) = \gdim_{R_\partial} \fC_N(\Gamma_0) + \gdim_{R_\partial} \fC_N(\Gamma_1)$.
\end{proof}

\begin{lemma}\label{lemma-RIIIinv-decomp-hmf}
In the category $\hmf_{R_\partial,w}$, we have
\begin{eqnarray*}
\Hom_\hmf(\fC_N(\Gamma_0),\fC_N(\Gamma_0)) \cong \Hom_\hmf(\fC_N(\Gamma_1),\fC_N(\Gamma_1)) & \cong & \Q, \\
\Hom_\hmf(\fC_N(\Gamma_0),\fC_N(\Gamma_1)) \cong \Hom_\hmf(\fC_N(\Gamma_1),\fC_N(\Gamma_0)) & \cong & 0.
\end{eqnarray*}
\end{lemma}

\begin{proof}
We only prove that $\Hom_\hmf(\fC_N(\Gamma_1),\fC_N(\Gamma_1)) \cong  \Q$ and $\Hom_\hmf(\fC_N(\Gamma_1),\fC_N(\Gamma_0)) \cong 0$ here. The proofs of the other two isomorphisms are very similar and left to the reader. 

First, we compute  $\Hom_\hmf(\fC_N(\Gamma_1),\fC_N(\Gamma_1))$. Since $\gdim_{R_\partial} \fC_N(\Gamma_1) \neq 0$, we know that $\fC_N(\Gamma_1)$ is not homotopic to $0$. So $\id_{\fC_N(\Gamma_1)}$ is not homotopic to $0$. This implies that $\dim_\Q \Hom_\hmf(\fC_N(\Gamma_1),\fC_N(\Gamma_1)) \geq 1$. By equation \eqref{eq-RIII-mf-Gamma1} and Lemma \ref{lemma-dual-Koszul}, we have
\[
\Hom_{R_\partial}(\fC_N(\Gamma_1),\fC_N(\Gamma_1)) \simeq \left(%
\begin{array}{cc}
  \ast_{2,2N} & x_1+x_2+x_3-x_4-x_5-x_6 \\
  \ast_{2,2N-2} & x_1x_2+x_2x_3+x_3x_1-x_4x_5-x_5x_6-x_6x_4 \\
  \ast_{2,2N-4} & x_1x_2x_3-x_4x_5x_6 \\
  \ast_{2,2N} & -(x_1+x_2+x_3-x_4-x_5-x_6) \\
  \ast_{2,2N-2} & -(x_1x_2+x_2x_3+x_3x_1-x_4x_5-x_5x_6-x_6x_4) \\
  \ast_{2,2N-4} & -(x_1x_2x_3-x_4x_5x_6) 
\end{array}%
\right)_{R_\partial}\left\langle 3\right\rangle \{3,3N-9\}.
\]
Applying Proposition \ref{prop-contraction-weak} to the top three rows of the right hand side, we get
\[
\Hom_{\hmf}(\fC_N(\Gamma_1),\fC_N(\Gamma_1)) \cong H^{0,0,0}\left(\left(%
\begin{array}{cc}
  \ast_{2,2N} & 0 \\
  \ast_{2,2N-2} & 0 \\
  \ast_{2,2N-4} & 0 
\end{array}%
\right)_{R'}\left\langle 3\right\rangle \{3,3N-9\}\right),
\]
where $R'=R_\partial/(x_1+x_2+x_3-x_4-x_5-x_6,x_1x_2+x_2x_3+x_3x_1-x_4x_5-x_5x_6-x_6x_4,x_1x_2x_3-x_4x_5x_6)$. As a $\zed_2 \oplus \zed^{\oplus 2}$-graded $R'$-module, 
\[
\left(%
\begin{array}{cc}
  \ast_{2,2N} & 0 \\
  \ast_{2,2N-2} & 0 \\
  \ast_{2,2N-4} & 0 
\end{array}%
\right)_{R'}\left\langle 3\right\rangle \{3,3N-9\} \cong
(R'\left\langle 1\right\rangle \{1,N-1\} \oplus R') \otimes_{R'} (R'\left\langle 1\right\rangle \{1,N-3\} \oplus R') \otimes_{R'} (R'\left\langle 1\right\rangle \{1,N-5\} \oplus R'),
\]
whose homogeneous component of $\zed_2 \oplus \zed^{\oplus 2}$-degree $(0,0,0)$ is $1$-dimensional over $\Q$. This implies that $\dim_\Q \Hom_\hmf(\fC_N(\Gamma_1),\fC_N(\Gamma_1)) \leq 1$. Thus, $\Hom_\hmf(\fC_N(\Gamma_1),\fC_N(\Gamma_1)) \cong  \Q$.

Similarly, by equations \eqref{eq-RIII-mf-Gamma0}, \eqref{eq-RIII-mf-Gamma1}, Lemma \ref{lemma-dual-Koszul} and Proposition \ref{prop-b-contraction}, we have
\begin{eqnarray*}
\Hom_{R_\partial}(\fC_N(\Gamma_1),\fC_N(\Gamma_0)) & \simeq & \left(%
\begin{array}{cc}
  \ast_{2,2N} & x_1+x_2-x_4-x_5 \\
  \ast_{2,2N-2} & x_1x_2-x_4x_5 \\
  \ast_{2,2N} & x_3-x_6 \\
  \ast_{2,2N} & -(x_1+x_2+x_3-x_4-x_5-x_6) \\
  \ast_{2,2N-2} & -(x_1x_2+x_2x_3+x_3x_1-x_4x_5-x_5x_6-x_6x_4) \\
  \ast_{2,2N-4} & -(x_1x_2x_3-x_4x_5x_6) 
\end{array}%
\right)_{R_\partial}\left\langle 3\right\rangle \{3,3N-7\} \\
& \simeq & \left(%
\begin{array}{cc}
  \ast_{2,2N} & 0 \\
  \ast_{2,2N-2} & 0 \\
  \ast_{2,2N-4} & 0 
\end{array}%
\right)_{\hat{R}}\left\langle 3\right\rangle \{3,3N-7\},
\end{eqnarray*}
where $\hat{R}=R_\partial/(x_1+x_2-x_4-x_5, x_1x_2-x_4x_5,x_3-x_6)\cong\Q[a,x_3]\otimes_\Q\Sym(\{x_1,x_2\})$. As a $\zed_2 \oplus \zed^{\oplus 2}$-graded $\hat{R}$-module, 
\[
\left(%
\begin{array}{cc}
  \ast_{2,2N} & 0 \\
  \ast_{2,2N-2} & 0 \\
  \ast_{2,2N-4} & 0 
\end{array}%
\right)_{\hat{R}}\left\langle 3\right\rangle \{3,3N-7\} \cong
(\hat{R}\left\langle 1\right\rangle \{1,N-1\} \oplus \hat{R}) \otimes_{\hat{R}} (\hat{R}\left\langle 1\right\rangle \{1,N-3\} \oplus \hat{R}) \otimes_{\hat{R}} (\hat{R}\left\langle 1\right\rangle \{1,N-3\} \oplus \hat{R}\{0,2\}),
\]
whose homogeneous component of $\zed_2 \oplus \zed^{\oplus 2}$-degree $(0,0,0)$ vanishes. This implies that
\[
\Hom_\hmf(\fC_N(\Gamma_1),\fC_N(\Gamma_0)) \cong H^{0,0,0}\left(\left(%
\begin{array}{cc}
  \ast_{2,2N} & 0 \\
  \ast_{2,2N-2} & 0 \\
  \ast_{2,2N-4} & 0 
\end{array}%
\right)_{\hat{R}}\left\langle 3\right\rangle \{3,3N-7\}\right) \cong 0.
\]
\end{proof}

We are now ready to prove Proposition \ref{prop-RIIIinv-decomp}. 

\begin{proof}[Proof of Proposition \ref{prop-RIIIinv-decomp}]
\begin{figure}[ht]
$
\xymatrix{
\input{RIII-decomp0} \ar@<1ex>[rd]^<<<<<<<<<<<<<<<{\phi} \ar@<10ex>[rr]^<<<<<<<<<<<<<<<<<<<<{f} && \input{RIII-decomp} \ar@<-8ex>[ll]^>>>>>>>>>>>>>>>>>>>>{\bar{f}} \ar@<1ex>[ld]^<<<<<<<<<<<<<<<{\chi^1} \\
& \input{RIII-decomp2} \ar@<1ex>[lu]^>>>>>>>>>>>>>>>{\bar{\phi}} \ar@<1ex>[ru]^>>>>>>>>>>>>>>>{\chi^0} &
}
$
\caption{}\label{RIIIinv-decomp-f-fig}

\end{figure}

Define morphisms $f:\fC_N(\Gamma_0) \rightarrow \fC_N(\Gamma)$ and $\bar{f}:\fC_N(\Gamma) \rightarrow \fC_N(\Gamma_0)$ by Figure \ref{RIIIinv-decomp-f-fig}. That is, $f = \chi^0\circ \phi$ and $\bar{f} = \bar{\phi}\circ \chi^1$, where 
\begin{itemize}
	\item $\phi$ and $\bar{\phi}$ are the morphisms associated to the edge splitting/merging in the left side of $\Gamma_0$ and $\Gamma_2$ defined in Lemma \ref{lemma-phi},
	\item $\chi^0$ and $\chi^1$ are the $\chi$-morphisms associated to the right side of $\Gamma$ and $\Gamma_2$ defined in Lemma \ref{lemma-def-chi}.
\end{itemize}
Note that 
\begin{itemize}
	\item $\phi$ and $\bar{\phi}$ are homogeneous morphisms of $\zed_2 \oplus \zed^{\oplus 2}$-degree $(0,0,-1)$,
	\item $\chi^0$ and $\chi^1$ are homogeneous morphisms of $\zed_2 \oplus \zed^{\oplus 2}$-degree $(0,0,1)$.
\end{itemize}
So $f$ and $\bar{f}$ are homogeneous morphisms of $\zed_2 \oplus \zed^{\oplus 2}$-degree $(0,0,0)$. Using Lemmas \ref{lemma-phi} and \ref{lemma-def-chi} again, we get
\begin{equation}\label{eq-RIIIinv-decomp-f-comp}
\bar{f} \circ f = \bar{\phi}\circ \chi^1 \circ \chi^0\circ \phi \simeq \bar{\phi}\circ \mathsf{m}(x_6-x_8) \circ \phi = -\bar{\phi}\circ \mathsf{m}(x_8) \circ \phi + \mathsf{m}(x_6) \circ \bar{\phi}\circ \phi \approx \id_{\fC_N(\Gamma_0)}.
\end{equation}

\begin{figure}[ht]
$
\xymatrix{
\input{RIII-decomp1} \ar@<1ex>[d]^>>>>>{\varphi} \ar@<10ex>[rr]^{g} && \input{RIII-decomp} \ar@<-8ex>[ll]^{\bar{g}} \ar@<1ex>[d]^>>>>>{\tilde{\chi}^0} \\
\input{RIII-decomp3} \ar@<1ex>[u]^<<<<<{\bar{\varphi}} \ar@<10ex>[rr]^{h} && \input{RIII-decomp3-prime} \ar@<1ex>[u]^<<<<<{\tilde{\chi}^1} \ar@<-8ex>[ll]^{\bar{h}}
}
$
\caption{}\label{RIIIinv-decomp-g-fig}

\end{figure}

Next, define $g:\fC_N(\Gamma_1) \rightarrow \fC_N(\Gamma)$ and $\bar{g}:\fC_N(\Gamma) \rightarrow \fC_N(\Gamma_1)$ by Figure \ref{RIIIinv-decomp-g-fig}. That is, $g= \tilde{\chi}^1 \circ h \circ \varphi$ and $\bar{g} = \bar{\varphi} \circ \bar{h} \circ \tilde{\chi}^0$, where
\begin{itemize}
	\item $\varphi$ and $\bar{\varphi}$ are the morphisms associated to the splitting/merging of the upper-left $2$-colored edge in $\Gamma_1$ defined in Lemma \ref{lemma-phi},
	\item $h$ and $\bar{h}$ are the homotopy equivalences induced by the edge sliding given in Lemmas \ref{lemma-edge-sliding}, \ref{lemma-edge-sliding-unique} and are homotopy inverses of each other,
	\item $\tilde{\chi}^0$ and  $\tilde{\chi}^1$ are the $\tilde{\chi}$-morphisms associated to the lower half of $\Gamma$ and $\Gamma_3'$ defined in Lemma \ref{lemma-def-tilde-chi}.
\end{itemize}
Note that 
\begin{itemize}
	\item $\varphi$ and $\bar{\varphi}$ are homogeneous morphisms of $\zed_2 \oplus \zed^{\oplus 2}$-degree $(0,0,-1)$,
	\item $h$ and $\bar{h}$ are homogeneous morphisms of $\zed_2 \oplus \zed^{\oplus 2}$-degree $(0,0,0)$,
	\item $\tilde{\chi}^0$ and $\tilde{\chi}^1$ are homogeneous morphisms of $\zed_2 \oplus \zed^{\oplus 2}$-degree $(0,0,1)$.
\end{itemize}
So $g$ and $\bar{g}$ are homogeneous morphisms of $\zed_2 \oplus \zed^{\oplus 2}$-degree $(0,0,0)$. Using Lemmas \ref{lemma-phi} and \ref{lemma-def-tilde-chi} again, we get
\begin{eqnarray}
\label{eq-RIIIinv-decomp-g-comp} \bar{g} \circ g & = & \bar{\varphi} \circ \bar{h} \circ \tilde{\chi}^0 \circ \tilde{\chi}^1 \circ h \circ \varphi \simeq \bar{\varphi} \circ \bar{h} \circ \mathsf{m}(x_6-x_7) \circ h \circ \varphi = \bar{\varphi} \circ \mathsf{m}(x_6-x_7) \circ \bar{h} \circ h \circ \varphi \\
\nonumber & \simeq & \bar{\varphi} \circ \mathsf{m}(x_6-x_7) \circ \varphi = \mathsf{m}(x_6) \circ \bar{\varphi} \circ \varphi - \bar{\varphi} \circ \mathsf{m}(x_7) \circ \varphi \approx \id_{\fC_N(\Gamma_1)}.
\end{eqnarray}

From Lemma \ref{lemma-RIIIinv-decomp-hmf}, we know that $\Hom_\hmf(\fC_N(\Gamma_0),\fC_N(\Gamma_1)) \cong \Hom_\hmf(\fC_N(\Gamma_1),\fC_N(\Gamma_0)) \cong 0$. So 
\begin{equation}\label{eq-RIIIinv-decomp-fg-comp}
\bar{f} \circ g \simeq 0 \hspace{2pc} \text{ and } \hspace{2pc} \bar{g} \circ f \simeq 0.
\end{equation}

Now consider the morphisms
\[
\xymatrix{
\fC_N(\Gamma) \ar@<1ex>[rr]^{\left(%
\begin{array}{c}
\bar{f} \\
\bar{g}
\end{array}%
\right)}&& {\left.%
\begin{array}{c}
\fC_N(\Gamma_0) \\
\oplus \\
\fC_N(\Gamma_1)
\end{array}%
\right.} \ar@<1ex>[ll]^{\left(%
\begin{array}{cc}
f,&g
\end{array}%
\right)}
}.
\]
From homotopies \eqref{eq-RIIIinv-decomp-f-comp}, \eqref{eq-RIIIinv-decomp-g-comp} and \eqref{eq-RIIIinv-decomp-fg-comp}, we know that, after possibly scaling $f$ and $g$, 
\[
{\left(%
\begin{array}{c}
\bar{f} \\
\bar{g}
\end{array}%
\right)}
\circ
{\left(%
\begin{array}{cc}
f,&g
\end{array}%
\right)}
\simeq 
{\left(%
\begin{array}{cc}
\id_{\fC_N(\Gamma_0)} & 0 \\
0 & \id_{\fC_N(\Gamma_1)}
\end{array}%
\right)}.
\]
Recall that, by Proposition \ref{fully-additive-hmf}, $\hmf_{R_\partial,w}$ is fully additive. So, by Lemma \ref{lemma-fully-additive-split}, there exists an object $M$ of $\hmf_{R_\partial,w}$ such that $\fC_N(\Gamma) \cong \fC_N(\Gamma_0) \oplus \fC_N(\Gamma_1) \oplus M$. From Lemma \ref{lemma-RIIIinv-decomp-obj}, we know that 
\[
\gdim_{R_\partial} M = \gdim_{R_\partial} \fC_N(\Gamma) - \gdim_{R_\partial} \fC_N(\Gamma_0) - \gdim_{R_\partial} \fC_N(\Gamma_1) =0.
\] 
By Corollary \ref{cor-homology-detects-homotopy}, this implies that $M \simeq 0$. Thus, $\fC_N(\Gamma) \cong \fC_N(\Gamma_0) \oplus \fC_N(\Gamma_1)$.
\end{proof}

\begin{figure}[ht]
$
\xymatrix{
\input{RIII-decomp1} \ar@<1ex>[d]^>>>>>{\psi} \ar@<10ex>[rr]^{\mathfrak{g}} && \input{RIII-decomp} \ar@<-8ex>[ll]^{\bar{\mathfrak{g}}} \ar@<1ex>[d]^>>>>>{\tilde{\chi}^0_u} \\
\input{RIII-decomp4} \ar@<1ex>[u]^<<<<<{\bar{\psi}} \ar@<10ex>[rr]^{\mathfrak{h}} && \input{RIII-decomp4-prime} \ar@<1ex>[u]^<<<<<{\tilde{\chi}^1_u} \ar@<-8ex>[ll]^{\bar{\mathfrak{h}}}
}
$
\caption{}\label{RIIIinv-decomp-g-fig-2}

\end{figure}

\begin{corollary}\label{cor-RIIIinv-decomp-g-def2}
Let $\mathfrak{g}:\fC_N(\Gamma_1) \rightarrow \fC_N(\Gamma)$ and $\bar{\mathfrak{g}}: \fC_N(\Gamma) \rightarrow \fC_N(\Gamma_1)$ be the morphisms defined by Figure \ref{RIIIinv-decomp-g-fig-2}. That is, $\mathfrak{g} = \tilde{\chi}^1_u \circ \mathfrak{h} \circ \psi$ and $\bar{\mathfrak{g}} = \bar{\psi} \circ \bar{\mathfrak{h}} \circ \tilde{\chi}^0_u$, where
\begin{itemize}
	\item $\psi$ and $\bar{\psi}$ are the morphisms associated to the splitting/merging of the lower-left $2$-colored edge in $\Gamma_1$ defined in Lemma \ref{lemma-phi},
	\item $\mathfrak{h}$ and $\bar{\mathfrak{h}}$ are the homotopy equivalences induced by the edge sliding given in Lemmas \ref{lemma-edge-sliding}, \ref{lemma-edge-sliding-unique} and are homotopy inverses of each other,
	\item $\tilde{\chi}^0_u$ and  $\tilde{\chi}^1_u$ are the $\tilde{\chi}$-morphisms associated to the upper half of $\Gamma$ and $\Gamma_3'$ defined in Lemma \ref{lemma-def-tilde-chi}.
\end{itemize}
Then $\mathfrak{g} \approx g$ and $\bar{\mathfrak{g}} \approx \bar{g}$, where $g$ and $\bar{g}$ are the morphisms defined by Figure \ref{RIIIinv-decomp-g-fig}.
\end{corollary}

\begin{proof}
Similar to the proof of Proposition \ref{prop-RIIIinv-decomp}, one can check that $\mathfrak{g}$ and $\bar{\mathfrak{g}}$ are homogeneous morphisms of $\zed_2 \oplus \zed^{\oplus 2}$-degree $(0,0,0)$ satisfying $\bar{\mathfrak{g}} \circ \mathfrak{g} \approx \id _{\fC_N(\Gamma_1)}$. Thus, $g$, $\bar{g}$, $\mathfrak{g}$ and $\bar{\mathfrak{g}}$ are all homotopically non-trivial homogeneous morphisms of $\zed_2 \oplus \zed^{\oplus 2}$-degree $(0,0,0)$. By Proposition \ref{prop-RIIIinv-decomp} and Lemma \ref{lemma-RIIIinv-decomp-hmf}, we have that 
\[
\Hom_\hmf(\fC_N(\Gamma),\fC_N(\Gamma_1)) \cong \Hom_\hmf(\fC_N(\Gamma_0),\fC_N(\Gamma_1)) \oplus \Hom_\hmf(\fC_N(\Gamma_1),\fC_N(\Gamma_1)) \cong \Q.
\]
This implies that $\bar{\mathfrak{g}} \approx \bar{g}$. Similarly, we have $\Hom_\hmf(\fC_N(\Gamma_1),\fC_N(\Gamma)) \cong \Q$, which implies that $\mathfrak{g} \approx g$.
\end{proof}

\subsection{Invariance under braid-like Reidemeister III moves}

\begin{figure}[ht]
$
\xymatrix{
\input{RIII1} && \input{RIII2}
}
$
\caption{}\label{RIIIinv-fig}

\end{figure}

\begin{proposition}\label{prop-RIIIinv}
Let $T_1$ and $T_2$ be the tangle diagrams in Figure \ref{RIIIinv-fig}. Then, for $N\geq 0$, $\fC_N(T_1) \simeq \fC_N(T_2)$ as chain complexes over the category $\hmf_{R_\partial,w}$, where 
\begin{eqnarray*}
R_\partial & = & \Q[a,x_1,x_2,x_3,x_4,x_5,x_6], \\
w & = & a(x_1^{N+1}+x_2^{N+1}+x_3^{N+1}-x_4^{N+1}-x_5^{N+1}-x_6^{N+1}).
\end{eqnarray*}
\end{proposition}

\begin{figure}[ht]
$
\xymatrix{
 & \input{RIII-res110} & \input{RIII-res100} &\\
\input{RIII-res111} & \input{RIII-res101} & \input{RIII-res010} & \input{RIII-res000}\\
 & \input{RIII-res011} & \input{RIII-res001} &
}
$
\caption{}\label{RIIIinv-res-fig}

\end{figure}

Let us consider $T_1$. Its resolutions are listed in Figure \ref{RIIIinv-res-fig}. We call the three crossings in $T_1$ the upper crossing, the lower crossing, the right crossing and denote by 
\begin{itemize}
	\item $\chi_u^0$ and $\chi_u^1$ the $\chi$-morphisms associated to the upper crossing,
	\item $\chi_l^0$ and $\chi_l^1$ the $\chi$-morphisms associated to the lower crossing,
	\item $\chi_r^0$ and $\chi_r^1$ the $\chi$-morphisms associated to the right crossing.
\end{itemize}
Then 
{\tiny
\begin{equation}\label{eq-RIII-T1-chain}
\fC_N(T_1)= 0 \rightarrow \underbrace{\fC_N(\Gamma_{111})\left\langle 3\right\rangle\{3,3N\}}_{-3} \xrightarrow{d_{-3}} \underbrace{{\left.%
\begin{array}{c}
\fC_N(\Gamma_{110})\left\langle 3\right\rangle\{3,3N-1\} \\
\oplus \\
\fC_N(\Gamma_{101})\left\langle 3\right\rangle\{3,3N-1\} \\
\oplus \\
\fC_N(\Gamma_{011})\left\langle 3\right\rangle\{3,3N-1\}
\end{array}%
\right.}}_{-2}  \xrightarrow{d_{-2}} \underbrace{{\left.%
\begin{array}{c}
\fC_N(\Gamma_{100})\left\langle 3\right\rangle\{3,3N-2\} \\
\oplus \\
\fC_N(\Gamma_{010})\left\langle 3\right\rangle\{3,3N-2\} \\
\oplus \\
\fC_N(\Gamma_{001})\left\langle 3\right\rangle\{3,3N-2\}
\end{array}%
\right.}}_{-1}   \xrightarrow{d_{-1}} \underbrace{\fC_N(\Gamma_{000})\left\langle 3\right\rangle\{3,3N-3\}}_{0} \rightarrow 0,
\end{equation}}

\noindent where 
\begin{eqnarray*}
d_{-3} & = & {\left(%
\begin{array}{c}
\chi^1_r \\
-\chi^1_l \\
\chi^1_u
\end{array}%
\right)}, \\
d_{-2} & = &  {\left(%
\begin{array}{ccc}
-\chi^1_l & -\chi^1_r & 0 \\
\chi^1_u & 0 & -\chi^1_r \\
0 & \chi^1_u & \chi^1_l
\end{array}%
\right)}, \\
d_{-1} & = & (\chi^1_u,\chi^1_l,\chi^1_r).
\end{eqnarray*}
From Lemma \ref{lemma-RIIIinv-decomp-obj}, we know that $\fC_N(\Gamma_{111})$ is an object of $\hmf_{R_\partial,w}$. It is straightforward to verify that $\fC_N(\Gamma_{\ve\mu\nu})$ is also an object of $\hmf_{R_\partial,w}$ for any other resolution $\Gamma_{\ve\mu\nu}$ of $T_1$. Thus, $\fC_N(T_1)$ is a chain complex over $\hmf_{R_\partial,w}$.

\begin{figure}[ht]
$
\xymatrix{
\input{RIII-res0} && \input{RIII-res1}
}
$
\caption{}\label{RIIIinv-res-fig-2}

\end{figure}

Next, we consider the MOY graphs in Figure \ref{RIIIinv-res-fig-2}. It is obvious that 
\begin{equation}\label{eq-RIII-Gamma0}
\Gamma_0 = \Gamma_{100} = \Gamma_{010}.
\end{equation} 
By Lemma \ref{lemma-decomp-II}, we know that
\begin{equation}\label{eq-RIII-decompII}
\fC_N(\Gamma_{110}) \simeq \fC_N(\Gamma_{0})\{0,1\} \oplus \fC_N(\Gamma_{0})\{0,-1\}.
\end{equation}
From the proof of Proposition \ref{prop-RIIinv}, we know that homotopy equivalence \eqref{eq-RIII-decompII} is given by a pair of morphisms of the form
\begin{equation}\label{eq-RIII-decompII-morph}
\xymatrix{
\fC_N(\Gamma_{110}) \ar@<1ex>[rr]^{\left(%
\begin{array}{c}
\bar{\phi} \\
P
\end{array}%
\right)} &&  {\left.%
\begin{array}{c}
\fC_N(\Gamma_{0})\{0,1\} \\
\oplus \\
\fC_N(\Gamma_{0})\{0,-1\}
\end{array}%
\right.}  \ar@<1ex>[ll]^{\left(%
\begin{array}{cc}
J,& \phi
\end{array}%
\right)},
}
\end{equation}
where $\phi$ and $\bar{\phi}$ are the morphisms associated to the splitting and merging of the $2$-colored edge in $\Gamma_0$. By Proposition \ref{prop-RIIIinv-decomp} and the proof of Lemma \ref{lemma-edge-sliding}, we know that 
\begin{equation}\label{eq-RIII-decompIV}
\fC_N(\Gamma_{111}) \simeq \fC_N(\Gamma_{0}) \oplus \fC_N(\Gamma_{1}).
\end{equation}
From the proof of Proposition \ref{prop-RIIIinv-decomp}, we know that homotopy equivalence \eqref{eq-RIII-decompIV} is given by the morphisms
\begin{equation}\label{eq-RIII-decompIV-morph}
\xymatrix{
\fC_N(\Gamma_{111}) \ar@<1ex>[rr]^{\left(%
\begin{array}{c}
\bar{f} \\
\bar{g}
\end{array}%
\right)} &&  {\left.%
\begin{array}{c}
\fC_N(\Gamma_{0})\\
\oplus \\
\fC_N(\Gamma_{1})
\end{array}%
\right.}  \ar@<1ex>[ll]^{\left(%
\begin{array}{cc}
f,& g
\end{array}%
\right)},
}
\end{equation}
where $f$, $\bar{f}$ and $g$, $\bar{g}$ are the morphisms given by Figures \ref{RIIIinv-decomp-f-fig} and \ref{RIIIinv-decomp-g-fig}.

Substituting \eqref{eq-RIII-Gamma0}-\eqref{eq-RIII-decompIV-morph} into $\fC_N(T_1)$, we get that, as a chain complex over $\hmf_{R_\partial,w}$, 
{\tiny
\begin{equation}\label{eq-RIII-T1-chain2}
\fC_N(T_1)\cong 0 \rightarrow \underbrace{\left.%
\begin{array}{c}
\fC_N(\Gamma_{0})\left\langle 3\right\rangle\{3,3N\}\\\
\oplus \\
\fC_N(\Gamma_{1})\left\langle 3\right\rangle\{3,3N\}\
\end{array}%
\right.}_{-3} \xrightarrow{\tilde{d}_{-3}} \underbrace{{\left.%
\begin{array}{c}
\fC_N(\Gamma_{0})\{3,3N\} \\
\oplus \\\fC_N(\Gamma_{0})\{3,3N-2\} \\
\oplus \\
\fC_N(\Gamma_{101})\left\langle 3\right\rangle\{3,3N-1\} \\
\oplus \\
\fC_N(\Gamma_{011})\left\langle 3\right\rangle\{3,3N-1\}
\end{array}%
\right.}}_{-2}  \xrightarrow{\tilde{d}_{-2}} \underbrace{{\left.%
\begin{array}{c}
\fC_N(\Gamma_{0})\left\langle 3\right\rangle\{3,3N-2\} \\
\oplus \\
\fC_N(\Gamma_{0})\left\langle 3\right\rangle\{3,3N-2\} \\
\oplus \\
\fC_N(\Gamma_{001})\left\langle 3\right\rangle\{3,3N-2\}
\end{array}%
\right.}}_{-1}   \xrightarrow{\tilde{d}_{-1}} \underbrace{\fC_N(\Gamma_{000})\left\langle 3\right\rangle\{3,3N-3\}}_{0} \rightarrow 0,
\end{equation}}

\noindent where 
\begin{eqnarray*}
\tilde{d}_{-3} & =  & {\left(%
\begin{array}{cc}
\bar{\phi}\circ\chi^1_r \circ f & \bar{\phi}\circ\chi^1_r \circ g \\
P\circ\chi^1_r \circ f & P\circ\chi^1_r \circ g \\
-\chi^1_l \circ f & -\chi^1_l \circ g\\
\chi^1_u \circ f & \chi^1_u \circ g
\end{array}%
\right)}, \\
\tilde{d}_{-2} & =  & {\left(%
\begin{array}{cccc}
-\chi^1_l\circ J & -\chi^1_l\circ \phi & -\chi^1_r & 0 \\
\chi^1_u \circ J & \chi^1_u \circ \phi & 0 & -\chi^1_r \\
0 & 0& \chi^1_u & \chi^1_l
\end{array}%
\right)}, \\
\tilde{d}_{-1} & =  & (\chi^1,\chi^1,\chi^1_r),
\end{eqnarray*}
and the morphism $\chi^1$ in $\tilde{d}_{-1}$ is the $\chi^1$-morphism associated to the $2$-colored edge in $\Gamma_0$.

Similar to the proof of Lemma \ref{lemma-RIIIinv-decomp-hmf}, one can check that, in $\hmf_{R_\partial,w}$, $\Hom_\hmf(\fC_N(\Gamma_1),\fC_N(\Gamma_0)) \cong 0$. This implies that the entry $\bar{\phi}\circ\chi^1_r \circ g$ of $\tilde{d}_{-3}$ is homotopic to $0$. Recall that, by Figure \ref{RIIIinv-decomp-f-fig}, $f$ is defined to be the composition $\fC_N(\Gamma_0) \xrightarrow{\phi} \fC_N(\Gamma_{110}) \xrightarrow{\chi_r^0} \fC_N(\Gamma_{111})$. So we have that, by Lemmas \ref{lemma-phi} and \ref{lemma-def-chi},
\[
\bar{\phi}\circ\chi^1_r \circ f = \bar{\phi}\circ\chi^1_r \circ \chi^0_r \circ \phi \simeq \bar{\phi}\circ \mathsf{m}(x_6-x_8) \circ \phi = \mathsf{m}(x_6) \circ\bar{\phi}\circ\phi - \bar{\phi}\circ \mathsf{m}(x_8) \circ \phi \approx \id_{\fC_N(\Gamma_0)}.
\]
Now we apply Gaussian Elimination(Lemma \ref{lemma-gaussian-elimination}) to the homotopy equivalence $\bar{\phi}\circ\chi^1_r \circ f$ in $\tilde{d}_{-3}$. Note that, since $\bar{\phi}\circ\chi^1_r \circ g \simeq 0$, the correction term from Lemma \ref{lemma-gaussian-elimination} is $0$ in this case. Thus, as a chain complex over $\hmf_{R_\partial,w}$,
{\tiny
\begin{equation}\label{eq-RIII-T1-chain3}
\fC_N(T_1) \simeq  0 \rightarrow \underbrace{
\fC_N(\Gamma_{1})\left\langle 3\right\rangle\{3,3N\}}_{-3} \xrightarrow{\hat{d}_{-3}} \underbrace{{\left.%
\begin{array}{c}
\fC_N(\Gamma_{0})\{3,3N-2\} \\
\oplus \\
\fC_N(\Gamma_{101})\left\langle 3\right\rangle\{3,3N-1\} \\
\oplus \\
\fC_N(\Gamma_{011})\left\langle 3\right\rangle\{3,3N-1\}
\end{array}%
\right.}}_{-2}  \xrightarrow{\hat{d}_{-2}} \underbrace{{\left.%
\begin{array}{c}
\fC_N(\Gamma_{0})\left\langle 3\right\rangle\{3,3N-2\} \\
\oplus \\
\fC_N(\Gamma_{0})\left\langle 3\right\rangle\{3,3N-2\} \\
\oplus \\
\fC_N(\Gamma_{001})\left\langle 3\right\rangle\{3,3N-2\}
\end{array}%
\right.}}_{-1}   \xrightarrow{\hat{d}_{-1}} \underbrace{\fC_N(\Gamma_{000})\left\langle 3\right\rangle\{3,3N-3\}}_{0} \rightarrow 0,
\end{equation}}

\noindent where 
\begin{eqnarray*}
\hat{d}_{-3} & =  & {\left(%
\begin{array}{c}
 P\circ\chi^1_r \circ g \\
 -\chi^1_l \circ g\\
 \chi^1_u \circ g
\end{array}%
\right)}, \\
\hat{d}_{-2} & =  & {\left(%
\begin{array}{ccc}
 -\chi^1_l\circ \phi & -\chi^1_r & 0 \\
 \chi^1_u \circ \phi & 0 & -\chi^1_r \\
 0& \chi^1_u & \chi^1_l
\end{array}%
\right)}, \\
\hat{d}_{-1} & =  & (\chi^1,\chi^1,\chi^1_r),
\end{eqnarray*}
and the morphism $\chi^1$ in $\hat{d}_{-1}$ is again the $\chi^1$-morphism associated to the $2$-colored edge in $\Gamma_0$.

Consider the entries $-\chi^1_l\circ \phi$ and $\chi^1_u \circ \phi$ in $\hat{d}_{-2}$. Applying the argument used to establish homotopies \eqref{eq-phi-bar-chi0} and \eqref{eq-chi1-phi}, we get that $-\chi^1_l\circ \phi\approx \chi^1_u \circ \phi \approx \id_{\fC_N(\Gamma_0)}$. Applying Gaussian Elimination(Lemma \ref{lemma-gaussian-elimination}) to the homotopy equivalence $-\chi^1_l\circ \phi$, we get that as a chain complex over $\hmf_{R_\partial,w}$,
{\tiny
\begin{equation}\label{eq-RIII-T1-chain4}
\fC_N(T_1) \simeq  0 \rightarrow \underbrace{
\fC_N(\Gamma_{1})\left\langle 3\right\rangle\{3,3N\}}_{-3} \xrightarrow{\check{d}_{-3}} \underbrace{{\left.%
\begin{array}{c}
\fC_N(\Gamma_{101})\left\langle 3\right\rangle\{3,3N-1\} \\
\oplus \\
\fC_N(\Gamma_{011})\left\langle 3\right\rangle\{3,3N-1\}
\end{array}%
\right.}}_{-2}  \xrightarrow{\check{d}_{-2}} \underbrace{{\left.%
\begin{array}{c}
\fC_N(\Gamma_{0})\left\langle 3\right\rangle\{3,3N-2\} \\
\oplus \\
\fC_N(\Gamma_{001})\left\langle 3\right\rangle\{3,3N-2\}
\end{array}%
\right.}}_{-1}   \xrightarrow{\check{d}_{-1}} \underbrace{\fC_N(\Gamma_{000})\left\langle 3\right\rangle\{3,3N-3\}}_{0} \rightarrow 0,
\end{equation}}

\noindent where 
\begin{eqnarray*}
\check{d}_{-3} & =  & {\left(%
\begin{array}{c}
 -\chi^1_l \circ g\\
 \chi^1_u \circ g
\end{array}%
\right)}, \\
\check{d}_{-2} & =  & {\left(%
\begin{array}{cc}
  c\cdot \chi^1_r & -\chi^1_r \\
  \chi^1_u & \chi^1_l
\end{array}%
\right)}, \\
\check{d}_{-1} & =  & (\chi^1,\chi^1_r),
\end{eqnarray*}
$c$ is a non-zero scalar\footnote{The entry $c\cdot \chi^1_r$ in $\check{d}_{-2}$ comes from the correction term in Lemma \ref{lemma-gaussian-elimination}.}, and the morphism $\chi^1$ in $\check{d}_{-1}$ is once more the $\chi^1$-morphism associated to the $2$-colored edge in $\Gamma_0$.

Schematically, we represent chain complex \eqref{eq-RIII-T1-chain4} by the diagram
\begin{equation}\label{eq-RIII-T1-chain5}
\xymatrix{
 & \Gamma_{101} \ar[r] \ar[rdd] & \Gamma_{0} \ar[rd] & \\
\Gamma_{1} \ar[ru] \ar[rd]& && \Gamma_{000}, \\
& \Gamma_{011} \ar[r] \ar[ruu] & \Gamma_{001} \ar[ru] &
}
\end{equation}
in which each arrow represents the corresponding entry in the above matrix presentation of the differential map $\check{d}$. The following two lemmas are straightforward adaptations of \cite[Lemmas 26 and 27]{KR1}.

\begin{lemma}\cite[Lemmas 27]{KR1}\label{lemma-RIIIinv-arrow-comp}
The composition of any pair of consecutive arrows in diagram \eqref{eq-RIII-T1-chain5} is homotopically non-trivial.
\end{lemma}
\begin{proof}
Let $R=\Q[x_1,x_2,x_3,x_4,x_5,x_6]$ and $w_1= x_1^{N+1}+x_2^{N+1}+x_3^{N+1}-x_4^{N+1}-x_5^{N+1}-x_6^{N+1}$. The standard quotient map $\pi_1:R_\partial \rightarrow R_\partial/(a-1) \cong R$ induces a functor $\varpi_1:\hmf_{R_\partial,w} \rightarrow \hmf_{R,w_1}$ that takes an object $M$ of $\hmf_{R_\partial,w}$ to the object $M/(a-1)M$ of $\hmf_{R,w_1}$. Comparing the definitions in the current paper to those in \cite{KR1}, one can see that:
\begin{itemize}
	\item for any MOY graph $\Gamma$ in this subsection, $\varpi_1(\fC_N(\Gamma))$ is the matrix factorization associated to $\Gamma$ in \cite{KR1},
	\item for any morphism in this subsection, $\varpi_1$ maps it to the corresponding morphism in \cite{KR1}.
\end{itemize}
As a functor from $\hmf_{R_\partial,w}$ to $\hmf_{R,w_1}$, $\varpi_1$ maps homotopic morphisms to homotopic morphisms. From \cite[Lemmas 27]{KR1}, we know that the image under $\varpi_1$ of the composition of any pair of consecutive arrows in diagram \eqref{eq-RIII-T1-chain5} is homotopically non-trivial. This implies that the composition of any pair of consecutive arrows in diagram \eqref{eq-RIII-T1-chain5} is homotopically non-trivial.
\end{proof}

\begin{lemma}\cite[Lemmas 26]{KR1}\label{lemma-RIIIinv-arrow-hmf}
Assume that $\Gamma$ and $\Gamma'$ are two MOY graphs in diagram \eqref{eq-RIII-T1-chain5} and that there is an arrow pointing from $\Gamma$ to $\Gamma'$ in diagram \eqref{eq-RIII-T1-chain5}. Then, in the category $\hmf_{R_\partial,w}$, we have 
\[
\Hom_\hmf(\fC_N(\Gamma), \fC_N(\Gamma')\{0,-1\}) \cong \Q.
\]
In particular, this space is spanned over $\Q$ by the corresponding arrow in diagram \eqref{eq-RIII-T1-chain5}.
\end{lemma}

\begin{proof}
The arrow pointing from $\Gamma$ to $\Gamma'$ in diagram \eqref{eq-RIII-T1-chain5} is an element of $\Hom_\hmf(\fC_N(\Gamma), \fC_N(\Gamma')\{0,-1\})$ and, by Lemma \ref{lemma-RIIIinv-arrow-comp}, is homotopically non-trivial. Thus, $\dim_\Q \Hom_\hmf(\fC_N(\Gamma), \fC_N(\Gamma')\{0,-1\}) \geq 1$. It remains to show that $\dim_\Q \Hom_\hmf(\fC_N(\Gamma), \fC_N(\Gamma')\{0,-1\}) \leq 1$. We only check this for the pair $\Gamma_{101}$ and $\Gamma_0$. The proofs for the other pairs are similar and left to the reader.

By Definition \ref{def-MOY-mf} and Proposition \ref{prop-b-contraction},
\begin{eqnarray*}
\fC_N(\Gamma_{101}) & \simeq & {\left(%
\begin{array}{cc}
 \ast_{2,2N} & x_1+x_2-x_4-x_8\\
 \ast_{2,2N-2} & x_1x_2-x_4x_8\\
 \ast_{2,2N} & x_8+x_3-x_5-x_6\\
 \ast_{2,2N-2} & x_8x_3-x_5x_6
\end{array}%
\right)_{R_\partial[x_8]}\{0,-2\}} \\
& \simeq & {\left(%
\begin{array}{cc}
 \ast_{2,2N} & x_1+x_2+x_3-x_4-x_5-x_6\\
 \ast_{2,2N-2} & x_1x_2-x_4(x_5+x_6-x_3)\\
 \ast_{2,2N-2} & (x_5+x_6-x_3)x_3-x_5x_6
\end{array}%
\right)_{R_\partial}\{0,-2\}}
\end{eqnarray*}
and
\[
\fC_N(\Gamma_0) \simeq \left(%
\begin{array}{cc}
 \ast_{2,2N} & x_1+x_2-x_4-x_5\\
 \ast_{2,2N-2} & x_1x_2-x_4x_5\\
 \ast_{2,2N} & x_3-x_6
\end{array}%
\right)_{R_\partial}\{0,-1\}.
\]
So, by Lemma \ref{lemma-dual-Koszul} and Proposition \ref{prop-contraction-weak},
\begin{eqnarray*}
&& \Hom_{\hmf}(\fC_N(\Gamma_{101}), \fC_N(\Gamma_0)\{0,-1\}) \\
& \simeq & H^{0,0,0}\left(\left(%
\begin{array}{cc}
 \ast_{2,2N} & x_1+x_2-x_4-x_5\\
 \ast_{2,2N-2} & x_1x_2-x_4x_5\\
 \ast_{2,2N} & x_3-x_6 \\
 \ast_{2,2N} & -(x_1+x_2+x_3-x_4-x_5-x_6)\\
 \ast_{2,2N-2} & -(x_1x_2-x_4(x_5+x_6-x_3))\\
 \ast_{2,2N-2} & -((x_5+x_6-x_3)x_3-x_5x_6)
\end{array}%
\right)_{R_\partial} \left\langle 3\right\rangle \{3,3N-7\}\right) \\
& \simeq & H^{0,0,0}\left( \left(%
\begin{array}{cc}
 \ast_{2,2N} & 0\\
 \ast_{2,2N-2} & 0\\
 \ast_{2,2N-2} & 0
\end{array}%
\right)_{R'} \left\langle 3\right\rangle \{3,3N-7\}\right),
\end{eqnarray*}
where $R'=R_\partial/(x_1+x_2-x_4-x_5,x_1x_2-x_4x_5,x_3-x_6)$. As a $\zed_2 \oplus \zed^{\oplus 2}$-graded $R'$-module,
\begin{eqnarray*}
&& {\left(%
 \begin{array}{cc}
 \ast_{2,2N} & 0\\
 \ast_{2,2N-2} & 0\\
 \ast_{2,2N-2} & 0
\end{array}%
\right)_{R'} \left\langle 3\right\rangle \{3,3N-7\}} \\
& \cong &
(R' \left\langle 1\right\rangle \{1,N-1\} \oplus R') \otimes_{R'} (R' \left\langle 1\right\rangle \{1,N-3\} \oplus R') \otimes_{R'}(R' \left\langle 1\right\rangle \{1,N-3\} \oplus R'),
\end{eqnarray*}
whose homogeneous component of $\zed_2 \oplus \zed^{\oplus 2}$-degree $(0,0,0)$ is $1$-dimensional. This implies that
\[
\Hom_\hmf(\fC_N(\Gamma_{101}), \fC_N(\Gamma_0)\{0,-1\}) \cong H^{0,0,0}\left(\left(%
 \begin{array}{cc}
 \ast_{2,2N} & 0\\
 \ast_{2,2N-2} & 0\\
 \ast_{2,2N-2} & 0
\end{array}%
\right)_{R'} \left\langle 3\right\rangle \{3,3N-7\} \right)
\]
is at most $1$-dimensional. Thus, $\Hom_\hmf(\fC_N(\Gamma_{101}), \fC_N(\Gamma_0)\{0,-1\}) \cong \Q$.
\end{proof}

Proposition \ref{prop-RIIIinv} now follows easily.

\begin{proof}[Proof of Proposition \ref{prop-RIIIinv}]
Note that diagram \eqref{eq-RIII-T1-chain5} is invariant under horizontal reflection of the MOY graphs. So, using similar argument, we can show that $\fC_N(T_2)$ is also homotopic as a chain complex over $\hmf_{R_\partial,w}$ to a chain complex of the schematic form \eqref{eq-RIII-T1-chain5} such that arrows in this new chain complex also satisfy Lemma \ref{lemma-RIIIinv-arrow-comp}. By Lemma \ref{lemma-RIIIinv-arrow-hmf}, corresponding arrows in the two schematic forms \eqref{eq-RIII-T1-chain5} for $\fC_N(T_1)$ and $\fC_N(T_2)$ are scalar multiples of each other. Using Lemma \ref{lemma-RIIIinv-arrow-comp}, it is easy to verify that these two chain complexes of schematic form \eqref{eq-RIII-T1-chain5} are isomorphic to each other as chain complexes over $\hmf_{R_\partial,w}$. This proves that $\fC_N(T_1) \simeq \fC_N(T_2)$ as chain complexes over $\hmf_{R_\partial,w}$.
\end{proof}

\section{The Decategorification of $\fH_N$}\label{sec-decat}

In this section, we establish the skein description of the decategorification $\fP_N$ of $\fH_N$ in Theorem \ref{thm-decat} and demonstrate the failure of $\fP_N$ to detect transverse non-simplicity from certain flype moves. 

\subsection{A skein description for $\fP_N$}\label{subsec-skein} The goal of this subsection is to prove Theorem \ref{thm-decat}. We start by recalling the ``invariant computation tree" constructed in \cite{FW}. 

\begin{definition}\cite[Definitions 1.1 and 1.3]{FW}\label{def-invariant-computation-tree}
For an $m$-strand closed braid $B_\pm = \beta \sigma_i^{\pm} \gamma$, the Conway split of $B_\pm$ at the crossing $\sigma_i^\pm$ produces two $m$-strand braids $B_0 = \beta \gamma$ and $B_\mp =  \beta \sigma_i^{\mp} \gamma$.

An invariant computation tree is a connected rooted oriented binary tree with each node labeled by a closed braid satisfying:
\begin{enumerate}
	\item If a node is labeled by a braid $B$ and its two children are labeled by $B'$ and $B''$, then $B'$ and $B''$ are obtained from $B$ by first applying a sequence of transverse Markov moves\footnote{Transverse Markov moves are call ``invariant Markov moves" in \cite{FW}.} and then doing a Conway split.
	\item Every terminal node is labeled with a closed braid with no crossings.
\end{enumerate}
\end{definition}

\begin{theorem}\cite[Theorem 1.7]{FW}\label{thm-invariant-computation-tree}
For any closed braid $B$, there exists an invariant computation tree whose root is labeled by $B$.
\end{theorem}

Before proving Theorem \ref{thm-decat}, we establish the following lemma, which will be used to prove Part 3 of Theorem \ref{thm-decat}.

\begin{lemma}\label{lemma-hat-P-unlink}
Denote by $U^{\sqcup m}$ the $m$-strand closed braid with no crossings. Define 
\[
\mathscr{H}_N(U^{\sqcup m}) := H(H(\fC_N(U^{\sqcup m})/a\fC_N(U^{\sqcup m}),d_{mf}),d_\chi),
\] 
which inherits the $\zed_2\oplus \zed^{\oplus 3}$-grading of $\fC_N(U^{\sqcup m})$. Its graded Euler characteristic is 
\[
\mathscr{P}_N(U^{\sqcup m}) := \sum_{(\ve,i,j,k)\in \zed_2 \oplus \zed^{\oplus3}}  (-1)^i \tau^\ve \alpha^j \xi^k \dim_\Q\mathscr{H}_N^{\ve,i,j,k}(U^{\sqcup m}) \in \zed[[\alpha,\xi]][\alpha^{-1},\xi^{-1},\tau]/(\tau^2-1).
\]
Then
\[
\mathscr{P}_N(U^{\sqcup m}) = \left(\frac{1+\tau\alpha^{-1}\xi^{-N+1}}{1-\xi^2}\right)^{m}.
\]
\end{lemma}

\begin{proof}
Put one marked point on each strand of $U^{\sqcup m}$. Note $U^{\sqcup m}$ has no crossings. It is straightforward to check that
\[
\fC_N(U^{\sqcup m})/a\fC_N(U^{\sqcup m}) \cong 0 \rightarrow  \underbrace{\left(%
\begin{array}{cc}
  0_{2,2N} & 0 \\
  0_{2,2N} & 0 \\
  \vdots & \vdots \\
  0_{2,2N} & 0
\end{array}%
\right)_{\Q[x_1,\dots,x_m]}}_{0} \rightarrow 0,
\]
where the Koszul matrix factorization has $m$ rows and $0_{2,2N}$ is a ``homogeneous $0$ of $\zed^{\oplus 2}$-degree $(2,2N)$". Note that both $d_{mf}$ and $d_\chi$ in $\fC_N(U^{\sqcup m})/a\fC_N(U^{\sqcup m})$ are $0$. So
\[
\mathscr{P}_N(U^{\sqcup m}) = \gdim_\Q \left(%
\begin{array}{cc}
  0_{2,2N} & 0 \\
  0_{2,2N} & 0 \\
  \vdots & \vdots \\
  0_{2,2N} & 0
\end{array}%
\right)_{\Q[x_1,\dots,x_m]} = \left(\frac{1+\tau\alpha^{-1}\xi^{-N+1}}{1-\xi^2}\right)^{m}.
\]
\end{proof}

We are now ready to prove Theorem \ref{thm-decat}.

\begin{proof}[Proof of Theorem \ref{thm-decat}]
Part 1 is a direct consequence of Theorem \ref{thm-trans-link-homology}. 

Now we consider Part 2. From local chain complexes \eqref{eq-def-chain-crossing+} and \eqref{eq-def-chain-crossing-}, one can see that
\begin{eqnarray*}
\fP_N (\input{crossing+inline}) & = & \tau\alpha\xi^{N-1}\fP_N (\input{arcs-inline}) - \tau\alpha\xi^{N}\fP_N (\input{wide-edge-inline}), \\
\fP_N (\input{crossing-inline}) & = & \tau\alpha^{-1}\xi^{-N+1}\fP_N (\input{arcs-inline}) - \tau\alpha^{-1}\xi^{-N}\fP_N (\input{wide-edge-inline}).
\end{eqnarray*}
It follows easily from these equations that
\[
\alpha^{-1}\xi^{-N} \fP_N (\input{crossing+inline}) - \alpha\xi^{N} \fP_N (\input{crossing-inline}) = \tau (\xi^{-1}-\xi)\fP_N (\input{arcs-inline}).
\]
This proves Part 2.

Next, we prove Part 3. For $m \geq 2$, denote by $\sigma_{m-1}^{\pm 1}$ the closed $m$-braid with a single $\pm$ crossing between the $(m-1)$-th and the $m$-th strands. By Part 2, we have 
\begin{equation}\label{eq-unlink-1}
\fP_N(U^{\sqcup m}) = \frac{\tau}{\xi^{-1}-\xi} \cdot (\alpha^{-1}\xi^{-N} \fP_N (\sigma_{m-1}) - \alpha\xi^{N} \fP_N (\sigma_{m-1}^{-1})).
\end{equation}
From Part 1, we know that 
\begin{equation}\label{eq-unlink-2}
\fP_N (\sigma_{m-1}) = \fP_N(U^{\sqcup m-1}).
\end{equation}
By Theorem \ref{thm-neg-stabilization} and Lemma \ref{lemma-hat-P-unlink}, we have
\begin{equation}\label{eq-unlink-3}
\fP_N (\sigma_{m-1}^{-1}) = \alpha^{-2}(\fP_N(U^{\sqcup m-1})- \mathscr{P}_N(U^{\sqcup m-1})) = \alpha^{-2}(\fP_N(U^{\sqcup m-1})- \left(\frac{1+\tau\alpha^{-1}\xi^{-N+1}}{1-\xi^2}\right)^{m-1}).
\end{equation}
Plugging \eqref{eq-unlink-2} and \eqref{eq-unlink-3} into \eqref{eq-unlink-1}, we get that, for $m \geq 2$,
\begin{equation}\label{eq-unlink-4}
\fP_N(U^{\sqcup m}) = \tau\alpha^{-1}([N] \fP_N(U^{\sqcup m-1}) + \frac{\xi^N}{\xi^{-1}-\xi}\left(\frac{1+\tau\alpha^{-1}\xi^{-N+1}}{1-\xi^2}\right)^{m-1}).
\end{equation}
where $[N]:= \frac{\xi^{-N}-\xi^N}{\xi^{-1} -\xi}$. From Corollary \ref{cor-unknot-neg-stab}, we have
\begin{equation}\label{eq-unlink-5}
\fP_N(U^{\sqcup 1}) = \fP_N(U) = \tau\alpha^{-1}([N] \frac{1}{1-\alpha^2} + \frac{\xi^N}{\xi^{-1}-\xi}).
\end{equation}
Let $\mathsf{f}_m = \frac{\fP_N(U^{\sqcup m})}{(\tau\alpha^{-1}[N])^m}$. Then, by \eqref{eq-unlink-4} and \eqref{eq-unlink-5}, the sequence $\{\mathsf{f}_m\}$ satisfies the recursive relation 
\begin{equation}\label{eq-unlink-6}
\begin{cases}
\mathsf{f}_m = \mathsf{f}_{m-1} + \frac{\xi^N}{(\xi^{-N}-\xi^N)}\left(\frac{\tau\alpha\xi^{-1}+\xi^{-N}}{\xi^{-N}-\xi^N}\right)^{m-1} & \text{for } m \geq 2, \\
\mathsf{f}_1 =  \frac{1}{1-\alpha^2} + \frac{\xi^N}{\xi^{-N}-\xi^N}. &
\end{cases}
\end{equation}
Therefore, 
\begin{eqnarray*}
\mathsf{f}_m &  = & \frac{1}{1-\alpha^2} + \frac{\xi^N}{(\xi^{-N}-\xi^N)} \sum_{l=0}^{m-1} \left(\frac{\tau\alpha\xi^{-1}+\xi^{-N}}{\xi^{-N}-\xi^N}\right)^{l} \\
& = & \frac{1}{1-\alpha^2} +\frac{\xi^N}{(\xi^{-N}-\xi^N)} \cdot \frac{1-\left(\frac{\tau\alpha\xi^{-1}+\xi^{-N}}{\xi^{-N}-\xi^N}\right)^{m}}{1-\frac{\tau\alpha\xi^{-1}+\xi^{-N}}{\xi^{-N}-\xi^N}} \\
& = & \frac{1}{1-\alpha^2} + \frac{\left(\frac{\tau\alpha\xi^{-1}+\xi^{-N}}{\xi^{-N}-\xi^N}\right)^{m}-1}{\tau\alpha\xi^{-N-1} +1}.
\end{eqnarray*}
Thus, 
\begin{equation}\label{eq-unlink-7}
\fP_N(U^{\sqcup m})=(\tau\alpha^{-1}[N])^m\mathsf{f}_m = (\tau \alpha^{-1} [N])^m ( \frac{1}{1-\alpha^2} + \frac{\left(\frac{\tau \alpha \xi^{-1} + \xi^{-N}}{\xi^{-N}-\xi^N}\right)^m-1}{\tau\alpha\xi^{-N-1} +1}).
\end{equation}
This proves Part 3.

Finally, from Theorem \ref{thm-invariant-computation-tree}, it is easy to see that Parts 1--3 uniquely determine $\fP_N$. So Part 4 is true.
\end{proof}

\subsection{Flype moves}\label{subsec-flype} It is not yet clear whether $\fP_N$ is truly a new invariant for transverse links or just a weird normalization of the classical HOMFLYPT polynomial. In this subsection, we prove Corollary \ref{cor-flype}, which suggests that $\fP_N$ can not detect transverse non-simplicity from flype moves.

\begin{proof}[Proof of Corollary \ref{cor-flype}]
First we prove that $\fP_N(B_1) = \fP_N(B_2)$. If $r=0$, then $B_1$ and $B_2$ are transverse isotopic. So $\fP_N(B_1) = \fP_N(B_2)$. It remains to prove $\fP_N(B_1) = \fP_N(B_2)$ for $r\neq 0$. We only prove this for $r>0$. The proof for $r<0$ is very similar and left to the reader. All braids in this part of the proof are $3$-braids. Note that, by the skein relation in Theorem \ref{thm-decat},
\[
\fP_N(\sigma_1^{2p+1}\sigma_2^{2r}\sigma_1^{2q}\sigma_2^{-1}) = \tau \alpha \xi^N(\xi^{-1}-\xi) \fP_N (\sigma_1^{2p+1}\sigma_2^{2r-1}\sigma_1^{2q}\sigma_2^{-1}) + \alpha^2 \xi^{2N} \fP_N (\sigma_1^{2p+1}\sigma_2^{2r-2}\sigma_1^{2q}\sigma_2^{-1}).
\]
Applying this successively, we get two polynomials $f_{0,r},f_{-1,r} \in \zed[\alpha,\xi,\alpha^{-1},\xi^{-1},\tau]/(\tau^2-1)$ such that
\[
\fP_N(B_1)= \fP_N(\sigma_1^{2p+1}\sigma_2^{2r}\sigma_1^{2q}\sigma_2^{-1}) = f_{0,r} \cdot \fP_N(\sigma_1^{2p+1}\sigma_2^{0}\sigma_1^{2q}\sigma_2^{-1}) + f_{-1,r} \cdot \fP_N(\sigma_1^{2p+1}\sigma_2^{-1}\sigma_1^{2q}\sigma_2^{-1}).
\]
Similarly,
\[
\fP_N(B_2)= \fP_N(\sigma_1^{2p+1}\sigma_2^{-1}\sigma_1^{2q}\sigma_2^{2r}) = f_{0,r} \cdot \fP_N(\sigma_1^{2p+1}\sigma_2^{-1}\sigma_1^{2q}\sigma_2^{0}) + f_{-1,r} \cdot \fP_N(\sigma_1^{2p+1}\sigma_2^{-1}\sigma_1^{2q}\sigma_2^{-1}).
\]
But $\sigma_1^{2p+1}\sigma_2^{0}\sigma_1^{2q}\sigma_2^{-1}$ and $\sigma_1^{2p+1}\sigma_2^{-1}\sigma_1^{2q}\sigma_2^{0}$ are transverse isotopic. So 
\[
\fP_N(\sigma_1^{2p+1}\sigma_2^{0}\sigma_1^{2q}\sigma_2^{-1})=  \fP_N(\sigma_1^{2p+1}\sigma_2^{-1}\sigma_1^{2q}\sigma_2^{0})
\]
and, therefore, $\fP_N(B_1) = \fP_N(B_2)$.

Now we prove that $\fP_N(B_3) = \fP_N(B_4)$. In the rest of this proof, each closed braid is denoted by a pair $(b,m)$, where $b$ is the braid word and $m$ is the number of strands in the closed braid. Recall that
\begin{eqnarray*}
B_3 & = & (\sigma_1\sigma_2^{-1}\sigma_1\sigma_2^{-1}\sigma_3^3\sigma_2\sigma_3^{-1}, 4), \\
B_4 & = & (\sigma_1\sigma_2^{-1}\sigma_1\sigma_2^{-1}\sigma_3^{-1}\sigma_2\sigma_3^{3},4).
\end{eqnarray*}
Applying the skein relation in Theorem \ref{thm-decat} successively to the crossings in $\sigma_3^{3}$ in $B_3$ and $B_4$, we know that there are polynomials $g_{0,r},g_{-1,r} \in \zed[\alpha,\xi,\alpha^{-1},\xi^{-1},\tau]/(\tau^2-1)$ such that
\begin{eqnarray*}
\fP_N(B_3) & = & g_{0,r} \cdot \fP_N(\sigma_1\sigma_2^{-1}\sigma_1\sigma_2^{-1}\sigma_3^0\sigma_2\sigma_3^{-1}, 4) + g_{-1,r} \cdot \fP_N(\sigma_1\sigma_2^{-1}\sigma_1\sigma_2^{-1}\sigma_3^{-1}\sigma_2\sigma_3^{-1}, 4), \\
\fP_N(B_4) & = & g_{0,r} \cdot \fP_N(\sigma_1\sigma_2^{-1}\sigma_1\sigma_2^{-1}\sigma_3^{-1}\sigma_2\sigma_3^{0},4) + g_{-1,r} \cdot \fP_N(\sigma_1\sigma_2^{-1}\sigma_1\sigma_2^{-1}\sigma_3^{-1}\sigma_2\sigma_3^{-1},4).
\end{eqnarray*}
Note that, by Theorem \ref{thm-decat},
\begin{eqnarray*}
&& \fP_N(\sigma_1\sigma_2^{-1}\sigma_1\sigma_2^{-1}\sigma_3^0\sigma_2\sigma_3^{-1}, 4) = \fP_N(\sigma_1\sigma_2^{-1}\sigma_1\sigma_3^{-1}, 4) \\
& = & \alpha^{-2}\xi^{-2N} \fP_N(\sigma_1\sigma_2^{-1}\sigma_1\sigma_3, 4)   -\tau\alpha^{-1}\xi^{-N}(\xi^{-1}-\xi)\fP_N(\sigma_1\sigma_2^{-1}\sigma_1, 4) \\
& = & \alpha^{-2}\xi^{-2N} \fP_N(\sigma_1\sigma_2^{-1}\sigma_1, 3)   -\tau\alpha^{-1}\xi^{-N}(\xi^{-1}-\xi)\fP_N(\sigma_1\sigma_2^{-1}\sigma_1, 4) \\
\end{eqnarray*}
and 
\begin{eqnarray*}
&&\fP_N(\sigma_1\sigma_2^{-1}\sigma_1\sigma_2^{-1}\sigma_3^{-1}\sigma_2\sigma_3^{0},4) = \fP_N(\sigma_1\sigma_2^{-1}\sigma_1\sigma_2^{-1}\sigma_3^{-1}\sigma_2,4) \\
& = & \alpha^{-2}\xi^{-2N} \fP_N(\sigma_1\sigma_2^{-1}\sigma_1\sigma_2^{-1}\sigma_3\sigma_2,4) - \tau\alpha^{-1}\xi^{-N}(\xi^{-1}-\xi) \fP_N(\sigma_1\sigma_2^{-1}\sigma_1\sigma_2^{-1}\sigma_2,4) \\
& = & \alpha^{-2}\xi^{-2N} \fP_N(\sigma_1\sigma_2^{-1}\sigma_1,3) - \tau\alpha^{-1}\xi^{-N}(\xi^{-1}-\xi) \fP_N(\sigma_1\sigma_2^{-1}\sigma_1,4). \\
\end{eqnarray*}
Thus, $\fP_N(\sigma_1\sigma_2^{-1}\sigma_1\sigma_2^{-1}\sigma_3^0\sigma_2\sigma_3^{-1}, 4) = \fP_N(\sigma_1\sigma_2^{-1}\sigma_1\sigma_2^{-1}\sigma_3^{-1}\sigma_2\sigma_3^{0},4)$. So $\fP_N(B_3) = \fP_N(B_4)$.
\end{proof}

\section{Relation to the $\slmf(N)$ Khovanov-Rozansky homology}\label{sec-module-structure}

We prove Theorem \ref{thm-sl-N-rel} in this section. Let us start with two algebraic lemmas.

\begin{lemma}\label{lemma-a-1-inj}
Let $\Q[a]$ be the graded polynomial ring with grading given by $\deg_a a =2$. Suppose that $M$ is a $\zed$-graded $\Q[a]$-module whose grading is bounded below and that $f(a)=\sum_{j=0}^m c_j a^j \in \Q[a]$ satisfies $c_0 \neq 0$. Then the endomorphism $M\xrightarrow{f(a)}M$ is injective. 
\end{lemma}

\begin{proof}
Assume the multiplication by $f(a)$ is not injective on $M$. Then there is an element $u$ of $M$ such that $u\neq 0$ and $f(a)u=0$. Write $u=\sum_j u_j$, where $u_j$ is the homogeneous part of $u$ of degree $j$. Since the grading of $M$ is bounded below, there exists a $j_0$ such that $u_{j_0} \neq 0$ and $u_j=0$ if $j<j_0$. Note that $c_0 u_j +\sum_{i=1}^m c_i a^i u_{j-2i}$ is the homogeneous part of $f(a)u$ of degree $j$ and $\sum_j (c_0 u_j +\sum_{i=1}^m c_i a^i u_{j-2i}) = f(a)u=0$. So $c_0 u_j +\sum_{i=1}^m c_i a^i u_{j-2i}=0$ for all $j$. When $j=j_0$, this implies $u_{j_0}=0$, which is a contradiction.
\end{proof}

\begin{lemma}\label{lemma-fg-graded-module-structure}
Suppose that $M$ is a finitely generated $\zed$-graded $\Q[a]$-module. Then, as a $\zed$-graded $\Q[a]$-module, $M\cong (\bigoplus_{j=1}^m \Q[a]\{s_j\}) \bigoplus (\bigoplus_{k=1}^n \Q[a]/(a^{l_k})\{t_k\})$, where the sequences $\{s_1,\dots,s_m\} \subset \zed$ and $\{(l_1,t_1),\dots,(l_n,t_n)\} \subset \zed^{\oplus 2}$ are uniquely determined by $M$ up to permutation.
\end{lemma}

\begin{proof}
Let $G$ be a finite homogeneous generating set of $M$. Denote by $\pi$ the $\Q[a]$-module map $\pi:\bigoplus_{u\in G} \Q[a]\{\deg u\} \rightarrow M$ which maps the $1$ in  $\Q[a]\{\deg u\}$ to $u$. Since $\pi$ is homogeneous, $\ker \pi$ inherits the grading of $\bigoplus_{u\in G} \Q[a]\{\deg u\}$. Since $\Q[a]$ is a principal ideal domain, $\ker \pi$ is also a free $\Q[a]$-module. Note that the gradings on $\bigoplus_{u\in G} \Q[a]\{\deg u\}$ and $\ker \pi$ are both bounded below. So, by Lemma \ref{lemma-homogeneous-basis-exists}, both of these admit homogeneous bases over $\Q[a]$. From this, it is easy to verify that there is a decomposition of chain complexes of graded $\Q[a]$ modules 
\begin{eqnarray*}
&& 0 \rightarrow \underbrace{\ker \pi}_{-1} \hookrightarrow \underbrace{\bigoplus_{u\in G} \Q[a]\{\deg u\}}_{0} \rightarrow 0 \\
& \cong & (\bigoplus_{j=1}^m 0 \rightarrow \underbrace{\Q[a]\{s_j\}}_{0} \rightarrow 0) \bigoplus (\bigoplus_{k=1}^n 0 \rightarrow \underbrace{\Q[a]\{2l_k+t_k\}}_{-1} \xrightarrow{a^{l_k}} \underbrace{\Q[a]\{t_k\}}_{0} \rightarrow 0).
\end{eqnarray*}
(See, for example, \cite[Lemma 4.13]{Wu-torsion}.) This implies that
\[
M \cong(\bigoplus_{u\in G} \Q[a]\{\deg u\})/\ker \pi \cong (\bigoplus_{j=1}^m \Q[a]\{s_j\}) \bigoplus (\bigoplus_{k=1}^n \Q[a]/(a^{l_k})\{t_k\}).
\]

The uniqueness part of the lemma is a slight refinement of the usual uniqueness theorem of the standard decompositions of finitely generated modules over a principal ideal domain. Suppose that
\begin{eqnarray*}
M & \cong & (\bigoplus_{j=1}^m \Q[a]\{s_j\}) \bigoplus (\bigoplus_{k=1}^n \Q[a]/(a^{l_k})\{t_k\}) \\
& \cong & (\bigoplus_{j=1}^{m'} \Q[a]\{s_j'\}) \bigoplus (\bigoplus_{k=1}^{n'} \Q[a]/(a^{l_k'})\{t_k'\}).
\end{eqnarray*}
Define $M_a = \{ r\in M ~|~ a^k r=0 \text{ for some } k\geq 0\}$. Then $M_a$ is a submodule of $M$ and 
\[
M_a \cong \bigoplus_{k=1}^n \Q[a]/(a^{l_k})\{t_k\} \cong \bigoplus_{k=1}^{n'} \Q[a]/(a^{l_k'})\{t_k'\}.
\]
By \cite[Lemma 4.14]{Wu-torsion}, this means that the sequences $\{(l_1,t_1),\dots,(l_n,t_n)\}$ and $\{(l_1',t_1'),\dots,(l_{n'}',t_{n'}')\}$ are permutations of each other. Moreover, we have
\[
\mathcal{M} := M/M_a \cong \bigoplus_{j=1}^m \Q[a]\{s_j\} \cong \bigoplus_{j=1}^{m'} \Q[a]\{s_j'\}.
\]
Thus, as a $\zed$-graded $\Q$-space, the graded dimension of $\mathcal{M}/a\mathcal{M}$ is 
\[
\gdim_\Q \mathcal{M}/a\mathcal{M} = \sum_{j=1}^m \alpha^{s_j} = \sum_{j=1}^{m'} \alpha^{s_j'}.
\]
This implies that $\{s_1,\dots,s_m\}$ and $\{s_1',\dots,s_{m'}'\}$ are permutations of each other. 
\end{proof}

Now we are ready to prove Theorem \ref{thm-sl-N-rel}.

\begin{proof}[Proof of Theorem \ref{thm-sl-N-rel}]
We prove Part 1 first. Comparing our definition of $\fC_N(B)$ to that in \cite{KR1}, one sees that $(\fC_N(B)/(a-1)\fC(B), d_{mf}, d_\chi)$ is isomorphic to the chain complex $(C(B), d_{mf},d_\chi)$ constructed in \cite{KR1} to define the $\slmf(N)$ link homology $H_N(B)$. Moreover, this isomorphism preserves the $\zed_2$-, homological and $x$-gradings. Recall that $\fC_N(B)$ is a free $\Q[a]$-module. So we have a short exact sequence
\[
0 \rightarrow \fC_N(B) \xrightarrow{a-1} \fC_N(B) \xrightarrow{\pi_1} C(B) \rightarrow 0
\] 
preserving the $\zed_2$-, homological and $x$-gradings, where $\pi_1$ is the standard quotient map 
\[
\fC_N(B) \rightarrow \fC_N(B)/(a-1)\fC_N(B) \cong C(B).
\] 
This short exact sequence induces a long exact sequence
{\tiny
\[
\cdots \rightarrow H^{\ve,i,\star,\star}(\fC_N(B),d_{mf}) \xrightarrow{a-1} H^{\ve,i,\star,\star}(\fC_N(B),d_{mf}) \rightarrow H^{\ve,i,\star,\star}(C(B),d_{mf}) \rightarrow H^{\ve+1,i,\star,\star}(\fC_N(B),d_{mf})\{-1,-N-1\} \xrightarrow{a-1} \cdots
\]}\vspace{-1pc}

\noindent preserving the $x$-grading. By Lemma \ref{lemma-a-1-inj}, the multiplication by $a-1$ is an injective homomorphism. So this long exact sequence gives a short exact sequence
\[
0\rightarrow (H^{\ve,\star,\star,\star}(\fC_N(B),d_{mf}),d_\chi) \xrightarrow{a-1} (H^{\ve,\star,\star,\star}(\fC_N(B),d_{mf}),d_\chi) \rightarrow (H^{\ve,\star,\star,\star}(C(B),d_{mf}),d_\chi) \rightarrow 0
\]
preserving the homological and $x$-gradings, which, in turn, induces a long exact sequence
\[
\cdots \rightarrow \fH_N^{\ve,i,\star,k}(B) \xrightarrow{a-1} \fH_N^{\ve,i,\star,k}(B) \rightarrow H_N^{\ve,i,k}(B) \rightarrow \fH_N^{\ve,i+1,\star,k}(B) \xrightarrow{a-1} \cdots
\]
Again, by Lemma \ref{lemma-a-1-inj}, the multiplication by $a-1$ is an injective homomorphism. So we get a short exact sequence 
\[
0 \rightarrow \fH_N^{\ve,i,\star,k}(B) \xrightarrow{a-1} \fH_N^{\ve,i,\star,k}(B) \rightarrow H_N^{\ve,i,k}(B) \rightarrow 0,
\]
which implies that $H_N^{\ve,i,k}(B) \cong \fH_N^{\ve,i,\star,k}(B)/(a-1)\fH_N^{\ve,i,\star,k}(B)$. This proves Part 1 of the theorem.

Now we prove part 2. Recall that $\fC_N^{\ve,i,\star,k}(B)$ is finitely generated over $\Q[a]$ for each triple $(\ve,i,k) \in\zed_2\oplus \zed^{\oplus2}$. Since $\Q[a]$ is a Noetherian ring, this implies that $\fH_N^{\ve,i,\star,k}(B)$ is a finitely generated $\zed$-graded
$\Q[a]$-module. Thus, by Lemma \ref{lemma-fg-graded-module-structure}, there is a decomposition
\[
\fH_N^{\ve,i,\star,k}(B)\cong (\bigoplus_{p=1}^{m_{\ve,i,k}} \Q[a]\{s_p\}) \bigoplus (\bigoplus_{q=1}^{n_{\ve,i,k}} \Q[a]/(a^{l_q})\{t_q\}),
\]
which is unique up to permutation of direct sum components. The only thing remaining is that $m_{\ve,i,k} = \dim_\Q H_N^{\ve,i,k}(B)$, which follows from Part 1 and the simple fact that
\begin{eqnarray*}
\Q[a]/(a-1) & \cong & \Q, \\
(\Q[a]/(a^l)) / (a-1) & \cong & 0.
\end{eqnarray*}
This completes the proof.
\end{proof}

\end{document}

%% file: crossing+inline.tex
\setlength{\unitlength}{.5pt}
\begin{picture}(40,30)(-20,10)

\put(-20,0){\vector(1,1){40}}

\put(20,0){\line(-1,1){18}}

\put(-2,22){\vector(-1,1){18}}

\end{picture}

%% file: crossing-inline.tex
\setlength{\unitlength}{.5pt}
\begin{picture}(40,30)(-20,10)

\put(20,0){\vector(-1,1){40}}

\put(-20,0){\line(1,1){18}}

\put(2,22){\vector(1,1){18}}

\end{picture}

%% file: arcs-inline.tex
\setlength{\unitlength}{.5pt}
\begin{picture}(40,30)(-20,10)

\qbezier(-20,0)(0,20)(-20,40)

\qbezier(20,0)(0,20)(20,40)

\put(20,40){\vector(1,1){0}}

\put(-20,40){\vector(-1,1){0}}

\end{picture}

%% file: wide-edge-inline.tex
\setlength{\unitlength}{.5pt}
\begin{picture}(40,30)(-20,10)

\put(20,0){\vector(-2,1){20}}

\put(-20,0){\vector(2,1){20}}

\put(0,30){\vector(-2,1){20}}

\put(0,30){\vector(2,1){20}}

\put(0,10){\vector(0,1){20}}

\put(3,16){\tiny{$2$}}

\end{picture}